\documentclass{article}

\usepackage{microtype}
\usepackage{graphicx}
\usepackage{subfigure}
\usepackage{booktabs} 

\usepackage{hyperref}

 \usepackage[accepted]{icml2024}

\usepackage{amsmath}
\usepackage{amssymb}
\usepackage{mathtools}
\usepackage{amsthm}

\def\b{{\mathbf{b}}}

\def\v{{\mathbf{v}}}

\def\w{{\mathbf{w}}}

\def\y{{\mathbf{y}}}

\def\b{{\mathbf{b}}}
\def\X{{\mathbf{X}}}
\def\Y{{\mathbf{Y}}}
\def\C{{\mathbf{C}}}
\def\A{{\mathbf{A}}}

\def\M{{\mathbf{M}}}

\def\C{{\mathbf{C}}}
\def\V{{\mathbf{V}}}

\def\Z{{\mathbf{Z}}}

\def\P{{\mathbf{P}}}

\DeclareMathOperator*{\argmax}{arg\,max}

\newcommand{\mA}{\mathcal{A}}
\newcommand{\mC}{\mathcal{C}}

\newcommand{\mK}{\mathcal{K}}

\newcommand{\mbS}{\mathbb{S}}

\newcommand{\mL}{\mathcal{L}}

\newcommand{\mD}{\mathcal{D}}

\newcommand{\trace}{\textnormal{\textrm{Tr}}}
\newcommand{\rank}{\textnormal{\textrm{rank}}}
\newcommand{\mat}{\textnormal{\textrm{mat}}}
\newcommand{\vect}{\textnormal{\textrm{vec}}}
\newcommand{\diam}{\textnormal{\textrm{diam}}}

\newcommand{\diag}{\textnormal{\textrm{diag}}}

\newcommand{\reals}{\mathbb{R}}
\newcommand{\sign}{\textrm{sign}}

\newcommand{\range}{\textnormal{\textrm{range}}}
\newcommand{\nullspace}{\textnormal{\textrm{null}}}

\newcommand{\blambda}{\boldsymbol{\lambda}}
\newcommand{\bmu}{\boldsymbol{\mu}}

\usepackage{multirow, makecell}
\usepackage{comment} 

\usepackage[capitalize,noabbrev]{cleveref}

\theoremstyle{plain}
\newtheorem{theorem}{Theorem}[section]

\newtheorem{lemma}[theorem]{Lemma}

\theoremstyle{definition}

\newtheorem{assumption}[theorem]{Assumption}
\theoremstyle{remark}
\newtheorem{remark}[theorem]{Remark}

\usepackage[textsize=tiny]{todonotes}

\icmltitlerunning{Low-Rank Extragradient Methods for Scalable Semidefinite Optimization}

\begin{document}

\twocolumn[
\icmltitle{Low-Rank Extragradient Methods for Scalable Semidefinite Optimization}



\icmlsetsymbol{equal}{*}

\begin{icmlauthorlist}
\icmlauthor{Dan Garber}{yyy}
\icmlauthor{Atara Kaplan}{yyy}
\end{icmlauthorlist}

\icmlaffiliation{yyy}{Faculty of Data and Decision Sciences, Technion - Israel Institute of Technology, Haifa, Israel}

\icmlcorrespondingauthor{Atara Kaplan}{ataragold@campus.technion.ac.il}
\icmlcorrespondingauthor{Dan Garber}{dangar@technion.ac.il}

\icmlkeywords{Machine Learning, ICML}

\vskip 0.3in
]



\printAffiliationsAndNotice{This version corrects an error in the previous version, as well as in the short version published in \textit{Operations Research Letters} \cite{garber2025low}: while in those versions we reported $\mathcal{O}(1/T)$ rates for the \textbf{best iterate}, in this corrected version these rates hold only w.r.t. the \textbf{average iterate}. The $\mathcal{O}(1/\sqrt{T})$ rates for the last iterate remain unaffected. \\} 

\begin{abstract}
We consider several classes of highly important semidefinite optimization problems that involve both a convex objective function (smooth or nonsmooth) and additional linear or nonlinear smooth and convex constraints, which are ubiquitous in statistics, machine learning, combinatorial optimization, and other domains. We focus on high-dimensional and plausible settings in which the problem admits a low-rank solution which also satisfies a low-rank complementarity condition. We provide several theoretical results proving  that, under these circumstances, the well-known Extragradient method, when initialized in the proximity of an optimal primal-dual solution, converges to a solution of the constrained optimization problem with its standard convergence rates guarantees, using only low-rank singular value decompositions (SVD) to project onto the positive semidefinite cone, as opposed to computationally-prohibitive full-rank SVDs required in worst-case. Our approach is supported by numerical experiments conducted with a dataset of Max-Cut instances.
\end{abstract}

\section{Introduction}
\label{sec:intro}
The basic semidefinite optimization problem (SDP) under consideration in this work is the following:
\begin{align} \label{mainProblemSDP}
\min f(\X) \quad  \textrm{s.t.} \quad \X\succeq 0,~ \mA(\X)=\b, 
\end{align}
where $f:\mathbb{S}^n\rightarrow\reals$ is convex and $\beta$-smooth, $\b\in\reals^m$, $\mA:\mathbb{S}^n\rightarrow\reals^m$ is a linear map defined as $\mA(\X)=(\langle\X,\A_1\rangle,\dots,\langle\X,\A_m\rangle)^{\top}$ for $\A_1,\dots,\A_m\in\mathbb{S}^n$, and $\mathbb{S}^n$ is the space of $n\times n$ real symmetric matrices.

We shall also consider extensions of \eqref{mainProblemSDP} including:
\begin{enumerate}
\item $\min_{\X}\lbrace f(\X)+g(\X)~|~\mA(\X)=\b,\ \X\succeq0\rbrace$, where $g:\mathbb{S}^n\rightarrow\reals$ is convex but nonsmooth and can be written as a maximum over smooth functions. 
\item $\min_{\X}\lbrace f(\X)~|~\mA(\X)=\b,\ g(\X)\le0,\ \X\succeq0\rbrace$, where $g(\X)=(g_1(\X),\ldots,g_d(\X))^{\top}$ such that for all $i\in\lbrace 1,\ldots,d\rbrace$ $g_i:\mathbb{S}^n\rightarrow\reals$ is convex and $\beta_g$-smooth.
\end{enumerate}

Many well-studied and important SDPs are known to have low-rank solutions. Low-rank solutions are prevalent in applications such as Max-Cut \cite{maxCut1}, matrix-completion \cite{matrixCompletion2, matrixCompletion3, matrixCompletion1}, phase-retrieval \cite{phaseRetrevial1, phaseRetrevial2}, $Z_2$-synchronization \cite{Z2sync1}.
In particular, it is known that for linear SDPs ($f(\X) = \trace(\C\X)$, $\C\in\mbS^n$) the rank of the optimal solution $r$ satisfies  $0.5r(r+1)\le m$  \cite{Barvinok1995ProblemsOD}.

Problem \eqref{mainProblemSDP} can  be written as a convex-concave saddle-point problem in the following standard way:
\begin{align} \label{problem:generalSDPprimal}
 \min_{\X\succeq0} \max_{\y\in\reals^m}f(\X) + {\y}^{\top}(\b-\mathcal{A}(\X)).
\end{align}
This formulation enables solving Problem \eqref{mainProblemSDP} via standard primal-dual projected gradient methods, such as the well-known projected Extragradient method \cite{extragradientK,NemirovskiEG} which converges to a saddle-point with rate $\mathcal{O}(1/t)$. The main drawback of these methods is the worst-case computation cost of the Euclidean projection onto the positive semidefinite (PSD) cone and the storage  for high-rank matrices.  
Even when the optimal solution to Problem \eqref{mainProblemSDP} is of low-rank, in worst-case, it cannot be ruled-out that the algorithm might need to maintain high-rank matrices throughout the optimization process.
The Euclidean projection of a point $\X\in\mathbb{S}^n$ which admits an eigen-decomposition $\X=\sum_{i=1}^n\lambda_i\v_i\v_i^{\top}$ onto the PSD cone  $\mathbb{S}^n_+:=\lbrace\Y\in\mathbb{S}^n~|~\Y\succeq0\rbrace$ is given by
\vskip -7.5mm
\begin{align} \label{def:EuclideanProjection}
\Pi_{\mathbb{S}^n_+}[\X]=\sum_{i=1}^n\max\lbrace\lambda_i,0\rbrace\v_i\v_i^{\top}.
\end{align}
\vskip -4mm
In particular,  computing the projection requires in worst-case a full singular value decomposition (SVD), which amounts to $\mathcal{O}(n^3)$ runtime in practical implementations, and to store high-rank matrices in memory. This forms the computational bottleneck in applying most standard first-order methods to large-scale SDPs.


Nevertheless, when there exists a low-rank optimal solution, it seems plausible to expect that, at least for well-conditioned instances and at least in some proximity of the optimal solution, it would be possible for such optimization methods to maintain only low-rank iterates, making them far more efficient and scalable for these important and notoriously difficult optimization problems. This reasoning motivates the idea of enforcing low-rank iterates by  substituting all (full) Euclidean projections with truncated projections which we define as:
\vspace{-1mm}
\begin{align} \label{def:rankRprojection}
\widehat{\Pi}^r_{\mathbb{S}^n_+}[\X]\hspace{-0.8mm}:=\hspace{-0.8mm} \Pi_{\mathbb{S}^n_+}\hspace{-1mm}\left[\hspace{-0.3mm}\sum_{i=1}^r\lambda_i\v_i\v_i^{\top}\hspace{-0.5mm}\right]\hspace{-1mm} =\hspace{-1mm}\sum_{i=1}^r\max\lbrace\lambda_i,0\rbrace\v_i\v_i^{\top},
\end{align}
\vskip -3mm
where $r\in[n]$ is the rank-truncation parameter. That is, only the top $r$ components in the eigen-decomposition of $\X$ are used, resulting in a matrix of rank at most $r$. 

The rank-$r$ truncated projection merely requires the computation of  the leading $r$ components in the SVD of the matrix to project (a rank-$r$ SVD) for which the runtime scales as $\mathcal{O}(rn^2)$ (and even faster when the input matrix is sparse) using fast iterative SVD methods, and is thus  far more efficient. 
We build on the classical \textit{complementarity condition} of optimal solutions for semidefinite programming, see for instance \cite{strictCompSDP, wolkowicz2012handbook}, to establish when (full) Euclidean projections could  be replaced with their truncated counterparts.  

Throughout we assume the following two standard assumptions hold. These are sufficient for strong duality to hold and the minimum of \eqref{mainProblemSDP} to be obtained.
\begin{assumption}[Slater condition] \label{Ass:slater}
There exists $\bar{\X}\in\mathbb{S}^n_{++} = \{\X\in\mbS^n~|~\X\succ 0\}$
such that  $\mathcal{A}(\bar{\X})=\b$.
\end{assumption}
\begin{assumption}\label{Ass:primalBounded}
Problem \eqref{mainProblemSDP} is bounded from below.
\end{assumption}
It is well known that under Assumptions  \ref{Ass:slater}, \ref{Ass:primalBounded} the following optimality conditions holds.

\textbf{Optimality conditions:} Any primal-dual solution $(\X^*,\y^*)\in\mbS^n_+\times\reals^m$  to Problem  \eqref{mainProblemSDP} satisfies: $\mA(\X^*)=\b$,  $\nabla{}f(\X^*)-\mathcal{A}^{\top}(\y^*)\succeq0$, and $\trace\left({(\nabla{}f(\X^*)-\mathcal{A}^{\top}(\y^*))\X^*}\right)=0$. The latter equality is known as \textit{complementarity} and it implies there exists $r \geq r^* = \rank(\X^*)$ such that the following  \textbf{complementarity condition} holds:
\begin{align}\label{eq:compcond}
\rank(\nabla{}f(\X^*)-\mathcal{A}^{\top}(\y^*))=n-r.
\end{align}

We shall refer to $\delta := \lambda_{n-r}(\nabla{}f(\X^*)-\mathcal{A}^{\top}(\y^*))$ (note $\delta > 0$) as the measure of complementarity  associated with  $(\X^*,\y^*)$. In particular, if \eqref{eq:compcond} holds for $r=r^*$ we shall say \textbf{strict complementarity} holds.




In a nutshell, we shall see that  if an optimal solution satisfies the complementarity condition \eqref{eq:compcond} with some parameter $r$, this implies that in a neighborhood of the optimal solution (the size of the neighborhood scales with the  complementarity measure $\delta$), indeed full Euclidean projections are equivalent  to their rank-$r$ truncated counterparts. Thus, the lower the rank $r$ for which the complementarity condition holds, the faster standard first-order methods could be implemented. In the most optimistic scenario, we have strict complementarity, i.e., $r = \rank(\X^*)$, which is a standard assumption for SDPs \cite{strictCompSDP, wolkowicz2012handbook, SDPstrictComplementarity}. Strict complementarity implies that locally, only SVDs of $\rank(\X^*)$ (which is assumed to be low) are required  in order to project onto $\mbS^n_+$.

\subsection{Motivation for Strict Complementarity}
As mentioned above, strict complementarity is a well established  assumption for SDPs \cite{strictCompSDP, wolkowicz2012handbook}.
In the case of linear SDPs where $f(\X)=\trace(\C\X)$ for some $\C\in\mathbb{S}^n$, \cite{SDPstrictComplementarity} prove that for a surjective map $\mA$, if slater's condition holds then for almost all cost matrices $\C$ strict complementarity holds, assuming a primal solution exists. 
In particular, they show this holds for problems such as Max-Cut, Orthogonal-Cut, and Product-SDP (optimization over product of spheres). They also show that under standard statistical assumptions, with high probability, the SDP relaxations for well-studied problems such as matrix completion, $Z_2$-synchronization and stochastic block-models also satisfy strict complementarity.
\cite{SDPstrictComplementarity} also numerically demonstrated, as we also do in our experiments (see  \cref{sec:experiments}),  
that strict complementarity holds for real-world instances of Max-Cut.

\subsection{Main Contributions (informal)}

\textbf{Smooth nonlinear SDPs}: We prove that under a $(r,\delta)$-complementarity condition, when applying the Extragradient method to the saddle-point formulation \eqref{problem:generalSDPprimal}, if the method is initialized within a distance that scales with $\delta$ from the saddle-point, then all projections onto the PSD cone can be replaced with the $r$-truncated
projection $\widehat{\Pi}^r_{\mathbb{S}^n_+}[\cdot]$, that is only SVD computations of rank-$r$ are required, without changing the sequence of iterates produced by the method. This leads to a $\mathcal{O}(1/T)$ convergence rate w.r.t. both objective function and affine constraints for the average iterate, and to a  $\mathcal{O}(1/\sqrt{T})$ rate for the last iterate. No knowledge of $\delta$ is required for the algorithm and we give a simple  procedure to verify that the algorithm indeed converges correctly (i.e., as if using standard computationally-expensive projections). We also show that increasing the rank of SVD computations beyond the rank of the complementarity condition can dramatically increase the radius of the ball around the optimal solution in which  all projections are low-rank. 
See \cref{sec:SDP}.

\textbf{Extensions}: We extend the above result to the following more involved models of interest:  
\vspace{-10pt}
\begin{enumerate}
\item
\textit{Augmented Lagrangian} approach for solving Problem \eqref{mainProblemSDP},  see \cref{sec:Augmented}.
\item
\textit{Nonsmooth objective function}  which can be written as a maximum over smooth functions, see \cref{sec:nonsmoothReg}.
\item
\textit{Convex and smooth constraints}, in addition to affine constraints, see \cref{sec:nonlinearInequalities}.  
\end{enumerate}

\textbf{Numerical evidence:} we use a dataset of Max-Cut instances to demonstrate: \textbf{I.} the presence of strict complementarity in real-world instances, \textbf{II.} that with simple initializations and from relatively early stages of the run, the Extragradient method indeed only requires SVDs of rank that matches that of the (low-rank) optimal solution to converge with its standard guarantees, and \textbf{III.} that increasing the rank of SVD computations beyond that of the complementarity condition \eqref{eq:compcond},  dramatically improves the conditioning  w.r.t. the use of truncated projections. We  demonstrate this increase in rank of SVDs  results in only moderate increase in computation time. See \cref{sec:experiments}.


\subsection{Related work}
\label{sec:relatedWork}
Solving SDPs at scale has been widely studied in many papers. In \cite{DingSpectralBundle} the authors show that under a strict complementarity assumption for \textit{linear SDP} problems, the spectral bundle method obtains a sublinear convergence rate which can speed up to a linear convergence rate when close to an optimal dual solution, and can be implemented efficiently using matrix sketching techniques.
 \cite{Yurtsever2019ScalableSP} also used  sketching techniques to solve \textit{linear SDPs} while storing only low-rank matrices and using efficient computation. 
There also exist conditional gradient methods for solving SDPs which avoid expensive projections such at \cite{CG3, CG2, garberCG}. These methods obtain slower rates of $\mathcal{O}(1/\sqrt{T})$ which speed up to $\mathcal{O}(1/T)$ for strongly convex settings \cite{garberCG}, however they often require to store high-rank matrices in memory. 

Aside from  convex methods, considerable efforts have been invested in developing the nonconvex Burer-Monteiro approach \cite{BurerM03} for solving low-rank SDPs, which involves replacing the PSD variable $\X$ with an explicit factorization of the form $\X=\V\V^{\top}$, where $\V\in\reals^{n\times r}$ for some $r\ge\rank(\X^*)$. Problem \eqref{mainProblemSDP} can be written as the following nonconvex problem:
$\min_{\V\in\reals^{n\times r}}\lbrace f(\V\V^{\top})~|~\mA(\V\V^{\top})=\b\rbrace$.
Riemannian gradient and  Riemannian trust region methods \cite{boumal2016non, BM1,BM2} can provably converge to second order stationary points for smooth manifolds.
When all second-order stationary points are also global optimal solutions, a convergence to a globally optimal solution is assured. However, following \cite{waldspurger2020rank} the authors of \cite{SDPstrictComplementarity} established that for several examples of linear SDPs ($f(\X)=\trace(\C\X)$), such as the Max-Cut SDP, and any $r$ satisfying $r(r+1)/2+r\le m$, there exist positive measure sets of cost matrices $\C$ for which a unique rank-$1$ optimal solution exists that also satisfies strict complementarity, yet there exists a non-optimal second order stationary point. In these cases there is no guarantee that the Burer-Monteiro approach will converge to a global optimal solution. 
Consequently, to guarantee convergence to a global solution of a linear SDP, these methods require a factorization of rank at least $r=\mathcal{O}(\sqrt{m})$, leading to a minimum storage requirement of $\mathcal{O}(n\sqrt{m})$. 

In a recent line of work \cite{garberNuclearNormMatrices, garberStochasticLowRank, ourExtragradient, garber2023efficiency}, it has been shown that for certain  low-rank convex optimization problems which satisfy a generalized strict complementarity condition, in proximity of the optimal solution, projected gradient methods could be implemented using only low-rank SVDs. In particular, \cite{ourExtragradient} established this for optimization problems over the spectrahedron\footnote{set of PSD matrices with unit trace} involving nonsmooth objective functions expressible as saddle-point problems.
Our paper builds on this approach, however it considers more complex and practically important scenarios that also encompass affine and convex inequality constraints.

\section{Smooth Semidefinite Programming}
\label{sec:SDP}
In this section we prove our main result regarding Problem \eqref{mainProblemSDP}: if the complementarity condition \eqref{eq:compcond} holds with low rank $r$,  then when initialized in the proximity of a saddle-point of the equivalent formulation \eqref{problem:generalSDPprimal},
the Extragradient method (see Algorithm \ref{alg:EG}) could be used to solve Problem \eqref{mainProblemSDP} using only low-rank SVD-based truncated projections $\widehat{\Pi}^r_{\mathbb{S}^n_+}[\cdot]$ (as defined in \eqref{def:rankRprojection}) instead of computationally-expensive standard projections which require full-rank SVDs. Towards this let us denote the Lagrangian $\mL(\X,\y):=f(\X) + {\y}^{\top}(\b-\mathcal{A}(\X))$.

The following lemma is central to our analysis and establishes how the complementarity condition induces a ball around a saddle-point of \eqref{problem:generalSDPprimal} in which the Euclidean projection onto the PSD cone could be indeed replaced with its $r$-truncated version when using projected gradient steps (as applied for instance in Algorithm \ref{alg:EG}).
The proof is given in Appendix \ref{appx:proofLemma24}. Throughout the sequel we denote $\Vert\mA\Vert:=\max\lbrace\Vert\mA(\X)\Vert_2~|~\Vert\X\Vert_F\le1\rbrace$ and $\Vert(\cdot,\cdot)\Vert$ to be the Euclidean norm over the product space $\mbS^n\times\reals^{m}$.
\begin{lemma} \label{lemma:radius}
Let $(\X^*,\y^*)$ be a saddle-point of Problem \eqref{problem:generalSDPprimal} and denote the rank of the complementarity condition $\tilde{r}=n-\rank(\nabla_{\X}\mL^*)$, where $\nabla_{\X}\mL^*:=\nabla_{\X}\mL(\X^*,\y^*)$. For any $r\ge\tilde{r}$, $\eta\ge0$, and $\X\in\mbS^n,(\Z,\w)\in\mathbb{S}^n_+\times\reals^m$, if
\begin{align} \label{ineq:boundRadiusLemma}
& \max\lbrace\Vert\X-\X^*\Vert_F,\Vert(\Z,\w)-(\X^*,\y^*)\Vert\rbrace \nonumber
\\  & \le \frac{\eta\max\left\lbrace\lambda_{n-r}(\nabla_{\X}\mL^*),\sqrt{\tilde{r}}\lambda_{n-r+\tilde{r}-1}(\nabla_{\X}\mL^*)\right\rbrace}{1+\sqrt{2}\eta\max\lbrace\beta,\Vert\mA\Vert\rbrace},
\end{align}
then $\rank(\Pi_{\mathbb{S}^n_+}[\X-\eta\nabla_{\X}\mL(\Z,\w)])\le r$ (meaning $\Pi_{\mathbb{S}^n_+}[\X-\eta\nabla_{\X}\mL(\Z,\w)]) =\widehat{\Pi}^r_{\mathbb{S}^n_+}[\X-\eta\nabla_{\X}\mL(\Z,\w)]$). 
\end{lemma}

\begin{remark}
Note that as the rank of the truncated projections $r$ increases beyond the minimal value corresponding to the complementarity condition ($\tilde{r}$), the eigenvalues $\lambda_{n-r}(\nabla_{\X}\mL(\X^*,\y^*))$ and $\lambda_{n-r+\tilde{r}-1}(\nabla_{\X}\mL(\X^*,\y^*))$ also increase. In particular, when $r$ is  sufficiently greater than the complementarity parameter $\tilde{r}$, the second term inside the $\max$ in the right-hand side (RHS) of  \eqref{ineq:boundRadiusLemma} can be significantly larger than the first term, as it also scales with $\sqrt{\tilde{r}}$.
Consequently, increasing $r$ increases the RHS of \eqref{ineq:boundRadiusLemma} --- representing the radius of the ball around an optimal solution within which rank-$r$ truncated projections coincide with accurate projections.
Of course, increasing $r$ requires higher-rank SVDs which increases the computation time.
\end{remark}

The following standard lemma (e.g.,  \cite{SABACHLagrangian}, yet a proof is given in \cref{appx:proofLemma25}) establishes that an approximated saddle-point of Problem \eqref{problem:generalSDPprimal} implies an approximated solution to Problem \eqref{mainProblemSDP}.
\begin{lemma} \label{lemma:boundf_and_norm}
Let $\X^*\succeq0$ be an optimal solution to Problem \eqref{mainProblemSDP} and let $(\widehat{\X},\widehat{\y})\in\mathbb{S}^n_+\times\reals^m$. Let $\y^*\in\argmax_{\y\in\reals^m}\mL(\X^*,\y)$.
Assume that for all $\y\in\lbrace\y\in\reals^m\ \vert\ \Vert\y\Vert_2\le2\Vert\y^*\Vert_2\rbrace$ it holds that
$\mathcal{L}(\widehat{\X},\y)-\mathcal{L}(\X^*,\widehat{\y})\le\varepsilon$.
Then,
$f(\widehat{\X})-f(\X^*)\le\varepsilon$ and $\Vert\mA(\widehat{\X})-\b\Vert_2\le\varepsilon/\Vert\y^*\Vert_2$.
\end{lemma}

\subsection{Projected Extagradient Method for Problem \eqref{mainProblemSDP}}

The projected extragradient method is a well-known method for solving saddle-point problems \cite{extragradientK,NemirovskiEG}. The method is written in Algorithm \ref{alg:EG}. 
\vspace{-3mm}
\begin{algorithm}[H]
	\caption{Extragradient method for saddle-point problems (see also \cite{extragradientK,NemirovskiEG})}\label{alg:EG}
	\begin{algorithmic}
		\STATE \textbf{Input:} step-size $\eta\ge0$ 
		\STATE \textbf{Initialization:} $(\X_1,\y_1)\in\mathbb{S}^n_+\times\reals^{m}$
		\FOR{$t = 1,2,...$} 
            \STATE $\Z_{t+1}=\Pi_{\mbS^n_+}[\X_t-\eta\nabla_{\X}\mL(\X_t,\y_t)]$   
			\STATE $\w_{t+1}=\y_t+\eta\nabla_{\y}\mL(\X_t,\y_t)$
			\STATE $\X_{t+1}=\Pi_{\mbS^n_+}[\X_t-\eta\nabla_{\X}\mL(\Z_{t+1},\w_{t+1})]$   
			\STATE $\y_{t+1}=\y_t+\eta\nabla_{\y}\mL(\Z_{t+1},\w_{t+1})$
        \ENDFOR
	\end{algorithmic}
\end{algorithm}
\vspace{-3mm}
We now state our main result. We give a short outline of the proof. The complete proof is given in \cref{appx:proofThm26}.
\begin{theorem} \label{thm:SDP}
Fix an optimal solution $\X^*\succeq0$ to Problem \eqref{mainProblemSDP}. Let $\y^*\in\reals^m$ be a corresponding dual solution and suppose $(\X^*,\y^*)$ satisfies the complementarity condition \eqref{eq:compcond} with parameter $\tilde{r}:=n-\rank(\nabla_{\X}\mL^*)$, where $\nabla_{\X}\mL^*:=\nabla_{\X}\mL(\X^*,\y^*)$. 
Let $\lbrace(\X_t,\y_t)\rbrace_{t\ge1}$ and $\lbrace(\Z_t,\w_t)\rbrace_{t\ge2}$ be the sequences of iterates generated by Algorithm \ref{alg:EG} with a fixed step-size
$\eta \le \frac{1}{\sqrt{2(\beta^2+\Vert\mA\Vert^2)}}$. Fix  $r \geq \tilde{r}$ and assume the initialization $(\X_1,\y_1)$ satisfies  $\Vert(\X_1,\y_1)-(\X^*,\y^*)\Vert_F\le R_0(r)$, where
\begin{align*}
& R_0(r):= 
\frac{\eta\max\left\lbrace\lambda_{n-r}(\nabla_{\X}\mL^*),\sqrt{\tilde{r}}\lambda_{n-r+\tilde{r}-1}(\nabla_{\X}\mL^*)\right\rbrace}{2(1+\sqrt{2}\eta\max\lbrace\beta,\Vert\mA\Vert\rbrace)}.
\end{align*}
Then, for all $t\ge1$, the projections $\Pi_{\mbS^n_+}[\cdot]$ in Algorithm \ref{alg:EG} could be replaced with rank-$r$ truncated projections (as defined in \eqref{def:rankRprojection}) without changing the sequences $\lbrace(\X_t,\y_t)\rbrace_{t\ge1}$ and $\lbrace(\Z_t,\w_t)\rbrace_{t\ge2}$. Moreover, for all $T\ge0$ it holds that
\begin{align*}
& f\left(\frac{1}{T}\sum_{t=1}^T \Z_{t+1}\right)-f(\X^*)\le R^2/(2\eta T),
\\ & \left\Vert\mA\left(\frac{1}{T}\sum_{t=1}^T \Z_{t+1}\right)-\b\right\Vert_2 \le R^2/(2\Vert\y^*\Vert_2\eta T),
\end{align*}
where \\
$R^2:=\Vert\X_1-\X^*\Vert_F^2+\max\limits_{\y\in\lbrace\y\in\reals^m\ \vert\ \Vert\y\Vert_2\le2\Vert\y^*\Vert_2\rbrace}\Vert\y_1-\y\Vert_2^2$.

Furthermore, if we take a step-size $\eta=\frac{1}{2\sqrt{\beta^2+\Vert\mA\Vert^2}}$, then for all $T\ge0$ it holds that
\begin{align*}
& f(\Z_{T+1})-f(\X^*)\le R_1/(\eta\sqrt{T}),
\\ & \Vert\mA(\Z_{T+1})-\b\Vert_2\le R_1/(\Vert\y^*\Vert_2\eta\sqrt{T}),
\end{align*}
 $$R_1:=3\sqrt{2}\max\lbrace 2R_0(r),3\Vert\y^*\Vert_2\rbrace\Vert(\X_1,\y_1)-(\X^*,\y^*)\Vert.$$
\end{theorem}
If strict complementarity holds, i.e., $\tilde{r}=r^*:=\rank(\X^*)$, and we use rank-$r^*$ truncated projections, then taking a step-size $\eta = 1/\sqrt{2(\beta^2+\Vert\mA\Vert^2)}$ and initializing with some point within a radius of $R_0(r^*)=
\frac{\lambda_{n-r^*}(\nabla_{\X}\mL(\X^*,\y^*))}{2\sqrt{2}(\sqrt{\beta^2+\Vert\mA\Vert^2}+\max\lbrace\beta,\Vert\mA\Vert\rbrace)}$ from the saddle-point, then for all $T\ge0$  it holds that
\begin{align*}
& f\left({\frac{1}{T}\sum_{t=1}^T \Z_{t+1}}\right)-f(\X^*)\le (R^2\sqrt{\beta^2+\Vert\mA\Vert^2})/(\sqrt{2}T),
\\ & \Vert\mA\left({\frac{1}{T}\sum_{t=1}^T \Z_{t+1}}\right)-\b\Vert_2\le (R^2\sqrt{\beta^2+\Vert\mA\Vert^2})/(\sqrt{2}\Vert\y^*\Vert_2T).
\end{align*}

\begin{remark}
The convergence rate in \cref{thm:SDP} is dependent on  the norm $\Vert\y^*\Vert_2$. Under Assumptions \ref{Ass:slater} and \ref{Ass:primalBounded} there exists a bounded dual solution $\y^*$. We derive explicit bounds in \cref{appx:boundsOptimalDualSolutions}.
\end{remark}

\paragraph{Proof idea of \cref{thm:SDP}}
Our proof extends the ideas from  \cite{ourExtragradient}. As can be seen from \eqref{def:EuclideanProjection}, when projecting a matrix onto the PSD cone, if the $(r+1)$-th eigenvalue of the matrix to project is non-positive, then the projection will be of rank at most $r$. The standard first-order optimality condition implies that for any step-size $\eta$, the matrix $\X^*-\eta\nabla_{\X}\mL(\X^*,\y^*)$  indeed satisfies  $\lambda_{r+1}(\X^*-\eta\nabla_{\X}\mL(\X^*,\y^*))<0$. Leveraging the complementarity condition, the smoothness of $\mL$,  and using perturbation bounds, it can be established that any points $(\X,\y),(\Z,\w)$ close enough to $(\X^*,\y^*)$ also satisfy that $\lambda_{r+1}(\X-\eta\nabla_{\X}\mL(\Z,\w))<0$. Thus, a projected extragradient step from these points will also return a matrix of rank at most $r$, and so it is sufficient to use rank-$r$ truncated projections. 
Therefore, if we initialize the algorithm sufficiently close to $(\X^*,\y^*)$, and employ  nonexpansive-type properties of the iterates of the extragradient method, which guarantee the iterates remain in proximity of  $(\X^*,\y^*)$, we can conclude that the  method maintains its standard convergence rate of $\mathcal{O}(1/T)$ while requiring only rank-$r$ truncated projections. The $\mathcal{O}(1/\sqrt{T})$ last iterate convergence rate has been established recently in \cite{lastIterateDecrease}.

\subsection{Certificates for Correctness of Low-Rank Truncated Projections}
From a practical point of view, the conditions  in \cref{thm:SDP} may not always be satisfied or easily verifiable. Nevertheless, there is a simple and efficient procedure for establishing whether the Euclidean projections onto the PSD cone of $\X_t-\eta\nabla_{\X}\mL(\X_t,\y_t)$ and $\X_t-\eta\nabla_{\X}\mL(\Z_{t+1},\w_{t+1})$ required in \cref{alg:EG} will return the same matrices as their rank-$r$ truncated projections (as defined in \eqref{def:rankRprojection}). 
By the definition of the Euclidean projection of some  $\X\in\mathbb{S}^n$ onto the PSD cone (see \eqref{def:EuclideanProjection}), the projection $\Pi_{\mathbb{S}^n_+}[\X]$ returns a matrix of rank at most $r$, in which case $\Pi_{\mathbb{S}^n_+}[\X]=\widehat{\Pi}^r_{\mathbb{S}^n_+}[\X]$, if and only if $\lambda_{r+1}(\X)\le0$. 
Therefore, when running \cref{alg:EG}, to ensure the rank-$r$ truncated projections are the correct projections, we merely need to check in each iteration whether
\begin{align} \label{cond:lowRankProjCondition}
& \lambda_{r+1}(\X_t-\eta\nabla_{\X}\mL(\X_t,\y_t))\le0\quad \textrm{and} \nonumber
\\ & \lambda_{r+1}(\X_t-\eta\nabla_{\X}\mL(\Z_{t+1},\w_{t+1}))\le0.
\end{align} 
This condition can easily be checked by computing the rank-$(r+1)$ SVDs instead of only the rank-$r$ SVD necessary for the rank-$r$ truncated projection itself. A similar condition was derived in \cite{ourExtragradient} for their spectrahedron-constrained setting.


\section{Augmented Lagrangian Method}
\label{sec:Augmented}
In practice, it has often been observed that solving the Augmented Lagrangian associated with Problem \eqref{mainProblemSDP} works better  \cite{AugmentedLagrangian1, AugmentedLagrangian2, AugmentedLagrangian3}. Problem \eqref{mainProblemSDP} can be written as an augmented Lagrangian for some parameter $\mu\ge0$ as follows:
\begin{align} \label{problem:augmented}
\min_{\X\succeq0} \max_{\y\in\reals^m}\hspace{-0.3mm} f(\X)\hspace{-0.7mm} + \hspace{-0.7mm} {\y}^{\top}(\b-\mathcal{A}(\X))\hspace{-0.7mm}+\hspace{-0.7mm}\frac{\mu}{2}\Vert\mA(\X)-\b\Vert_2^2.
\end{align}

A similar result to \cref{thm:SDP}, which follows as an immediate application of  \cref{thm:SDP}, can be obtained for this formulation as well. The formal statement of the theorem and its proof are given in \cref{appx:thmAugmented}.

\section{Nonsmooth Objective Functions}
\label{sec:nonsmoothReg}

An important extension to consider is the addition of  a nonsmooth function to the objective, i.e., the problem:
\begin{align} \label{problem:nonsmoothRegularizer}
\min f(\X)+g(\X) \quad \textrm{s.t.} \quad \X\succeq 0, ~\mathcal{A}(\X)=\b,
\end{align}
where $g:\mathbb{S}^n\rightarrow\reals$ is convex but nonsmooth and can be written as the maximum of smooth functions over a compact and convex subset $\mK\subset\mathbb{Y}$ of some finite linear space over the reals $\mathbb{Y}$. That is, $g$ can be written as $g(\X)=\max_{\blambda\in\mK}G(\X,\blambda)$, where $G(\cdot,\blambda)$ is convex for all $\blambda\in\mK$ and $G(\X,\cdot)$ is concave for all $\X\succeq0$. We assume it is efficient to compute Euclidean projections onto $\mK$.

Examples include problems such as sparsity-promoting $\ell_1$-regularized least squares or robust $\ell_1$-regression. Indeed the $\ell_1$-norm can be written as $\Vert\X\Vert_1=\max_{\Vert\Y\Vert_{\infty}\le1}\langle\X,\Y\rangle$ which fits our structure.

Problem \eqref{problem:nonsmoothRegularizer} can be written as the saddle-point problem:
\begin{align} \label{problem:nonsmoothRegularizersSP}
\min_{\X\succeq0} \max_{\substack{\blambda\in\mK \\ \y\in\reals^m}}f(\X) +G(\X,\blambda)+\y^{\top}(\b-\mathcal{A}(\X)).
\end{align}
We assume $G(\cdot,\cdot)$ is smooth with respect to all components. That is, we assume there exist constants $\rho_{X}$,$\rho_{\lambda}$,$\rho_{X\lambda}$,$\rho_{\lambda X}\ge0$ such that for all $\X,\widetilde{\X}\succeq0$ and all $\blambda,\widetilde{\blambda}\in\mK$ it holds that,
\begin{small}
\begin{align} \label{ineq:betasNonsmoothRegGpart}
& \Vert\nabla_{\X}G(\X,\blambda)-\nabla_{\X}G(\widetilde{\X},\blambda)\Vert_{F} \le\rho_{X}\Vert\X-\widetilde{\X}\Vert_{F}, \nonumber \\
&\Vert\nabla_{\blambda}G(\X,\blambda)-\nabla_{\blambda}G(\X,\widetilde{\blambda})\Vert_{2} \le\rho_{\lambda}\Vert\blambda-\widetilde{\blambda}\Vert_{2}, \nonumber \\
& \Vert\nabla_{\X}G(\X,\blambda)-\nabla_{\X}G(\X,\widetilde{\blambda})\Vert_{F} \le\rho_{X\lambda}\Vert\blambda-\widetilde{\blambda}\Vert_{2}, \nonumber \\
& \Vert\nabla_{\blambda}G(\X,\blambda)-\nabla_{\blambda}G(\widetilde{\X},\blambda)\Vert_{2} \le\rho_{\lambda X}\Vert\X-\widetilde{\X}\Vert_{F}.
\end{align}
\end{small}
We denote $\mL(\X,\y,\blambda):=f(\X) +G(\X,\blambda) + {\y}^{\top}(\b-\mathcal{A}(\X))$. The complementarity condition for this case is defined below and is consistent with \eqref{eq:compcond}, as both are related to the rank of $\nabla_{\X}\mL(\X^*,\y^*,\blambda^*)$.

\textbf{Optimality conditions:} Any primal-dual solution $(\X^*,(\y^*,\blambda^*))\in\mbS^n_+\times(\reals^m\times\mK)$  to Problem  \eqref{problem:nonsmoothRegularizersSP} satisfies:  $\mA(\X^*)=\b$,  $\nabla{}f(\X^*)-\mathcal{A}^{\top}(\y^*)+\nabla_{\X}G(\X^*,\blambda^*)\succeq0$, and $\trace\left({\left({\nabla{}f(\X^*)-\mathcal{A}^{\top}(\y^*)+\nabla_{\X}G(\X^*,\blambda^*)\succeq0}\right)\X^*}\right)=0$. The latter \textit{complementarity} equality implies there exists $r \geq r^* = \rank(\X^*)$ such that the following \textbf{complementarity condition} holds:
\begin{align}\label{eq:compcond:NS}
\hspace{-5pt}\rank(\nabla{}f(\X^*)\hspace{-1pt}-\hspace{-1pt}\mathcal{A}^{\top}(\y^*)\hspace{-1pt}+\hspace{-1pt}\nabla_{\X}G(\X^*,\blambda^*))\hspace{-1pt}=\hspace{-1pt}n-r.
\end{align}

Problem \eqref{problem:nonsmoothRegularizersSP} requires gradient ascent type updates for both dual variables $\y$ and $\blambda$. For clarity we present the projected extragradient method for solving  Problem \eqref{problem:nonsmoothRegularizersSP} in \cref{alg:EGregularized}, which is the same as \cref{alg:EG} with the additional updates for the $\blambda$ variable.   
\vspace{-3mm}
\begin{algorithm}[H]
	\caption{Extragradient method for Problem \eqref{problem:nonsmoothRegularizersSP}}\label{alg:EGregularized}
	\begin{algorithmic}
		\STATE \textbf{Input:} step-size $\eta\ge0$ 
		\STATE \textbf{Initialization:} $(\X_1,\y_1,\blambda_1)\in\mathbb{S}^n_+\times\reals^{m}\times\mK$
		\FOR{$t = 1,2,...$} 
            \STATE $\Z_{t+1}=\Pi_{\mbS^n_+}[\X_t-\eta\nabla_{\X}\mL(\X_t,\y_t,\blambda_t)]$   
			\STATE $\w_{t+1}=\y_t+\eta\nabla_{\y}\mL(\X_t,\y_t,\blambda_t)$
			\STATE $\bmu_{t+1} = \Pi_{\mK}[\bmu_t+\eta\nabla_{\blambda}\mL(\X_t,\y_t,\blambda_t)]$
			\STATE $\X_{t+1}=\Pi_{\mbS^n_+}[\X_t-\eta\nabla_{\X}\mL(\Z_{t+1},\w_{t+1},\bmu_{t+1})]$   
			\STATE $\y_{t+1}=\y_t+\eta\nabla_{\y}\mL(\Z_{t+1},\w_{t+1},\bmu_{t+1})$
			\STATE $\blambda_{t+1}=\Pi_{\mK}[\blambda_t+\eta\nabla_{\blambda}\mL(\Z_{t+1},\w_{t+1},\bmu_{t+1})]$
        \ENDFOR
	\end{algorithmic}
\end{algorithm}
\vskip -4mm
We now state our main result for this setting. The proof is given in \cref{appx:proofThm45}.
\begin{theorem} \label{thm:nonsmoothRegularizer}
Fix an optimal solution $\X^*\succeq0$ to Problem \eqref{problem:nonsmoothRegularizer}. Let $(\y^*,\blambda^*)\in\reals^m\times\mK$ be a corresponding dual solution and suppose $(\X^*,(\y^*,\blambda^*))$ satisfies the complementarity condition \eqref{eq:compcond:NS} with parameter  $\tilde{r}:=n-\rank(\nabla_{\X}\mL^*)$, where $\nabla_{\X}\mL^*:=\nabla_{\X}\mL(\X^*,\y^*,\blambda^*)$. 
Let $\lbrace(\X_t,(\y_t,\blambda_t))\rbrace_{t\ge1}$ and $\lbrace(\Z_t,(\w_t,\bmu_t))\rbrace_{t\ge2}$ be the sequences of iterates generated by  \cref{alg:EGregularized} with a fixed step-size\footnote{For the $\mathcal{O}(1/T)$ result we may take a slightly larger step-size. See restatement of the theorem in \cref{appx:proofThm45}.}
\vspace{-2mm}
\begin{align*}
\eta \le \min\left\lbrace\hspace{-1mm}\begin{array}{l}\frac{1}{2\sqrt{(\beta+\rho_{X})^2+\Vert\mA\Vert^2+\rho_{\lambda X}^2}},
\\ \frac{1}{\sqrt{2}\sqrt{\rho_{\lambda}^2+2\max\lbrace \Vert\mA\Vert^2,\rho_{X\lambda}^2\rbrace}},
\\ \frac{1}{\beta+\rho_{X}+\sqrt{2}\max\lbrace \Vert\mA\Vert,\rho_{X\lambda}\rbrace},\frac{1}{\rho_{\lambda}+\sqrt{\Vert\mA\Vert^2+\rho_{\lambda X}}}\end{array}\hspace{-1mm}\right\rbrace.
\end{align*}
Fix some $r\ge\tilde{r}$ and assume the initialization $(\X_1,(\y_1,\blambda_1))$ satisfies that $\Vert(\X_1,\y_1,\blambda_1)-(\X^*,\y^*,\blambda^*)\Vert_F\le R_0(r)$ where
\vspace{-8pt}
\begin{align*}
& R_0(r):= 
\frac{\eta\max\left\lbrace\lambda_{n-r}(\nabla_{\X}\mL^*),\sqrt{\tilde{r}}\lambda_{n-r+\tilde{r}-1}(\nabla_{\X}\mL^*)\right\rbrace}{2(1+2\eta\max\lbrace\beta+\rho_{X},\Vert\mA\Vert,\rho_{X\lambda}\rbrace)}.
\end{align*}
Then, for all $t\ge1$ the projections $\Pi_{\mbS^n_+}[\cdot]$ in \cref{alg:EGregularized} could be replaced with rank-r truncated projections without changing the sequences $\lbrace(\X_t,(\y_t,\blambda_t))\rbrace_{t\ge1}$ and $\lbrace(\Z_t,(\w_t,\bmu_t))\rbrace_{t\ge2}$. Moreover, for all $T\ge0$ it holds that
\begin{align*}
& f(\bar{\X})+g(\bar{\X})-f(\X^*)-g(\X^*) 
\le R^2/(2\eta T),
\\ & \Vert\mA(\bar{\X})-\b\Vert_2  \le R^2/(2\Vert\y^*\Vert_2\eta T),
\end{align*}
where $\bar{\X}=\frac{1}{T}\sum_{t=1}^T \Z_{t+1}$ and $R^2=\Vert\X_1-\X^*\Vert_F^2+\max\limits_{\y\in\lbrace\y\in\reals^m\ \vert\ \Vert\y\Vert_2\le2\Vert\y^*\Vert_2\rbrace}\Vert\y_1-\y\Vert_2^2+\max\limits_{\blambda\in\mK}\Vert\blambda_1-\blambda\Vert_2^2$.

Furthermore, 
for all $T\ge0$ it holds that
\begin{align*}
& f(\Z_{T+1})+g(\Z_{T+1})-f(\X^*)-g(\X^*)\le R_1/(\eta\sqrt{T}),
\\ & \Vert\mA(\Z_{T+1})-\b\Vert_2\le R_1/(\Vert\y^*\Vert_2\eta\sqrt{T}),
\end{align*}
where $R_1:=3\sqrt{2} \max\lbrace 2R_0(r),3\sqrt{2}\Vert\y^*\Vert_2,\sqrt{2}\diam(\mK)\rbrace
\cdot\Vert(\X_1,\y_1,\blambda_1)-(\X^*,\y^*,\blambda^*)\Vert$.
\end{theorem}


\section{Optimization with Linear Equality and Convex Inequality Constraints}
\label{sec:nonlinearInequalities}

We now consider the case of smooth and convex inequality constraints, in addition to  linear equalities. This problem can be written in the following saddle-point formulation:
\begin{align} \label{problem:ineqConstrained}
& \min f(\X)\quad \textrm{s.t.} \quad \X\succeq 0, ~\mathcal{A}(\X) = \b, ~ g(\X) \leq 0 \nonumber
\\ & = \min_{\X\succeq0} \max_{\substack{\boldsymbol{\lambda}\in\reals_+^d \\ \boldsymbol{\y}\in\reals^m}}f(\X) + \boldsymbol{\lambda}^{\top}g(\X)+\y^{\top}(\b-\mathcal{A}(\X)),
\end{align}
where $f:\mathbb{S}^n\rightarrow\reals$ is convex and $\beta$-smooth, $g=(g_1,\ldots,g_d)^{\top}$ such that for all $i\in\lbrace 1,\ldots,d\rbrace$, $g_i:\mathbb{S}^n\rightarrow\reals$ is convex and $\beta_g$-smooth.

For this problem Slater's condition is slightly different than \cref{Ass:slater} and can be written as:
\begin{assumption}[Slater condition] \label{Ass:slater_inequality}
There exists $\bar{\X}\in\mathbb{S}^n_{++}$\footnote{If there are only inequality constraints it is sufficient for the Slater condition to hold for some $\bar{\X}\in\mathbb{S}^n_{+}$.}
such that $g_i(\bar{\X})<0$ for all $i\in\lbrace1,\ldots, d\rbrace$, and $\mathcal{A}(\bar{\X})=\b$.
\end{assumption}

Denote $\mL(\X,\y,\blambda):=f(\X) + \blambda^{\top}g(\X)+\y^{\top}(\b-\mA(\X))$. We now define complementarity for  Problem \eqref{problem:ineqConstrained}. Note  this definition is also consistent with \cref{eq:compcond} as it also involves the rank of $\nabla_{\X}\mL(\X^*,\y^*,\blambda^*)$.

\textbf{Optimality conditions:} Any primal-dual solution $(\X^*,(\y^*,\blambda^*))\in\mbS^n_+\times(\reals^m\times\reals^d_+)$  to Problem  \eqref{problem:ineqConstrained} satisfies:  $\mA(\X^*)=\b$, $g(\X^*)\leq 0$,  $\nabla{}f(\X^*)-\mathcal{A}^{\top}(\y^*)+\sum_{i=1}^{d}\boldsymbol{\lambda}_i^*\nabla{}g_i(\X^*)\succeq0$, $\blambda^*_ig_i(\X^*)=0$ for all $i\in[d]$, and $\trace\left({\left({\nabla{}f(\X^*)-\mathcal{A}^{\top}(\y^*)+\sum_{i=1}^{d}\boldsymbol{\lambda}_i^*\nabla{}g_i(\X^*)}\right)\X^*}\right)=0$. The latter \textit{complementarity} equality implies there exists $r \geq r^* = \rank(\X^*)$ such that the  following \textbf{complementarity condition} holds:
\vspace{-6pt}
\begin{align}\label{eq:compcond:const}
\hspace{-6pt}\rank(\nabla{}f(\X^*)\hspace{-1pt}-\hspace{-1pt}\mathcal{A}^{\top}(\y^*)\hspace{-1pt}+\hspace{-1pt}\sum_{i=1}^{d}\boldsymbol{\lambda}_i^*\nabla{}g_i(\X^*))\hspace{-1pt}=\hspace{-1pt}n\hspace{-1pt}-\hspace{-1pt}r.
\end{align}

The projected extragradient method for Problem \eqref{problem:ineqConstrained} is as in \cref{alg:EGregularized}, but when replacing the set $\mK$ with $\reals^d_+$. 
We now present our main result for this setting. The proof is given in \cref{appx:proofThm56}. 
\begin{theorem}\label{thm:inequlalityConstraints}
Fix an optimal solution $\X^*\succeq0$ to Problem \eqref{problem:ineqConstrained}. Let $(\y^*,\blambda^*)\in\reals^m\times\reals^d_+$ be a corresponding dual solution and suppose $(\X^*,(\y^*,\blambda^*))$ satisfies the complementarity condition \eqref{eq:compcond:const} with rank $\tilde{r}:=n-\rank(\nabla_{\X}\mL^*)$, where $\nabla_{\X}\mL^*:=\nabla_{\X}\mL(\X^*,\y^*,\blambda^*)$.  
Fix some $r\ge\tilde{r}$ and denote
$\delta(r):=
 \max\left\lbrace\lambda_{n-r}(\nabla_{\X}\mL^*),\sqrt{\tilde{r}}\lambda_{n-r+\tilde{r}-1}(\nabla_{\X}\mL^*)\right\rbrace$. Denote $\mD_X=\lbrace\X\succeq0~\vert~\Vert\X-\X^*\Vert_F\le2\delta(r)\rbrace$ and $\mD_{\lambda}=\lbrace\blambda\in\reals^d_+~\vert~\Vert\blambda-\blambda^*\Vert_2\le2\delta(r)\rbrace$.
Let $\lbrace(\X_t,(\y_t,\blambda_t))\rbrace_{t\ge1}$ and $\lbrace(\Z_t,(\w_t,\bmu_t))\rbrace_{t\ge2}$ be the sequences of iterates generated by \cref{alg:EGregularized} (with $\mK=\reals^d_+$) with a fixed step-size\footnote{For the sake of the $\mathcal{O}(1/T)$ rate result we may take a slightly larger step-size. See restatement of the theorem in \cref{appx:proofThm56}.}
\begin{align*}
& \eta \hspace{-0.2mm}\le \hspace{-0.2mm}
\min\hspace{-0.6mm}\left\lbrace\hspace{-2.5mm}\begin{array}{l}\frac{1}{2\sqrt{(\beta+\max_{\blambda\in\mD_{\lambda}}\Vert\blambda\Vert_1\beta_g)^2+\Vert\mA\Vert^2+\beta_g^2}},
\\ \frac{1}{2\max\lbrace \Vert\mA\Vert,\sqrt{d}\max_{\X\in\mD_X}\max_{i\in[d]}\Vert\nabla{}g_i(\X)\Vert_F\rbrace},
\\ \frac{1}{\beta+\max\limits_{\blambda\in\mD_{\lambda}}\hspace{-1mm}\Vert\blambda\Vert_1\beta_g+\sqrt{2}\max\lbrace \Vert\mA\Vert,\max\limits_{\substack{\X\in\mD_X \\ i\in[d]}}\hspace{-1mm}\Vert\nabla{}g_i(\X)\Vert_F\rbrace}
\end{array}\hspace{-3mm}\right\rbrace.
\end{align*}
Assume the initialization $(\X_1,(\y_1,\blambda_1))$ satisfies  $\Vert(\X_1,\y_1,\blambda_1)-(\X^*,\y^*,\blambda^*)\Vert_F\le R_0(r)$ where
{\small
\begin{align*}
\hspace{-2pt} R_0(r) \hspace{-1pt}=\hspace{-1pt} \delta(r)\min\hspace{-2pt}\left\lbrace\vcenter{\hspace{-2pt}\hbox{$\displaystyle
1,\begin{small}
\frac{(1/2)\eta}{1\hspace{-0.5mm}+\hspace{-0.5mm}2\eta\max\left\lbrace
\hspace{-2mm}\begin{array}{l}\beta\hspace{-0.5mm}+\hspace{-0.5mm}\max\limits_{\blambda\in\mD_{\lambda}}\hspace{-0.5mm}\Vert\blambda\Vert_1\beta_g,\Vert\mA\Vert,\\ \sqrt{d}\max\limits_{\substack{\X\in\mathcal{D}_X \\ i\in[d]}}\hspace{-0.5mm}\Vert\nabla{}g_i(\X)\Vert_F\end{array}
\hspace{-2mm}\right\rbrace}
\end{small}
$}}\hspace{-29mm}\right\rbrace.
\end{align*}}

Then, for all $t\ge1$, the projections $\Pi_{\mbS^n_+}[\cdot]$ in \cref{alg:EGregularized} could be replaced with rank-r truncated projections without changing the sequences $\lbrace(\X_t,(\y_t,\blambda_t))\rbrace_{t\ge1}$ and $\lbrace(\Z_t,(\w_t,\bmu_t))\rbrace_{t\ge2}$. Moreover, for all $T\ge0$ it holds that
\begin{align*}
 f(\bar{\X})-f(\X^*) & \le R^2/(2\eta T),
\\   \Vert\mA(\bar{\X})-\b\Vert_2
& \le R^2/(2\Vert\y^*\Vert_2\eta T),
\\ \sum_{i\in\lbrace i\in[d]~\vert~g_i(\bar{\X})>0\rbrace}g_i(\bar{\X}) & \le R^2/(2\eta T),
\end{align*}
where $\bar{\X}=\frac{1}{T}\sum_{t=1}^T \Z_{t+1}$ and $R^2:=(\Vert\X_1-\X^*\Vert_F^2+\max\limits_{\y\in\mC_y}\Vert\y_1-\y\Vert_2^2+\max\limits_{\blambda\in\mC_{\lambda}}\Vert\blambda_1-\blambda\Vert_2^2)$ such that $\mC_y:=\lbrace\y\in\reals^m\ \vert\ \Vert\y\Vert_2\le2\Vert\y^*\Vert_2\rbrace$ and $\mC_{\lambda}:=\lbrace\blambda\in\reals^d_+\ \vert\ \Vert\blambda\Vert_2\le\Vert\blambda^*\Vert_2+\sqrt{d}\rbrace$.

Furthermore, 
for all $T\ge0$ it holds that
\begin{align*}
 f(\Z_{T+1})-f(\X^*) & \le R_1/(\eta\sqrt{T}),
\\ \Vert\mA(\Z_{T+1})-\b\Vert_2 & \le R_1/(\Vert\y^*\Vert_2\eta\sqrt{T})
\\ \sum_{i\in\lbrace i\in[d]~\vert~g_i(\Z_{T+1})>0\rbrace}g_i(\Z_{T+1}) & \le R_1/(\eta \sqrt{T}),
\end{align*}
where $R_1:=3\sqrt{2} \max\lbrace 2R_0(r),3\sqrt{2}\Vert\y^*\Vert_2,\sqrt{2}(2\Vert\blambda^*\Vert+\sqrt{d})\rbrace
\Vert(\X_1,\y_1,\blambda_1)-(\X^*,\y^*,\blambda^*)\Vert$.
\end{theorem}

\begin{remark}
In practice, it may not be possible to estimate the parameters $\max_{\X\in\mD_X}\max_{i\in[d]}\Vert\nabla{}g_i(\X)\Vert_F$ and $\max_{\blambda\in\mD_{\lambda}}\Vert\blambda\Vert_1$ in the choice for the step-size $\eta$. Nevertheless, \cref{thm:inequlalityConstraints} at least demonstrates the existence of a fixed step-size for which our results hold. 
\end{remark}

\begin{remark}
Similarly to the setting in \cref{sec:SDP}, there exist explicit bounds on the dual  solutions $\y^*$ and $\blambda^*$, see \cref{appx:boundsOptimalDualSolutions}.
\end{remark}

\section{Empirical Evidence}
\label{sec:experiments}
\begin{table*}[t]
\caption{Numerical results for the Max-Cut problem.}
\label{table:maxCut}
\vskip 0.15in
\begin{center}
\begin{small}
\begin{sc}
\begin{tabular}{lccccccccc}
\toprule
  &  &  &  &  & \multicolumn{5}{c}{1st iter. from which rank of projections $\leq r$} \\
\cmidrule(lr){6-10}
   \multirowcell{-2}{graph\\ \ }  & \multirowcell{-2}{$r^*=$\\ $\rank(\X^*)$} & \multirowcell{-2}{strict\\ compl.} & \multirowcell{-2}{$\langle\C,\widehat{\X}-\X^*\rangle$ \\$/\vert\langle\C,\X^*\rangle\vert$} & \multirowcell{-2}{$\Vert\mA(\widehat{\X})-\b\Vert_2$\\ \ } & $r=r^*$   & $r=2r^*$ & $r=4r^*$ & $r=8r^*$ & $r=12r^*$ \\
\midrule
  \textbf{G1} & $13$ & $0.0188$ & $-1.8\times{10}^{-9}$ & $6.7\times{10}^{-8}$ & $120$ & $20$ & $15$ & $3$ & $3$ \\ 
  \textbf{G2} & $13$ & $0.0223$ & $-2.3\times{10}^{-9}$ & $2.6\times{10}^{-8}$ & $114$ & $20$ & $4$ & $3$ & $3$ \\ 
  \textbf{G3} & $14$ & $0.0388$ & $-1.4\times{10}^{-9}$ & $1.1\times{10}^{-7}$  & $161$ & $18$ & $3$ & $3$ & $3$ \\ 
    \textbf{G4} & $14$ & $0.0913$ & $-1.2\times{10}^{-8}$ & $1.1\times{10}^{-7}$  & $82$ & $19$ & $3$ & $3$ & $3$ \\ 
    \textbf{G5} & $12$ & $0.0274$ & $-3.0\times{10}^{-9}$ & $2.4\times{10}^{-8}$ & $166$ & $23$ & $15$ & $3$ & $3$ \\
 \textbf{G6} & $13$ & $0.0232$ & $-4.2\times{10}^{-9}$ & $5.3\times{10}^{-8}$ & $114$ & $15$ & $7$ & $2$ & $2$ \\ 
   \textbf{G7} & $12$ & $0.0011$ & $-2.7\times{10}^{-8}$ & $5.4\times{10}^{-8}$ & $269$ & $16$ & $7$ & $2$ & $2$ \\ 
 \textbf{G8} & $12$ & $0.0568$ & $-5.2\times{10}^{-9}$ & $8.6\times{10}^{-9}$ & $105$ & $16$ & $7$ & $2$ & $2$ \\ 
  \textbf{G9} & $12$ & $0.0453$ & $-4.0\times{10}^{-9}$ & $5.0\times{10}^{-7}$ & $126$ & $17$ & $7$ & $2$ & $2$ \\ 
   \textbf{G10} & $12$ & $0.0274$ & $-3.3\times{10}^{-9}$ & $2.4\times{10}^{-8}$ & $108$ & $16$ & $7$ & $2$ & $2$ \\ 
  \textbf{G11} & $6$ & $2.0\times{10}^{-5}$ & $1.6\times{10}^{-4}$ & $5.7\times{10}^{-6}$ & $-$ & $-$ & $942$ & $55$ & $21$  \\ 
  \textbf{G12} & $8$ & $1.7\times{10}^{-4}$ & $2.8\times{10}^{-4}$ & $1.7\times{10}^{-5}$  & $-$ & $1960$ & $135$ & $24$ & $4$ \\ 
  \textbf{G13} & $8$ & $0.0014$ & $2.5\times{10}^{-4}$ & $2.0\times{10}^{-5}$ & $-$ & $1350$ & $128$ & $21$ & $13$  \\ 
  \textbf{G14} & $13$ & $0.0174$ & $-8.2\times{10}^{-9}$ & $9.0\times{10}^{-7}$ & $645$ & $195$ & $68$ & $31$ & $23$  \\ 
  \textbf{G15} & $13$ & $0.0063$ & $-2.6\times{10}^{-8}$ & $2.0\times{10}^{-6}$ & $863$ & $198$ & $69$ & $33$ & $23$  \\ 
  \textbf{G16} & $14$ & $0.0177$ & $-1.1\times{10}^{-8}$ & $2.7\times{10}^{-6}$ & $637$ & $150$ & $61$ & $24$ & $14$  \\ 
   \textbf{G17} & $13$ & $0.0188$ & $-2.2\times{10}^{-8}$ & $1.6\times{10}^{-6}$ & $575$ & $165$ & $66$ & $29$ & $20$  \\ 
  \textbf{G18} & $10$ & $0.0083$ & $-4.2\times{10}^{-9}$ & $9.7\times{10}^{-9}$ & $680$ & $41$ & $17$ & $8$ & $3$ \\ 
  \textbf{G19} & $9$ & $0.0058$ & $-2.6\times{10}^{-8}$ & $5.3\times{10}^{-9}$ & $586$ & $45$ & $19$ & $9$ & $3$ \\ 
    \textbf{G20} & $9$ & $0.0142$ & $-8.3\times{10}^{-9}$ & $3.1\times{10}^{-9}$ & $692$ & $46$  & $20$& $9$ & $3$ \\
\bottomrule
\end{tabular}
\end{sc}
\end{small}
\end{center}
\vskip -0.1in
\end{table*}
To support our theoretical findings, we demonstrate our results on a dataset of Max-Cut SDP instances. The standard formulation of the Max-Cut SDP is:
\vspace{-1mm}
\begin{align*}
\min_{} \langle\C,\X\rangle
\ \textrm{ s.t.}\ & \diag(\X)=1
\textrm{ and } \X\succeq0, 
\end{align*}
where $\C\in\mathbb{S}^n$ is the negative graph Laplacian. This problem is an instance of Problem \eqref{mainProblemSDP}.
Following  \cite{SDPstrictComplementarity}, we use graphs $\textbf{G1}-\textbf{G20}$ of the Gset dataset \cite{Gset} for which $n=800$. 
For each graph we use the CVX SDPT3 4.0 solver with HKM search directions to obtain an optimal solution $(\X^*,\y^*)$ and use it to estimate the rank of the optimal solution $r^*=\rank(\X^*)$ (as the number of eigenvalues larger than $10^{-2}$, since these are at least three orders of magnitudes larger) and the strict complementarity measure is calculated as $\lambda_{n-r^*}(\nabla_{\X}\mL(\X^*,\y^*))$.

First, we run Algorithm \ref{alg:EG} with SVDs of rank $r^*=\rank(\X^*)$ to produce an approximated solution $\widehat{\X}$, which we take to be the last iterate $\Z_{T+1}$. We initialize the $\X$ variable 
with the matrix $\X_1=(1/\trace(\boldsymbol{\Lambda}_{r^*}))\sign(\V_{r^*})\boldsymbol{\Lambda}_{r^*}\sign(\V_{r^*})^{\top}$, where $\V_{r^*}(-\boldsymbol{\Lambda}_{r^*}){\V_{r^*}}^{\top}$ is the rank-$r^*$ approximation of $-\C$ (this always resulted in a PSD matrix), and we initialize the $\y$ variable with $\y_1=\textbf{0}$. The theoretical step-size for this problem is $0.5$, yet we manually tuned the step-size $\eta$ for best performance. 
To verify the convergence of the method to an optimal primal point, we set the number of iterations $T$ to be sufficiently large to obtain a (relative) approximation error no larger than $10^{-4}$ calculated as $\langle\C,\widehat{\X}-\X^*\rangle/\vert\langle\C,\X^*\rangle\vert$, and we calculate the feasibility error by $\Vert\mA(\widehat{\X})-\b\Vert_2$.
Details regarding the specific choices for $\eta$ and $T$ are provided in \cref{appx:moreExperiments}. 
In addition, we record the number of the first iteration starting from which Condition \eqref{cond:lowRankProjCondition} holds for all subsequent gradient steps $\X_t-\eta\nabla_{\X}\mL(\X_t,\y_t)$ and $\X_t-\eta\nabla_{\X}\mL(\Z_{t+1},\w_{t+1})$.  

We repeated the above simulations using higher ranks of SVD computations in Algorithm  \ref{alg:EG} (to compute the truncated projections $\widehat{\Pi}_{\mbS^n_+}^r[\cdot[$). We take: $r = 2r^*, 4r^*, 8r^*, 16r^*$.

As can be seen in Table \ref{table:maxCut}, all instances indeed satisfy  strict complementarity. We also see that for all instances, except for \textbf{G11-G13} for which the strict complementarity measure is very small, indeed already when setting $r=r^*$ and from a relatively early stage of the run (for graphs \textbf{G1-G10} we used 1000 iterations and for graphs \textbf{G14-G20} we used 5000 iterations) indeed all projections are of rank at most $r^*$, meaning that the method indeed converges to an optimal solution. We also clearly see  how increasing the size of the SVD rank $r$ dramatically improves the situation, by significantly lowering the index of the first iteration from which on, all projection are of rank at most $r$. This is in particular apparent for the ill-conditioned instances \textbf{G11-G13}. Note that for \textbf{G11} for instance, even taking $r=12r^*$ still results in rank less than 10\% of the dimension (800), which is the rank required in worst-case by standard projected gradient methods for convex optimization.

Since increasing the SVD rank $r$  intuitively also increases the runtime of computing SVDs, we plot the runtime vs. the number of iterations. In Figure \ref{fig:time_tableTC_some} we show the plots for the two most ill-conditioned instances \textbf{G11, G12}, the rest are given in \cref{appx:moreExperiments} together with accompanying details. We can see in Table \ref{table:maxCut} and Figure \ref{fig:time_tableTC_some} that, while increasing $r$ can dramatically improve  convergence (in the sense that truncated projections match the accurate ones), the cost in additional runtime is quite moderate and in particular sublinear in $r$. It can even be the case (as with \textbf{G11} in  Figure \ref{fig:time_tableTC_some}) that increasing $r$ results in better-conditioned SVD computations which reduces runtime due to faster convergence.
\vspace{-20pt}
\begin{figure}[h] 
\vskip 0.2in
\begin{center}
\centering   
\subfigure[G11]{\includegraphics[width=0.49\columnwidth]{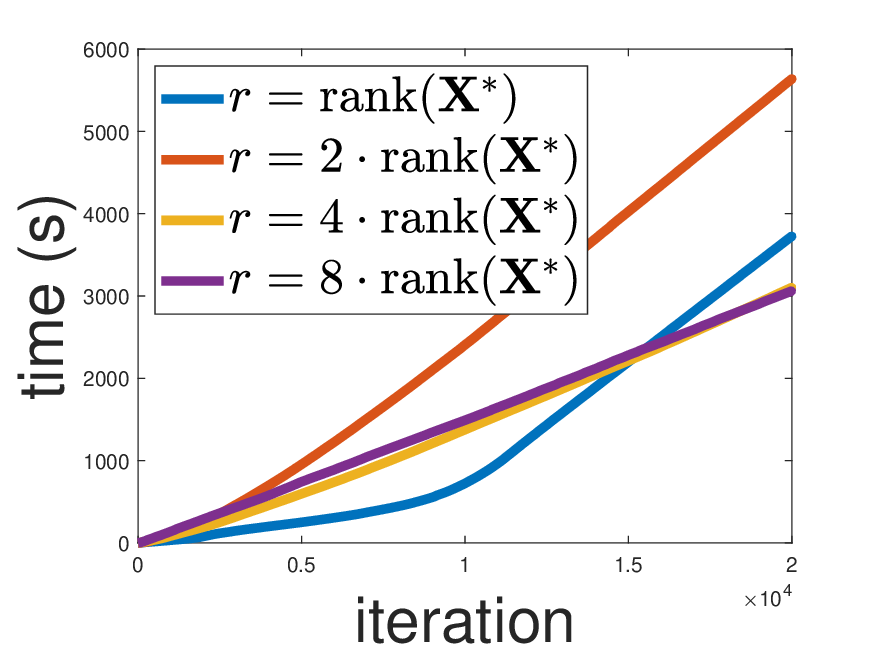}}
\subfigure[G12]{\includegraphics[width=0.49\columnwidth]{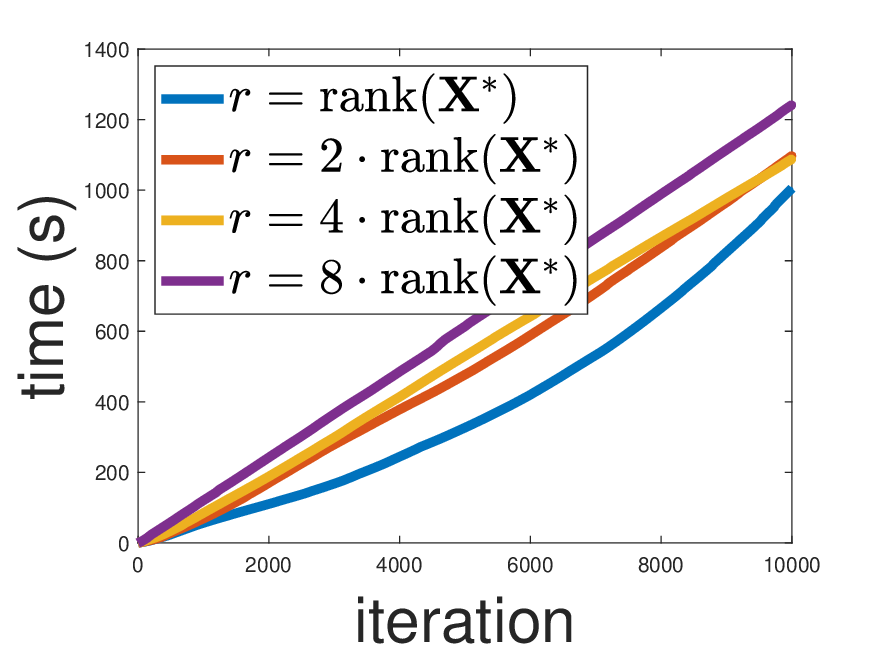}}
\caption{Time as a function of number of iterations it takes to run Algorithm \ref{alg:EG} with varying SVD ranks for the Max-Cut problem.}
\label{fig:time_tableTC_some}
\end{center}
\vskip -0.2in
\end{figure}

Finally, as discussed in \cref{sec:relatedWork}, the highly-popular Burer-Monteiro approach for the Max-Cut problem requires a factorization of a size at least $r(r+1)/2+r>n$ to ensure convergence to a global optimal solution \cite{waldspurger2020rank, SDPstrictComplementarity}. In our case, this corresponds to a factorization of rank $r\ge39$. In contrast, our method with truncated projections of rank below the Burer-Monteiro threshold, such as when $r=r^*$, $r=2r^*$, and in some instances even $r=4r^*$, exhibits on almost all instances provable convergence after just a few dozen iterations.

\nocite{langley00}

\bibliography{bibs}
\bibliographystyle{icml2024}

\newpage
\appendix
\onecolumn
\section{Proofs Omitted from Section \ref{sec:SDP}}\label{appx:proofsSecSDP}

We first restate each lemma and then prove it.

%
%

\subsection{Proof of Lemma \ref{lemma:radius}}
\label{appx:proofLemma24}

To prove \cref{lemma:radius} we first prove the following lemma, which shows the relationship between the range of the optimal solution $\X^*$ and the null-space of the gradient with respect to $\X$ of the Lagrangian at a saddle-point $(\X^*,\y^*)$. This proves the complementarity condition \eqref{eq:compcond} holds.
\begin{lemma} \label{lemma:range_null}
Let $(\X^*,\y^*)$ be a saddle-point of Problem \eqref{problem:generalSDPprimal} and assume Assumptions \ref{Ass:slater} and \ref{Ass:primalBounded} hold. Then, $\nabla_{\X}\mL(\X^*,\y^*)=\nabla{}f(\X^*)-\mA^{\top}(\y^*)\succeq0$ and $\range(\X^*)\subseteq\nullspace(\nabla_{\X}\mL(\X^*,\y^*))$.
\end{lemma}

\begin{proof}
Under Assumptions \ref{Ass:slater} and \ref{Ass:primalBounded}, strong duality holds for Problem \eqref{problem:generalSDPprimal}. This implies that for the optimal solution $\X^*$ and some optimal dual solution $(\y^*,\Z^*)$ the KKT conditions for Problem \eqref{problem:generalSDPprimal} hold.
Writing the KKT conditions we obtain that
\begin{align*}
& 0=\nabla{}f(\X^*)-\mA^{\top}(\y^*)-\Z^*,
\\ & \langle\Z^*,\X^*\rangle=0,
\\ & \mA(\X^*)=\b,
\\ & \X^*,\Z^*\succeq0.
\end{align*} 
In particular, since $\Z^*\succeq0$ we have from the first order optimality condition that $\Z^*=\nabla{}f(\X^*)-\mA^{\top}(\y^*)\succeq0$.

In addition, the fact that $\langle\Z^*,\X^*\rangle=0$ and $\X^*,\Z^*\succeq0$ imply that $\Z^*\X^*=0$. This means that $\range(\X^*)\subseteq\nullspace(\Z^*)$. 

\end{proof}

Now we restate and prove \cref{lemma:radius}.
\begin{lemma}
Let $(\X^*,\y^*)$ be a saddle-point of Problem \eqref{problem:generalSDPprimal} and denote the rank of the complementarity condition $\tilde{r}=n-\rank(\nabla_{\X}\mL(\X^*,\y^*))$. For any $r\ge\tilde{r}$, $\eta\ge0$, and $\X\in\mbS^n,(\Z,\w)\in\mathbb{S}^n_+\times\reals^m$, if
\begin{align*}
& \max\lbrace\Vert\X-\X^*\Vert_F,\Vert(\Z,\w)-(\X^*,\y^*)\Vert\rbrace
\le \frac{\eta\max\left\lbrace\lambda_{n-r}(\nabla_{\X}\mL(\X^*,\y^*)),\sqrt{\tilde{r}}\lambda_{n-r+\tilde{r}-1}(\nabla_{\X}\mL(\X^*,\y^*))\right\rbrace}{1+\sqrt{2}\eta\max\lbrace\beta,\Vert\mA\Vert\rbrace},
\end{align*}
then $\rank(\Pi_{\mathbb{S}^n_+}[\X-\eta\nabla_{\X}\mL(\Z,\w)])\le r$ (meaning $\Pi_{\mathbb{S}^n_+}[\X-\eta\nabla_{\X}\mL(\Z,\w)]) =\widehat{\Pi}^r_{\mathbb{S}^n_+}[\X-\eta\nabla_{\X}\mL(\Z,\w)]$). 
\end{lemma}

\begin{proof}
Denote $\P:=\X-\eta\nabla{}\mL(\Z,\w)$. For $\rank(\Pi_{\mathbb{S}^n_+}[\P])\le r$ it must hold that $\lambda_{r+1}(\P)\le0$. 

Denote $\P^*:=\X^*-\eta\nabla{}\mL(\X^*,\y^*)$. Invoking \cref{lemma:range_null} we have that any eigenvector $\v$ of $\X^*$ corresponding to a positive eigenvalue is an eigenvector of $\nabla_{\X}\mL(\X^*,\y^*)$ corresponding to the eigenvalue $0$. Therefore, $\P^*$ satisfies that
\begin{align} \label{ineq:inProofLambdasOpt}
i\le \rank(\X^*):\quad & \lambda_i(\P^*) =  \lambda_i(\X^*), \nonumber
\\ i>\rank(\X^*):\quad & \lambda_i(\P^*) =  -\eta\lambda_{n-i+1}(\nabla_{\X}\mL(\X^*,\y^*)).
\end{align}

We will first bound the distance between $\P$ and $\P^*$.
Using the smoothness of $\mL$ we have that
\begin{align} 
\Vert\nabla_{\X}\mL(\Z,\w)-\nabla_{\X}\mL(\X^*,\y^*)\Vert_F & \le \Vert\nabla_{\X}\mL(\Z,\w)-\nabla_{\X}\mL(\X^*,\w)\Vert_F + \Vert\nabla_{\X}\mL(\X^*,\w)-\nabla_{\X}\mL(\X^*,\y^*)\Vert_F \nonumber
\\ & = \Vert\nabla{}f(\Z)-\nabla{}f(\X^*)\Vert_F + \Vert\mA^{\top}(\w-\y^*)\Vert_F \nonumber
\\ & \le \beta\Vert\Z-\X^*\Vert_F+\Vert\mA\Vert\Vert\w-\y^*\Vert_2 \nonumber
\\ & \le \max\lbrace\beta,\Vert\mA\Vert\rbrace\left(\Vert\Z-\X^*\Vert_F+\Vert\w-\y^*\Vert_2\right) \nonumber
\\ & \le \sqrt{2}\max\lbrace\beta,\Vert\mA\Vert\rbrace\left(\Vert(\Z,\w)-(\X^*,\y^*)\Vert\right) \label{ineq:inProofSmoothnessXBound_1}
\end{align}
Therefore, 
\begin{align} \label{ineq:inProofP_Pstar_dis}
\Vert\P-\P^*\Vert_F  & \le \Vert\X-\X^*\Vert_F+\eta\Vert\nabla_{\X}\mL(\Z,\w)-\nabla_{\X}\mL(\X^*,\y^*)\Vert_F \nonumber
\\ & \le \Vert\X-\X^*\Vert_F +\sqrt{2}\eta\max\lbrace\beta,\Vert\mA\Vert\rbrace\Vert(\Z,\w)-(\X^*,\y^*)\Vert \nonumber
\\ & \le (1+\sqrt{2}\eta\max\lbrace\beta,\Vert\mA\Vert\rbrace)\max\lbrace\Vert\X-\X^*\Vert_F,\Vert(\Z,\w)-(\X^*,\y^*)\Vert\rbrace,
\end{align}
where the second inequality follows from \eqref{ineq:inProofSmoothnessXBound_1}.

Now we will bound $\lambda_{r+1}(\P)$. Using Weyl's inequality we have that
\begin{align*}
\lambda_{r+1}(\P) 
& \le  \lambda_{r+1}(\P^*) + \lambda_1(\P-\P^*)
\\ & \underset{(a)}{=} -\eta\lambda_{n-r}(\nabla_{\X}\mL(\X^*,\y^*)) + \lambda_1(\P-\P^*)
\\ & \le -\eta\lambda_{n-r}(\nabla_{\X}\mL(\X^*,\y^*)) + \Vert\P-\P^*\Vert_F
\\ & \underset{(b)}{\le} -\eta\lambda_{n-r}(\nabla_{\X}\mL(\X^*,\y^*))
+ (1+2\eta\max\lbrace\beta,\Vert\mA\Vert\rbrace)\max\lbrace\Vert\X-\X^*\Vert_F,\Vert(\Z,\w)-(\X^*,\y^*)\Vert\rbrace,
\end{align*}
where (a) follows from \eqref{ineq:inProofLambdasOpt} and (b) follows from \eqref{ineq:inProofP_Pstar_dis}.
Therefore, the condition $\lambda_{r+1}(\P)\le0$ holds if
\begin{align} \label{ineq:RadiusOpt1}
& \max\lbrace\Vert\X-\X^*\Vert_F,\Vert(\Z,\w)-(\X^*,\y^*)\Vert\rbrace \le \frac{\eta}{1+\sqrt{2}\eta\max\lbrace\beta,\Vert\mA\Vert\rbrace}\lambda_{n-r}(\nabla_{\X}\mL(\X^*,\y^*)).
\end{align}

Alternatively, if $r\ge 2\tilde{r}-1$ then using the general Weyl inequality we obtain that
\begin{align*}
\lambda_{r+1}(\P) 
& \le  \lambda_{r-\tilde{r}+2}(\P^*) + \lambda_{\tilde{r}}(\P-\P^*)
\underset{(a)}{=} -\eta\lambda_{n-r+\tilde{r}-1}(\nabla_{\X}\mL(\X^*,\y^*)) + \lambda_{\tilde{r}}(\P-\P^*)
\\ & = -\eta\lambda_{n-r+\tilde{r}-1}(\nabla_{\X}\mL(\X^*,\y^*)) +\sqrt{\lambda_{\tilde{r}}^2(\P-\P^*)} 
\\ & \le -\eta\lambda_{n-r+\tilde{r}-1}(\nabla_{\X}\mL(\X^*,\y^*)) +\frac{1}{\sqrt{\tilde{r}}}\Vert\P-\P^*\Vert_F
\\ & \underset{(b)}{\le} -\eta\lambda_{n-r+\tilde{r}-1}(\nabla_{\X}\mL(\X^*,\y^*)) 
+ \frac{1}{\sqrt{\tilde{r}}}(1+\sqrt{2}\eta\max\lbrace\beta,\Vert\mA\Vert\rbrace)\max\lbrace\Vert\X-\X^*\Vert_F,\Vert(\Z,\w)-(\X^*,\y^*)\Vert\rbrace,
\end{align*} 
where here too (a) follows from \eqref{ineq:inProofLambdasOpt} and (b) follows from \eqref{ineq:inProofP_Pstar_dis}.
Therefore, the condition $\lambda_{r+1}(\P)\le0$ holds if
\begin{align} \label{ineq:RadiusOpt2}
\max\lbrace\Vert\X-\X^*\Vert_F,\Vert(\Z,\w)-(\X^*,\y^*)\Vert\rbrace \le \frac{\eta}{1+\sqrt{2}\eta\max\lbrace\beta,\Vert\mA\Vert\rbrace}\sqrt{\tilde{r}}\lambda_{n-r+\tilde{r}-1}(\nabla_{\X}\mL(\X^*,\y^*)).
\end{align}
Taking the maximum between \eqref{ineq:RadiusOpt1} and \eqref{ineq:RadiusOpt2} we obtain the radius in the lemma.

\end{proof}

\subsection{Proof of Lemma \ref{lemma:boundf_and_norm}}
\label{appx:proofLemma25}

\begin{lemma}
Let $\X^*\succeq0$ be an optimal solution to Problem \eqref{mainProblemSDP} and let $(\widehat{\X},\widehat{\y})\in\mathbb{S}^n_+\times\reals^m$. Let $\y^*\in\argmax_{\y\in\reals^m}\mL(\X^*,\y)$.
Denote $\mC=\lbrace\y\in\reals^m\ \vert\ \Vert\y\Vert_2\le2\Vert\y^*\Vert_2\rbrace$. Assume that for all $\y\in\mC$ 
\begin{align*}
\mathcal{L}(\widehat{\X},\y)-\mathcal{L}(\X^*,\widehat{\y})\le\varepsilon.
\end{align*}
Then it holds that
$f(\widehat{\X})-f(\X^*)\le\varepsilon$ and $\Vert\mA(\widehat{\X})-\b\Vert_2\le\varepsilon/\Vert\y^*\Vert_2$.
\end{lemma}

\begin{proof}
For all $\y\in\mC$ it holds that
$\mathcal{L}(\widehat{\X},\y)-\mathcal{L}(\X^*,\widehat{\y})\le\varepsilon$.
In particular, denoting $\alpha:=2\Vert\y^*\Vert$ and plugging in $\y=\frac{\alpha(\b-\mA(\widehat{\X}))}{\Vert\mA(\widehat{\X})-\b\Vert_2}\in\mC$, we obtain
\begin{align} \label{ineq:func&normBound}
& f(\widehat{\X})+\y^{\top}(\b-\mA(\widehat{\X}))-f(\X^*) = f(\widehat{\X})-f(\X^*)+\alpha\Vert\mA(\widehat{\X})-\b\Vert_2\le\varepsilon.
\end{align}

Since $\Vert\mA(\widehat{\X})-\b\Vert_2\ge0$ it follows immediately that $f(\widehat{\X})-f(\X^*)\le\varepsilon$. 
In addition, since $(\X^*,\y^*)$ is a saddle-point of Problem \eqref{problem:generalSDPprimal} it holds that 
\begin{align*}
f(\X^*) = \mL(\X^*,\y^*) \le \mL(\widehat{\X},\y^*) = f(\widehat{\X}) + {\y^*}^{\top}(\b-\mA(\widehat{\X})).
\end{align*}
Plugging this into \eqref{ineq:func&normBound}, we obtain
\begin{align} \label{ineq:func&normBound_1}
\alpha\Vert\mA(\widehat{\X})-\b\Vert_2 \le \varepsilon+{\y^*}^{\top}(\b-\mA(\widehat{\X})) \le \varepsilon+\Vert\y^*\Vert_2\Vert\mA(\widehat{\X})-\b\Vert_2. 
\end{align}
Rearranging, we obtain the desired result of $\Vert\mA(\widehat{\X})-\b\Vert_2\le\varepsilon/(\alpha-\Vert\y^*\Vert_2)=\varepsilon/\Vert\y^*\Vert_2$.

\end{proof}

\subsection{Proof of \cref{thm:SDP}}
\label{appx:proofThm26}
\begin{theorem}
Fix an optimal solution $\X^*\succeq0$ to Problem \eqref{mainProblemSDP}. Let $\y^*\in\reals^m$ be a corresponding dual solution and suppose $(\X^*,\y^*)$ satisfies the complementarity condition \eqref{eq:compcond} with parameter $\tilde{r}:=n-\rank(\nabla_{\X}\mL(\X^*,\y^*))$. 
Let $\lbrace(\X_t,\y_t)\rbrace_{t\ge1}$ and $\lbrace(\Z_t,\w_t)\rbrace_{t\ge2}$ be the sequences of iterates generated by Algorithm \ref{alg:EG} with a fixed step-size
$\eta \le \frac{1}{\sqrt{2(\beta^2+\Vert\mA\Vert^2)}}$. Fix some $r \geq \tilde{r}$ and assume the initialization $(\X_1,\y_1)$ satisfies that $\Vert(\X_1,\y_1)-(\X^*,\y^*)\Vert_F\le R_0(r)$, where
\begin{align*}
& R_0(r):= 
\frac{\eta\max\left\lbrace\lambda_{n-r}(\nabla_{\X}\mL(\X^*,\y^*)),\sqrt{\tilde{r}}\lambda_{n-r+\tilde{r}-1}(\nabla_{\X}\mL(\X^*,\y^*))\right\rbrace}{2(1+\sqrt{2}\eta\max\lbrace\beta,\Vert\mA\Vert\rbrace)}.
\end{align*}
Then, for all $t\ge1$, the projections $\Pi_{\mbS^n_+}[\cdot]$ in Algorithm \ref{alg:EG} could be replaced with rank-$r$ truncated projections (as defined in \eqref{def:rankRprojection}) without changing the sequences $\lbrace(\X_t,\y_t)\rbrace_{t\ge1}$ and $\lbrace(\Z_t,\w_t)\rbrace_{t\ge2}$. Moreover, for all $T\ge0$ it holds that
\begin{align*}
& f\left(\frac{1}{T}\sum_{t=1}^T\Z_{t+1}\right)-f(\X^*)\le R^2/(2\eta T),
\\ & \left\Vert\mA\left(\frac{1}{T}\sum_{t=1}^T\Z_{t+1}\right)-\b\right\Vert_2\le R^2/(2\Vert\y^*\Vert_2\eta T),
\end{align*}
where 
$R^2:=\Vert\X_1-\X^*\Vert_F^2+\max\limits_{\y\in\lbrace\y\in\reals^m\ \vert\ \Vert\y\Vert_2\le2\Vert\y^*\Vert_2\rbrace}\Vert\y_1-\y\Vert_2^2$.

Furthermore, if we take a step-size $\eta=\frac{1}{2\sqrt{\beta^2+\Vert\mA\Vert^2}}$, then for all $T\ge0$ it holds that
\begin{align*}
& f(\Z_{T+1})-f(\X^*)\le R_1/(\eta\sqrt{T}),
\\ & \Vert\mA(\Z_{T+1})-\b\Vert_2\le R_1/(\Vert\y^*\Vert_2\eta\sqrt{T}),
\end{align*}
where $R_1:=3\sqrt{2}\max\lbrace 2R_0(r),3\Vert\y^*\Vert_2\rbrace\Vert(\X_1,\y_1)-(\X^*,\y^*)\Vert$.
\end{theorem}

\begin{proof}
To invoke the convergence results of the projected extragradient method obtained in previous papers, we denote with $\beta_X,\beta_y,\beta_{Xy},\beta_{yX}\ge0$ the constants such that for any $\X,\widetilde{\X}\succeq0$ and $\y,\widetilde{\y}\in\reals^m$ the following four inequalities hold:
\begin{align} \label{ineq:betas}
& \Vert\nabla_{\X}\mL(\X,\y)-\nabla_{\X}\mL(\widetilde{\X},\y)\Vert_{F} \le\beta_{X}\Vert\X-\widetilde{\X}\Vert_{F}, \nonumber \\
&\Vert\nabla_{\y}\mL(\X,\y)-\nabla_{\y}\mL(\X,\widetilde{\y})\Vert_{2} \le\beta_{y}\Vert\y-\widetilde{\y}\Vert_{2}, \nonumber \\
& \Vert\nabla_{\X}\mL(\X,\y)-\nabla_{\X}\mL(\X,\widetilde{\y})\Vert_{F} \le\beta_{Xy}\Vert\y-\widetilde{\y}\Vert_{2}, \nonumber \\
& \Vert\nabla_{\y}\mL(\X,\y)-\nabla_{\y}\mL(\widetilde{\X},\y)\Vert_{2} \le\beta_{yX}\Vert\X-\widetilde{\X}\Vert_{F},
\end{align}
and we denote by $\beta_{\mL}$ the full Lipschitz parameter of the gradient of $\mL$, that is for any $\X,\widetilde{\X}\succeq0$ and $\y,\widetilde{\y}\in\reals^m$, 
\begin{align*}
\Vert(\nabla_{\X}\mL(\X,\y),-\nabla_{\y}\mL(\X,\y))-(\nabla_{\X}\mL(\widetilde{\X},\widetilde{\y}),-\nabla_{\y}\mL(\widetilde{\X},\widetilde{\y}))\Vert \le\beta_{\mL}\Vert(\X,\Y)-(\widetilde{\X},\widetilde{\y})\Vert.
\end{align*}
As shown in \cite{ourExtragradient}, the relationship between $\beta_{\mL}$ and $\beta_X,\beta_y,\beta_{Xy},\beta_{yX}$ is given by $\beta_{\mL} = \sqrt{2}\max\lbrace\sqrt{\beta_X^2+\beta_{yX}^2},\sqrt{\beta_y^2+\beta_{Xy}^2}\rbrace$.
By the definition of $\mL$ it is simple to see that $\beta_{X} = \beta$, $\beta_{Xy}=\beta_{yX}=\Vert\mA\Vert$, and $\beta_y=0$. Therefore, we have that $\beta_{\mL}=\sqrt{2}\sqrt{\beta^2+\Vert\mA\Vert^2}$.

Assuming $\eta\le 1/\beta_{\mL}$, by invoking Lemma 8 in \cite{ourExtragradient}, we have that for all $t\ge2$ the iterates of the projected extragradient method satisfy that
\begin{align*}
\Vert(\X_t,\y_t)-(\X^*,\y^*)\Vert\le \Vert(\X_{t-1},\y_{t-1})-(\X^*,\y^*)\Vert,
\end{align*}
and from the nonexpansiveness of the Euclidean projection, we have that for all $t\ge1$
\begin{align*}
\Vert(\Z_{t+1},\w_{t+1})-(\X^*,\y^*)\Vert & \le \left\Vert\left(\begin{array}{c}\X_{t}-\eta\nabla_{\X}\mL(\X_{t},\y_{t})
\\ \y_{t}+\eta\nabla_{\y}\mL(\X_{t},\y_{t})\end{array}\right)-\left(\begin{array}{c}\X^*-\eta\nabla_{\X}\mL(\X^*,\y^*)
\\ \y^*+\eta\nabla_{\y}\mL(\X^*,\y^*)
\end{array}\right)\right\Vert
\\ & \le \left\Vert\left(\begin{array}{c}\X_{t} \\ \y_{t} \end{array}\right)-\left(\begin{array}{c}\X^* \\ \y^*\end{array}\right)\right\Vert 
+\eta\left\Vert\left(\begin{array}{c}\nabla_{\X}\mL(\X_{t},\y_{t}) \\ -\nabla_{\y}\mL(\X_{t},\y_{t}) \end{array}\right)-\left(\begin{array}{c}\nabla_{\X}\mL(\X^*,\y^*) \\ -\nabla_{\y}\mL(\X^*,\y^*) \end{array}\right)\right\Vert
\\ & \le (1+\eta\beta_{\mL})\Vert(\X_{t},\y_{t})-(\X^*,\y^*)\Vert
\\ & \le 2\Vert(\X_{t},\y_{t})-(\X^*,\y^*)\Vert.
\end{align*}
%
Unrolling the recursion and using our initialization choice of $(\X_1,\y_1)$, we obtain that for all $t\ge1$, 
\begin{align*}
\max\left\lbrace\Vert(\X_{t},\y_{t})-(\X^*,\y^*)\Vert,\Vert(\Z_{t+1},\w_{t+1})-(\X^*,\y^*)\Vert\right\rbrace 
& \le 2\Vert(\X_{t-1},\y_{t-1})-(\X^*,\y^*)\Vert 
\\ & \le \ldots \le 2\Vert(\X_{1},\y_{1})-(\X^*,\y^*)\Vert\le 2R_0(r).
\end{align*}

Since for all $t\ge1$ it holds that 
\begin{align*}
\Vert\X_{t}-\X^*\Vert & \le \Vert(\X_{t},\y_{t})-(\X^*,\y^*)\Vert \le\max\left\lbrace\Vert(\X_{t},\y_{t})-(\X^*,\y^*)\Vert,\Vert(\Z_{t+1},\w_{t+1})-(\X^*,\y^*)\Vert\right\rbrace,
\end{align*}
we have that for all $t\ge1$ the condition in \cref{lemma:radius} holds for $(\X_{t},\y_{t})$ and $(\Z_{t+1},\w_{t+1})$, and so for all $t\ge1$ it follows that 
\begin{align*}
& \rank(\Pi_{\mbS^n_+}[\X_t-\eta\nabla{}_{\X}\mL(\X_t,\y_t)])\le r
\\ & \rank(\Pi_{\mbS^n_+}[\X_t-\eta\nabla{}_{\X}\mL(\Z_{t+1},\w_{t+1})])\le r.
\end{align*}
Hence, the iterates of projected extragradient method will remain unchanged when replacing all projections onto the PSD cone with their rank-$r$ truncated counterparts, and so, the method will also maintain its original convergence rate.

In the proof of the projected extragradient method (see within proof of Lemma 7 in \cite{ourExtragradient}), the following inequality has been derived for any $\X\succeq0$ and $\y\in\reals^m$ and step-size $\eta\le\min\left\lbrace\frac{1}{\beta_{X}+\beta_{Xy}},\frac{1}{\beta_{y}+\beta_{yX}},\frac{1}{\beta_{X}+\beta_{yX}},\frac{1}{\beta_{y}+\beta_{Xy}}\right\rbrace$:
\begin{align*}
\mL(\Z_{t+1},\y) - \mL(\X,\w_{t+1}) & \le \frac{1}{2\eta}\left(\Vert(\X_t,\y_t)-(\X,\y)\Vert^2-\Vert(\X_{t+1},\y_{t+1})-(\X,\y)\Vert^2\right).
\end{align*}

Denote $\mC=\lbrace\y\in\reals^m\ \vert\ \Vert\y\Vert_2\le2\Vert\y^*\Vert_2\rbrace$. 
\begin{align} \label{ineq:inProof1111}
\frac{1}{T}\sum_{t=1}^T \mL(\Z_{t+1},\y) - \frac{1}{T}\sum_{t=1}^T \mL(\X^*,\w_{t+1}) 
& \le \frac{1}{2\eta T}\max_{\y\in\mC}\Vert(\X_1,\y_1)-(\X^*,\y)\Vert^2 \nonumber
\\ & = \frac{1}{2\eta T}\left(\Vert\X_1-\X^*\Vert_F^2+\max_{\y\in\mC}\Vert\y_1-\y\Vert_2^2\right).
\end{align}
By the linearity of $\mL(\X^*,\cdot)$ it follows that $\frac{1}{T}\sum_{t=1}^T \mL(\X^*,\w_{t+1}) = \mL(\X^*,\frac{1}{T}\sum_{t=1}^T\w_{t+1})$ and by the convexity of $\mL(\cdot,y)$ it follows that $\mL\left(\frac{1}{T}\sum_{t=1}^T\Z_{t+1},\y\right)\le\frac{1}{T}\sum_{t=1}^T \mL(\Z_{t+1},\y)$. Therefore, in particular, for all $\y\in\mC$ it holds that
\begin{align} \label{ineq:inProof111} 
\mL\left(\frac{1}{T}\sum_{t=1}^T\Z_{k+1},\y\right) -  \mL\left(\X^*,\frac{1}{T}\sum_{t=1}^T\w_{t+1}\right) 
& \le \frac{1}{2\eta T}\left(\Vert\X_1-\X^*\Vert_F^2+\max_{\y\in\mC}\Vert\y_1-\y\Vert_2^2\right).
\end{align}
Invoking \cref{lemma:boundf_and_norm} we have that as desired
\begin{align*}
f\left(\frac{1}{T}\sum_{t=1}^T\Z_{k+1}\right)-f(\X^*)\le \frac{1}{2\eta T}\left(\Vert\X_1-\X^*\Vert_F^2+\max_{\y\in\mC}\Vert\y_1-\y\Vert_2^2\right)
\end{align*}
and 
\begin{align*}
\left\Vert\mA\left(\frac{1}{T}\sum_{t=1}^T\Z_{k+1}\right)-\b\right\Vert_2\le\frac{1}{2\eta T}\frac{\left(\Vert\X_1-\X^*\Vert_F^2+\max_{\y\in\mC}\Vert\y_1-\y\Vert_2^2\right)}{\Vert\y^*\Vert_2}.
\end{align*}

Taking a step-size $\eta\le\frac{1}{\sqrt{2}\sqrt{\beta^2+\Vert\mA\Vert^2}}$ satisfies 
both conditions $\eta\le\min\left\lbrace\frac{1}{\beta_{X}+\beta_{Xy}},\frac{1}{\beta_{y}+\beta_{yX}},\frac{1}{\beta_{X}+\beta_{yX}},\frac{1}{\beta_{y}+\beta_{Xy}}\right\rbrace$ and $\eta\le\frac{1}{\beta_{\mL}}$ as required.

For the second part of the theorem, 
denote $\mD=\left\lbrace(\X,\y)\in\mathbb{S}^n_+\times\reals^m~|~\Vert(\X,\y)-(\X^*,\y^*)\Vert\le \max\lbrace 2R_0(r),3\Vert\y^*\Vert_2\rbrace\right\rbrace$. As we showed above, all iterates are within the ball $\mD$. By the convexity of $\mL(\cdot,\w_{T+1})$ and linearity of $\mL(\Z_{T+1},\cdot)$, for all $(\X,\y)\in\mD$
it holds that
\begin{align} \label{ineq:inProof000}
\mL(\Z_{T+1},\w_{T+1})-\mL(\X,\w_{T+1}) & \le \langle\Z_{T+1}-\X,\nabla_{\X}\mL(\Z_{T+1},\w_{T+1})\rangle \nonumber
\\ \mL(\Z_{T+1},\y)-\mL(\Z_{T+1},\w_{T+1}) & = \langle\w_{T+1}-\y,-\nabla_{\y}\mL(\Z_{T+1},\w_{T+1})\rangle.
\end{align}
For all $\y\in\mC$ it holds that $(\X^*,\y)\in\mD$ since $\Vert(\X^*,\y)-(\X^*,\y^*)\Vert = \Vert\y-\y^*\Vert_2\le\Vert\y\Vert_2+\Vert\y^*\Vert_2\le3\Vert\y^*\Vert_2$. Therefore, summing the two inequalities in \eqref{ineq:inProof000} and taking the maximum over all  $(\X',\y')\in\mD$ we obtain that in particular for all $\y\in\mC$ it holds that
\begin{align*}
\mL(\Z_{T+1},\y)-\mL(\X^*,\w_{T+1}) & \le \max_{(\X',\y')\in\mD}\left( \langle\Z_{T+1}-\X',\nabla_{\X}\mL(\Z_{T+1},\w_{T+1})\rangle + \langle\w_{T+1}-\y',-\nabla_{\y}\mL(\Z_{T+1},\w_{T+1})\rangle\right)
\\ & \underset{(a)}{\le} \frac{1}{\sqrt{T}}\frac{3\max\lbrace 2R_0(r),3\Vert\y^*\Vert_2\rbrace\Vert(\X_1,\y_1)-(\X^*,\y^*)\Vert}{\eta\sqrt{1-\eta^2\beta_{\mL}^2}}
\\ & \underset{(b)}{\le} \frac{1}{\sqrt{T}}\frac{3\sqrt{2}\max\lbrace 2R_0(r),3\Vert\y^*\Vert_2\rbrace\Vert(\X_1,\y_1)-(\X^*,\y^*)\Vert}{\eta}.
\end{align*}
where (a) follows from Theorem 6 in \cite{lastIterateDecrease} and (b) follows from our choice of step-size which satisfies that $\eta\le 1/(\sqrt{2}\beta_{\mL})$. Invoking \cref{lemma:boundf_and_norm} we obtain the result.

\end{proof}

\section{Theorem and Proof Omitted from \cref{sec:Augmented}}
\label{appx:thmAugmented}

In this section we denote the augmented Lagrangian as $\mL_{\mu}(\X,\y):=f(\X)+\y^{\top}(\b-\mA(\X))+\frac{\mu}{2}\Vert\mA(\X)-\b\Vert^2_2$ for any $\mu\ge0$. We will now state the full theorem and prove it.

\begin{theorem} \label{thm:SPD_aug}
Fix an optimal solution $\X^*\succeq0$ to Problem \eqref{mainProblemSDP}. Let $\y^*\in\reals^m$ be a corresponding dual solution and suppose $(\X^*,\y^*)$ satisfies the complementarity condition \eqref{eq:compcond} with parameter $\tilde{r}:=n-\rank(\nabla_{\X}\mL(\X^*,\y^*))$.
Let $\lbrace(\X_t,\y_t)\rbrace_{t\ge1}$ and $\lbrace(\Z_t,\w_t)\rbrace_{t\ge2}$ be the sequences of iterates generated by the projected extragradient method as in Algorithm \ref{alg:EG} with a fixed step-size
$\eta \le \frac{1}{\sqrt{2}\sqrt{(\beta+\mu\Vert\mA\Vert^2)^2+\Vert\mA\Vert^2}}$. Fix some $r\ge\tilde{r}$ and assume the initialization $(\X_1,\y_1)$ satisfies that $\Vert(\X_1,\y_1)-(\X^*,\y^*)\Vert_F\le R_0(r)$ where
\begin{align*}
& R_0(r):=
\frac{\eta\max\left\lbrace\lambda_{n-r}(\nabla_{\X}\mL(\X^*,\y^*)),\sqrt{\tilde{r}}\lambda_{n-r+\tilde{r}-1}(\nabla_{\X}\mL(\X^*,\y^*))\right\rbrace}{2(1+\sqrt{2}\eta\max\lbrace\beta+\mu\Vert\mA\Vert^2,\Vert\mA\Vert\rbrace)}.
\end{align*}
Then, for all $t\ge1$ the projections $\Pi_{\mbS^n_+}[\cdot]$ in Algorithm \ref{alg:EG} could be replaced with rank-r truncated projections without changing the sequences $\lbrace(\X_t,\y_t)\rbrace_{t\ge1}$ and $\lbrace(\Z_t,\w_t)\rbrace_{t\ge2}$. Moreover, for all $T\ge0$ it holds that
\begin{align*}
& f\left(\frac{1}{T}\sum_{t=1}^T\Z_{k+1}\right)-f(\X^*)\le R^2/(2\eta T),
\\ & \left\Vert\mA\left(\frac{1}{T}\sum_{t=1}^T\Z_{k+1}\right)-\b\right\Vert_2\le R^2/(2(\Vert\y^*\Vert_2+\mu/2)\eta T),
\end{align*}
where $R^2:=\Vert\X_1-\X^*\Vert_F^2+\max_{\y\in\lbrace\y\in\reals^m\ \vert\ \Vert\y\Vert_2\le2\Vert\y^*\Vert_2\rbrace}\Vert\y_1-\y\Vert_2^2$.

Furthermore, if we take a step-size $\eta=\frac{1}{2\sqrt{(\beta+\mu\Vert\mA\Vert^2)^2+\Vert\mA\Vert^2}}$ , then for all $T\ge0$ it holds that
\begin{align*}
& f(\Z_{T+1})-f(\X^*)\le R_1/(\eta\sqrt{T}),
\\ & \Vert\mA(\Z_{T+1})-\b\Vert_2\le R_1/((\Vert\y^*\Vert_2+\mu/2)\eta\sqrt{T}),
\end{align*}
where $R_1:=3\sqrt{2}\max\lbrace 2R_0(r),3\Vert\y^*\Vert_2\rbrace\Vert(\X_1,\y_1)-(\X^*,\y^*)\Vert$.

\end{theorem}

\begin{proof}
The theorem follows from Theorem \ref{thm:SDP} by replacing the function $f$ with $f(\X)+(\mu/2)\Vert\mA(\X)-\b\Vert_2^2$. 
This causes the smoothness parameter $\beta_X$ to change to $\beta_X=\beta + \mu\Vert\mA\Vert^2$ while the rest of the smoothness parameters remain unchanged. 

The one other difference is in inequality \eqref{ineq:func&normBound} in the proof of Lemma \ref{lemma:boundf_and_norm} where there is an additional term of $(\mu/2)\Vert\mA(\X)-\b\Vert_2^2$. This causes $\alpha$ in inequality \eqref{ineq:func&normBound_1} to be replaced with $\alpha+\mu/2$. Choosing $\alpha=2\Vert\y^*\Vert_2$ as we chose in the proof of Lemma \ref{lemma:boundf_and_norm}, we will obtain a that the denominator of the bound on $\Vert\mA(\widehat{\X})-\b\Vert_2$ increases to $\Vert\y^*\Vert_2+\mu/2$ as in the theorem.
\end{proof}

\section{Proofs Omitted from \cref{sec:nonsmoothReg}}

To prove \cref{thm:nonsmoothRegularizer} we first need to prove several helpful lemmas.

The following lemma extends  \cref{lemma:range_null} to the setting of Problem \eqref{problem:nonsmoothRegularizersSP} where there is an additional dual variable $\blambda$. 
\begin{lemma} \label{lemma:range_null_nonsmooth}
Let $(\X^*,\y^*,\blambda^*)$ be a saddle-point of Problem \eqref{problem:nonsmoothRegularizer} and assume Assumptions \ref{Ass:slater} and \ref{Ass:primalBounded} hold. Then, $\nabla_{\X}\mL(\X^*,\y^*,\blambda^*)=\nabla{}f(\X^*)-\mA^{\top}(\y^*)+\nabla_{\X}G(\X^*,\blambda^*)\succeq0$ and $\range(\X^*)\subseteq\nullspace(\nabla_{\X}\mL(\X^*,\y^*,\blambda^*))$.
\end{lemma}

\begin{proof}
Under Assumptions \ref{Ass:slater} and \ref{Ass:primalBounded}, strong duality holds for Problem \eqref{problem:generalSDPprimal}. This implies that for the optimal solution $\X^*$ and some optimal dual solution $(\y^*,\blambda^*,\Z^*)$ the KKT conditions for Problem \eqref{problem:generalSDPprimal} hold.
Writing the KKT conditions we obtain that
\begin{align*}
& 0=\nabla{}f(\X^*)-\mA^{\top}(\y^*)+\nabla_{\X}G(\X^*,\blambda^*)-\Z^*,
\\ & \langle\Z^*,\X^*\rangle=0,
\\ & \mA(\X^*)=\b,
\\ & \X^*,\Z^*\succeq0,
\\ & \blambda^*\in\mK.
\end{align*} 
In particular, since $\Z^*\succeq0$ we have from the first order optimality condition that $\Z^*=\nabla{}f(\X^*)-\mA^{\top}(\y^*)+\nabla_{\X}G(\X^*,\blambda^*)\succeq0$.

In addition, the fact that $\langle\Z^*,\X^*\rangle=0$ and $\X^*,\Z^*\succeq0$ imply that $\Z^*\X^*=0$. This means that $\range(\X^*)\subseteq\nullspace(\Z^*)$. 
\end{proof}

The updates of the projected extragradient method are with respect to the Lagrangian $\mL(\X,\y,\blambda)$. The following lemma shows how a convergence with respect to the Lagrangian implies that the original nonsmooth problem $f(\X)+g(\X)$ converges while also satisfying the linear constraints.
\begin{lemma} \label{lemma:boundf_and_norm_regularized}
Let $\X^*\succeq0$ be an optimal solution to Problem \eqref{problem:nonsmoothRegularizer} and let $(\widehat{\X},(\widehat{\y},\widehat{\blambda}))\in\mathbb{S}^n_+\times(\reals^m\times\mK)$. Let $(\y^*,\blambda^*)\in\argmax_{\y\in\reals^m,\ \blambda\in\mK}\mL(\X^*,\y,\blambda)$.
Denote $\mC=\lbrace\y\in\reals^m\ \vert\ \Vert\y\Vert_2\le2\Vert\y^*\Vert_2\rbrace$. Assume that for all $(\y,\blambda)\in\mC\times\mK$ 
\begin{align*}
\mL(\widehat{\X},\y,\blambda)-\mL(\X^*,\widehat{\y},\widehat{\blambda})\le\varepsilon.
\end{align*}
Then it holds that
$f(\widehat{\X})+g(\widehat{\X})-f(\X^*)-g(\X^*)\le\varepsilon$ and $\Vert\mA(\widehat{\X})-\b\Vert_2\le\varepsilon/\Vert\y^*\Vert_2$.
\end{lemma}

\begin{proof}
For all $(\y,\blambda)\in\mC\times\mK$  it holds that
$\mL(\widehat{\X},\y,\blambda)-\mL(\X^*,\widehat{\y},\widehat{\blambda})\le\varepsilon$.
In particular, denoting $\alpha:=2\Vert\y^*\Vert$ and plugging in $\y=\frac{\alpha(\b-\mA(\widehat{\X}))}{\Vert\mA(\widehat{\X})-\b\Vert_2}\in\mC$ and $\blambda\in\argmax_{\blambda\in\mK}G(\widehat{\X},\blambda)$, we obtain
\begin{align} \label{ineq:func&normBound_regularized}
& f(\widehat{\X})+G(\widehat{\X},\blambda)+\y^{\top}(\b-\mA(\widehat{\X}))-f(\X^*)-G(\X^*,\widehat{\blambda}) \nonumber
\\ & = f(\widehat{\X})+g(\widehat{\X})-f(\X^*)-G(\X^*,\widehat{\blambda})+\alpha\Vert\mA(\widehat{\X})-\b\Vert_2^2\le\varepsilon.
\end{align}

By the relationship between $G$ and $g$ we have that
\begin{align*}
g(\X^*)=\max_{\blambda\in\mK}G(\X^*,\blambda) \ge G(\X^*,\widehat{\blambda}).
\end{align*}
Plugging this into \eqref{ineq:func&normBound_regularized} we obtain that
\begin{align*}
f(\widehat{\X})+g(\widehat{\X})-f(\X^*)-g(\X^*)+\alpha\Vert\mA(\widehat{\X})-\b\Vert_2^2\le\varepsilon.
\end{align*}

Since $\Vert\mA(\widehat{\X})-\b\Vert_2\ge0$ it follows immediately that $f(\widehat{\X})+g(\widehat{\X})-f(\X^*)-g(\X^*)\le\varepsilon$. 
%

In addition, since $(\X^*,\y^*,\blambda^*)$ is a saddle-point it holds that 
\begin{align*}
f(\X^*)+g(\X^*) & = \mL(\X^*,\y^*,\blambda^*) \le \mL(\widehat{\X},\y^*,\blambda^*) 
\\ & = f(\widehat{\X}) + {\y^*}^{\top}(\b-\mA(\widehat{\X}))+G(\widehat{\X},\blambda^*)
\\ & \le f(\widehat{\X})+ {\y^*}^{\top}(\b-\mA(\widehat{\X}))+ g(\widehat{\X}).
\end{align*}
Plugging this into \eqref{ineq:func&normBound_regularized}, we obtain that
\begin{align*}
\alpha\Vert\mA(\widehat{\X})-\b\Vert_2 \le \varepsilon+{\y^*}^{\top}(\b-\mA(\widehat{\X})) \le \varepsilon+\Vert\y^*\Vert_2\Vert\mA(\widehat{\X})-\b\Vert_2. 
\end{align*}
Rearranging, we obtain the desired result of $\Vert\mA(\widehat{\X})-\b\Vert_2\le\varepsilon/(\alpha-\Vert\y^*\Vert_2)=\varepsilon/\Vert\y^*\Vert_2$.
\end{proof}

\subsection{Proof of \cref{thm:nonsmoothRegularizer}}
\label{appx:proofThm45}

We first restate the theorem and then prove it. 

\begin{theorem}
Fix an optimal solution $\X^*\succeq0$ to Problem \eqref{problem:nonsmoothRegularizer}. Let $(\y^*,\blambda^*)\in\reals^m\times\mK$ be a corresponding dual solution and suppose $(\X^*,(\y^*,\blambda^*))$ satisfies the complementarity condition \eqref{eq:compcond:NS} with parameter  $\tilde{r}:=n-\rank(\nabla_{\X}\mL(\X^*,\y^*,\blambda^*))$.
Let $\lbrace(\X_t,(\y_t,\blambda_t))\rbrace_{t\ge1}$ and $\lbrace(\Z_t,(\w_t,\bmu_t))\rbrace_{t\ge2}$ be the sequences of iterates generated by  \cref{alg:EGregularized} with a fixed step-size
\begin{align*}
\eta \le \min\left\lbrace\begin{array}{l}\frac{1}{\sqrt{2}\sqrt{(\beta+\rho_{X})^2+\Vert\mA\Vert^2+\rho_{\lambda X}^2}},\frac{1}{\sqrt{2}\sqrt{\rho_{\lambda}^2+2\max\lbrace \Vert\mA\Vert^2,\rho_{X\lambda}^2\rbrace}},
\\ \frac{1}{\beta+\rho_{X}+\sqrt{2}\max\lbrace \Vert\mA\Vert,\rho_{X\lambda}\rbrace},\frac{1}{\rho_{\lambda}+\sqrt{\Vert\mA\Vert^2+\rho_{\lambda X}}}\end{array}\right\rbrace.
\end{align*}
Fix some $r\ge\tilde{r}$ and assume the initialization $(\X_1,(\y_1,\blambda_1))$ satisfies that $\Vert(\X_1,\y_1,\blambda_1)-(\X^*,\y^*,\blambda^*)\Vert_F\le R_0(r)$ where
\begin{align*}
& R_0(r):= \frac{\eta\max\left\lbrace\lambda_{n-r}(\nabla_{\X}\mL(\X^*,\y^*,\blambda^*)),\sqrt{\tilde{r}}\lambda_{n-r+\tilde{r}-1}(\nabla_{\X}\mL(\X^*,\y^*,\blambda^*))\right\rbrace}{2(1+2\eta\max\lbrace\beta+\rho_{X},\Vert\mA\Vert,\rho_{X\lambda}\rbrace)}.
\end{align*}
Then, for all $t\ge1$ the projections $\Pi_{\mbS^n_+}[\cdot]$ in \cref{alg:EGregularized} could be replaced with rank-r truncated projections without changing the sequences $\lbrace(\X_t,(\y_t,\blambda_t))\rbrace_{t\ge1}$ and $\lbrace(\Z_t,(\w_t,\bmu_t))\rbrace_{t\ge2}$. Moreover, and for all $T\ge0$ it holds that
\begin{align*}
& f\left(\frac{1}{T}\sum_{t=1}^T\Z_{k+1}\right)+g\left(\frac{1}{T}\sum_{t=1}^T\Z_{k+1}\right)-f(\X^*)-g(\X^*) 
\le R^2/(2\eta T),
\\ & \left\Vert\mA\left(\frac{1}{T}\sum_{t=1}^T\Z_{k+1}\right)-\b\right\Vert_2  \le R^2/(2\Vert\y^*\Vert_2\eta T),
\end{align*}
where $R^2=\Vert\X_1-\X^*\Vert_F^2+\max_{\y\in\lbrace\y\in\reals^m\ \vert\ \Vert\y\Vert_2\le2\Vert\y^*\Vert_2\rbrace}\Vert\y_1-\y\Vert_2^2+\max_{\blambda\in\mK}\Vert\blambda_1-\blambda\Vert_2^2$.

Furthermore, if we take a step-size
\begin{align*}
\eta = \min\left\lbrace\begin{array}{l}\frac{1}{2\sqrt{(\beta+\rho_{X})^2+\Vert\mA\Vert^2+\rho_{\lambda X}^2}},
\frac{1}{2\sqrt{\rho_{\lambda}^2+2\max\lbrace \Vert\mA\Vert^2,\rho_{X\lambda}^2\rbrace}},
\\ \frac{1}{\beta+\rho_{X}+\sqrt{2}\max\lbrace \Vert\mA\Vert,\rho_{X\lambda}\rbrace},\frac{1}{\rho_{\lambda}+\sqrt{\Vert\mA\Vert^2+\rho_{\lambda X}}}\end{array}\right\rbrace,
\end{align*}
then for all $T\ge0$ it holds that
\begin{align*}
& f(\Z_{T+1})+g(\Z_{t+1})-f(\X^*)-g(\X^*)\le R_1/(\eta\sqrt{T}),
\\ & \Vert\mA(\Z_{T+1})-\b\Vert_2\le R_1/(\Vert\y^*\Vert_2\eta\sqrt{T}),
\end{align*}
where
$R_1:=3\sqrt{2} \max\lbrace 2R_0(r),3\sqrt{2}\Vert\y^*\Vert_2,\sqrt{2}\cdot\diam(\mK)\rbrace
\Vert(\X_1,\y_1,\blambda_1)-(\X^*,\y^*,\blambda^*)\Vert$.

%
\end{theorem}

\begin{proof}
Using the inequalities of \eqref{ineq:betasNonsmoothRegGpart} we can calculate the smoothness constants of $\mL(\X,\y,\blambda)=f(\X) +G(\X,\blambda)+\y^{\top}(\b-\mA(\X))$ as defined in \eqref{ineq:betas} with respect to the variables $\X$ and $(\y,\blambda)$. Here  $\nabla_{\X}\mL=\frac{\partial \mL}{\partial \X}$ and $\nabla_{(\y,\blambda)}\mL=(\frac{\partial \mL}{\partial \y},\frac{\partial \mL}{\partial \blambda})$. We have that
\begin{align*}
\Vert\nabla_{\X}\mL(\X,\y,\blambda)-\nabla_{\X}\mL(\widetilde{\X},\y,\blambda)\Vert_{F} 
& = \Vert\nabla f(\X)-\nabla f(\widetilde{\X})+\nabla_{\X}G(\X,\blambda)-\nabla_{\X}G(\widetilde{\X},\blambda)\Vert_F
\\ &  \le(\beta+\rho_{X})\Vert\X-\widetilde{\X}\Vert_{F}, \\
\Vert\nabla_{(\y,\blambda)}\mL(\X,\y,\blambda)-\nabla_{(\y,\blambda)}\mL(\X,\widetilde{\y},\widetilde{\blambda})\Vert_{2} 
& = \Vert\nabla_{\blambda}G(\X,\blambda)-\nabla_{\blambda}G(\X,\widetilde{\blambda})\Vert_2 \le\rho_{\lambda}\Vert\blambda-\widetilde{\blambda}\Vert_{2}
\le\rho_{\lambda}\Vert(\y,\blambda)-(\widetilde{\y},\widetilde{\blambda})\Vert_{2}, \\
 \Vert\nabla_{\X}\mL(\X,\y,\blambda)-\nabla_{\X}\mL(\X,\widetilde{\y},\widetilde{\blambda})\Vert_{F}
& =\Vert\mA^{\top}(\y-\widetilde{\y})+\nabla_{\X}G(\X,\blambda)-\nabla_{\X}G(\X,\widetilde{\blambda})\Vert_{F}
\\ & \le \Vert\mA\Vert\Vert\y-\widetilde{\y}\Vert_2+\rho_{X\lambda}\Vert\blambda-\widetilde{\blambda}\Vert_{2}
 \le \sqrt{2}\max\lbrace \Vert\mA\Vert,\rho_{X\lambda}\rbrace\Vert(\y,\blambda)-(\widetilde{\y},\widetilde{\blambda})\Vert_{2},\\
\Vert\nabla_{(\y,\blambda)}\mL(\X,\y,\blambda)-\nabla_{(\y,\blambda)}\mL(\widetilde{\X},\y,\blambda)\Vert_{2} 
& = \left\Vert\left(\begin{array}{c}\mA(\X-\widetilde{\X})
\\ \nabla_{\blambda}G(\X,\blambda)-\nabla_{\blambda}G(\widetilde{\X},\blambda)\end{array}\right)\right\Vert_{2}  
\\ & = \sqrt{\Vert\mA(\X-\widetilde{\X})\Vert_2^2+\Vert\nabla_{\blambda}G(\X,\blambda)-\nabla_{\blambda}G(\widetilde{\X},\blambda)\Vert_2^2}
\\ & \le\sqrt{\Vert\mA\Vert^2+\rho_{\lambda X}^2}\Vert\X-\widetilde{\X}\Vert_{F},
\end{align*}
and so, we obtain that $\beta_X = \beta+\rho_{X}$, $\beta_{(y,\lambda)} = \rho_{\lambda}$, $\beta_{X(y,\lambda)} =  \sqrt{2}\max\lbrace \Vert\mA\Vert,\rho_{X\lambda}\rbrace$, $\beta_{(y,\lambda)X} = \sqrt{\Vert\mA\Vert^2+\rho_{\lambda X}^2}$, and  $\beta_{\mL}=\sqrt{2}\max\left\lbrace\sqrt{(\beta+\rho_{X})^2+\Vert\mA\Vert^2+\rho_{\lambda X}^2},\sqrt{\rho_{\lambda}^2+2\max\lbrace\Vert\mA\Vert^2,\rho_{X\lambda}^2\rbrace}\right\rbrace$.

Using the smoothness constants of $f$ and $G$ we also have that
\begin{align} 
& \Vert\nabla_{\X}\mL(\Z,\w,\bmu)-\nabla_{\X}\mL(\X^*,\y^*,\blambda^*)\Vert_F \nonumber
\\ & \le \Vert\nabla_{\X}\mL(\Z,\w,\bmu)-\nabla_{\X}\mL(\X^*,\w,\bmu)\Vert_F \nonumber
 + \Vert\nabla_{\X}\mL(\X^*,\w,\bmu)-\nabla_{\X}\mL(\X^*,\y^*,\bmu)\Vert_F \nonumber
\\ & \ \ \ + \Vert\nabla_{\X}\mL(\X^*,\y^*,\bmu)-\nabla_{\X}\mL(\X^*,\y^*,\blambda^*)\Vert_F \nonumber
\\ & = \Vert\nabla{}f(\Z)-\nabla{}f(\X^*)+\nabla_{\X}G(\Z,\bmu)-\nabla_{\X}G(\X^*,\bmu)\Vert_F + \Vert\mA^{\top}(\w-\y^*))\Vert_F 
+ \Vert\nabla_{\blambda}G(\X^*,\bmu)-\nabla_{\blambda}G(\X^*,\blambda^*)\Vert_F \nonumber
\\ & \underset{(a)}{\le} (\beta+\rho_X)\Vert\Z-\X^*\Vert_F+\Vert\mA\Vert\Vert\w-\y^*\Vert_2+\rho_{X\lambda}\Vert\bmu-\blambda^*\Vert_2 \nonumber
\\ & \le \max\lbrace\beta+\rho_X,\Vert\mA\Vert,\rho_{X\lambda}\rbrace\left(\Vert\Z-\X^*\Vert_F+\Vert\w-\y^*\Vert_2+\Vert\bmu-\blambda^*\Vert_2\right) \nonumber
\\ & \le \sqrt{2}\max\lbrace\beta+\rho_X,\Vert\mA\Vert,\rho_{X\lambda}\rbrace\left(\Vert(\Z,\w)-(\X^*,\y^*)\Vert+\Vert\bmu-\blambda^*\Vert_2\right)  \nonumber
\\ & \le 2\max\lbrace\beta+\rho_X,\Vert\mA\Vert,\rho_{X\lambda}\rbrace\Vert(\Z,\w,\bmu)-(\X^*,\y^*,\blambda^*)\Vert. \label{ineq:inProofSmoothnessXBound_2}
\end{align} 
Hence, we can replace \eqref{ineq:inProofSmoothnessXBound_1} with \eqref{ineq:inProofSmoothnessXBound_2} in the proof of \cref{lemma:radius}, and we obtain that for any $r\ge\tilde{r}$, $\eta\ge0$, and $\X\in\mbS^n,(\Z,\w,\bmu)\in\mathbb{S}^n_+\times\reals^m\times\mK$, if
\begin{align} \label{ineq:radiusNonsmooth}
& \max\lbrace\Vert\X-\X^*\Vert_F,\Vert(\Z,\w,\bmu)-(\X^*,\y^*,\blambda^*)\Vert\rbrace \nonumber
\\ & \le \frac{\eta\max\left\lbrace\lambda_{n-r}(\nabla_{\X}\mL(\X^*,\y^*,\blambda^*)),\sqrt{\tilde{r}}\lambda_{n-r+\tilde{r}-1}(\nabla_{\X}\mL(\X^*,\y^*,\blambda^*))\right\rbrace}{1+2\eta\max\lbrace\beta+\rho_X,\Vert\mA\Vert,\rho_{X\lambda}\rbrace}
\end{align}
then $\rank(\Pi_{\mathbb{S}^n_+}[\X-\eta\nabla{}\mL(\Z,\w,\bmu)])\le r$.

Replacing the $\y$ variable in Theorem \ref{thm:SDP} with the product $(\y,\blambda)\in\reals^m\times\mK$, replacing the function $f(\X)$ with $f(\X)+\max_{\blambda\in\mK}G(\X,\blambda)$, setting the correct smoothness parameters in the step size, and adjusting the bound $R_0(r)$ so as to fulfil the requirement of \eqref{ineq:radiusNonsmooth}, we have from Theorem \ref{thm:SDP} that all projections can be replaced with rank-$r$ truncated  projections while maintaining the original convergence rate of the extragradient method.
Denoting $\mC=\lbrace\y\in\reals^m\ \vert\ \Vert\y\Vert_2\le2\Vert\y^*\Vert_2\rbrace$, we also have that inequality \eqref{ineq:inProof1111} can be replaced with 
the following inequality for all $\y\in\mC$ and $\blambda\in\mK$: 
\begin{align} \label{ineq:inProof222}
& \frac{1}{T}\sum_{t=1}^T \mL(\Z_{t+1},\y,\blambda) - \frac{1}{T}\sum_{t=1}^T \mL(\X^*,\w_{t+1},\bmu_{t+1}) 
\le \frac{1}{2\eta T}\left(\Vert\X_1-\X^*\Vert_F^2+\max_{\y\in\mC}\Vert\y_1-\y\Vert_2^2+\max_{\blambda\in\mK}\Vert\blambda_1-\blambda\Vert_2^2\right).
\end{align}
The second term in the LHS of \eqref{ineq:inProof222} can be bounded by
\begin{align} \label{ineq:inProof221}
\frac{1}{T}\sum_{t=1}^T \mL(\X^*,\w_{t+1},\bmu_{t+1}) \nonumber
& = f(\X^*)+\frac{1}{T}\sum_{t=1}^TG(\X^*,\bmu_{t+1})+\frac{1}{T}\sum_{t=1}^T\w_{t+1}^{\top}(\b-\mA(\X^*)) \nonumber 
\\ & \underset{(a)}{\le} f(\X^*)+G\left(\X^*,\frac{1}{T}\sum_{t=1}^T\bmu_{t+1}\right)+\frac{1}{T}\sum_{t=1}^T\w_{t+1}^{\top}(\b-\mA(\X^*)) \nonumber
\\ & = \mL\left(\X^*,\frac{1}{T}\sum_{t=1}^T\w_{t+1},\frac{1}{T}\sum_{t=1}^T\bmu_{t+1}\right),
\end{align}
where (a) follows from the concavity of $G(\X^*,\cdot)$. Plugging \eqref{ineq:inProof221} into \eqref{ineq:inProof222} and using 
the convexity of $\mL(\cdot,\y,\blambda)$, we have in particular that for all $\y\in\mC$ and $\blambda\in\mK$ it holds that
\begin{align*}
& \mL\left(\frac{1}{T}\sum_{t=1}^T\Z_{t+1},\y,\blambda\right) - \mL\left(\X^*,\frac{1}{T}\sum_{t=1}^T\w_{t+1},\frac{1}{T}\sum_{t=1}^T\bmu_{t+1}\right) \nonumber
\\ & \le \frac{1}{2\eta T}\left(\Vert\X_1-\X^*\Vert_F^2+\max_{\y\in\mC}\Vert\y_1-\y\Vert_2^2+\max_{\blambda\in\mK}\Vert\blambda_1-\blambda\Vert_2^2\right).
\end{align*}
Invoking Lemma \ref{lemma:boundf_and_norm_regularized} we obtain the bounds on $f\left(\frac{1}{T}\sum_{t=1}^T\Z_{t+1}\right)+g\left(\frac{1}{T}\sum_{t=1}^T\Z_{t+1}\right)-f(\X^*)-g(\X^*)$ and $\left\Vert\mA\left(\frac{1}{T}\sum_{t=1}^T\Z_{t+1}\right)-\b\right\Vert_2$ in the statement of the theorem.

For the second part of the theorem, 
denote $$\mD=\left\lbrace(\X,\y,\blambda)\in\mathbb{S}^n_+\times\reals^m\times\mK~|~\Vert(\X,\y,\blambda)-(\X^*,\y^*,\blambda^*)\Vert\le \max\lbrace 2R_0(r),3\sqrt{2}\Vert\y^*\Vert_2,\sqrt{2}\cdot\diam(\mK)\rbrace\right\rbrace.$$ 
As we showed above, all iterates are within the ball $\mD$. By the convexity of $\mL(\cdot,\w_{T+1},\bmu_{T+1})$ and concavity of $\mL(\Z_{T+1},\cdot,\cdot)$, for all $(\X,\y,\blambda)\in\mD$
it holds that
\begin{align} \label{ineq:inProof0001}
\mL(\Z_{T+1},\w_{T+1},\bmu_{T+1})-\mL(\X,\w_{T+1},\bmu_{T+1}) & \le \langle\Z_{T+1}-\X,\nabla_{\X}\mL(\Z_{T+1},\w_{T+1},\bmu_{T+1})\rangle \nonumber
\\ \mL(\Z_{T+1},\y,\blambda)-\mL(\Z_{T+1},\w_{T+1},\bmu_{T+1}) & \le \langle\w_{T+1}-\y,-\nabla_{\y}\mL(\Z_{T+1},\w_{T+1},\bmu_{T+1})\rangle \nonumber
\\ & \ \ \ +\langle\bmu_{T+1}-\blambda,-\nabla_{\blambda}\mL(\Z_{T+1},\w_{T+1},\bmu_{T+1})\rangle.
\end{align}
For all $\y\in\mC$ and $\blambda\in\mK$ it holds that $(\X^*,\y,\blambda)\in\mD$ since 
\begin{align*}
\Vert(\X^*,\y,\blambda)-(\X^*,\y^*,\blambda^*)\Vert & = \Vert(\y,\blambda)-(\y^*,\blambda^*)\Vert_2=\sqrt{\Vert\y-\y^*\Vert_2^2+\Vert\blambda-\blambda^*\Vert_2^2}
\\ & \le\sqrt{(\Vert\y\Vert_2+\Vert\y^*\Vert_2)^2+\Vert\blambda-\blambda^*\Vert_2^2}\le\sqrt{(3\Vert\y^*\Vert_2)^2+\Vert\blambda-\blambda^*\Vert_2^2}
\\ & \le\sqrt{2}\max\lbrace 3\Vert\y^*\Vert_2,\diam(\mK)\rbrace.
\end{align*}
Therefore, summing the two inequalities in \eqref{ineq:inProof0001} and taking the maximum over all  $(\X',\y',\blambda')\in\mD$ we obtain that in particular for all $\y\in\mC$ and $\blambda\in\mK$ it holds that
\begin{align*}
& \mL(\Z_{T+1},\y,\blambda)-\mL(\X^*,\w_{T+1},\bmu_{T+1})
\\ & \le \max_{(\X',\y',\blambda')\in\mD}\left(\begin{array}{l} \langle\Z_{T+1}-\X',\nabla_{\X}\mL(\Z_{T+1},\w_{T+1},\bmu_{T+1})\rangle + \langle\w_{T+1}-\y',-\nabla_{\y}\mL(\Z_{T+1},\w_{T+1},\bmu_{T+1})\rangle
\\ + \langle\bmu_{T+1}-\blambda',-\nabla_{\blambda}\mL(\Z_{T+1},\w_{T+1},\bmu_{T+1})\rangle\end{array}\right)
\\ & \underset{(a)}{\le} \frac{1}{\sqrt{T}}\frac{3 \max\lbrace 2R_0(r),3\sqrt{2}\Vert\y^*\Vert_2,\sqrt{2}\cdot\diam(\mK)\rbrace\Vert(\X_1,\y_1,\blambda_1)-(\X^*,\y^*,\blambda^*)\Vert}{\eta\sqrt{1-\eta^2\beta_{\mL}^2}}
\\ & \underset{(b)}{\le} \frac{1}{\sqrt{T}}\frac{3\sqrt{2} \max\lbrace 2R_0(r),3\sqrt{2}\Vert\y^*\Vert_2,\sqrt{2}\cdot\diam(\mK)\rbrace\Vert(\X_1,\y_1,\blambda_1)-(\X^*,\y^*,\blambda^*)\Vert}{\eta}.
\end{align*}
where (a) follows from Theorem 6 in \cite{lastIterateDecrease} and (b) follows from our choice of step-size which satisfies that $\eta\le 1/(\sqrt{2}\beta_{\mL})$. Invoking \cref{lemma:boundf_and_norm_regularized} we obtain the result.
\end{proof}

\section{Proofs omitted from \cref{sec:nonlinearInequalities}}

We first present and prove two lemmas that are necessary for proving \cref{thm:inequlalityConstraints}. 

The following lemma is analogous to \cref{lemma:radius}. In the setting of Problem \eqref{problem:ineqConstrained} since the smoothness of $\mL$ is not bounded, we need to consider points within a bounded set.
\begin{lemma} \label{lemma:radius_inequalites}
Let $(\X^*,\y^*,\blambda^*)$ be a saddle-point of Problem \eqref{problem:ineqConstrained} and denote $\tilde{r}=n-\rank(\nabla_{\X}\mL(\X^*,\y^*,\blambda^*))$. 
Set some $r\ge\tilde{r}$ and denote $\delta(r):=\max\left\lbrace\lambda_{n-r}(\nabla_{\X}\mL(\X^*,\y^*,\blambda^*)),\sqrt{\tilde{r}}\lambda_{n-r+\tilde{r}-1}(\nabla_{\X}\mL(\X^*,\y^*,\blambda^*))\right\rbrace.$
Denote $\mD_X=\lbrace\X\succeq0~\vert~\Vert\X-\X^*\Vert_F\le2\delta(r)\rbrace$ and $\mD_{\lambda}=\lbrace\blambda\in\reals^d_+~\vert~\Vert\blambda-\blambda^*\Vert_2\le2\delta(r)\rbrace$.
Then, for any $\eta\ge0$ and $\X\in\mbS^n,(\Z,\w,\bmu)\in\mD_X\times\reals^m\times\mD_{\lambda}$, if
\begin{align*}
& \max\lbrace\Vert\X-\X^*\Vert_F,\Vert(\Z,\w,\bmu)-(\X^*,\y^*,\blambda^*)\Vert\rbrace
\\ & \le \frac{\eta\cdot\max\left\lbrace\lambda_{n-r}(\nabla_{\X}\mL(\X^*,\y^*,\blambda^*)),\sqrt{\tilde{r}}\lambda_{n-r+\tilde{r}-1}(\nabla_{\X}\mL(\X^*,\y^*,\blambda^*))\right\rbrace}{1+2\eta\max\lbrace\beta+\sup_{\blambda\in\mD_{\lambda}}\Vert\blambda\Vert_1\beta_g,\Vert\mA\Vert,\sqrt{d}\sup_{\X\in\mD_X}\max_{i\in[d]}\Vert\nabla{}g_i(\X)\Vert_F\rbrace}
\end{align*}
then $\rank\left(\Pi_{\mathbb{S}^n_+}[\X-\eta\nabla{}\mL(\Z,\w,\bmu)]\right)\le r$.
\end{lemma}

\begin{proof}
Note that the addition of the conditions  $\blambda^{\top} g(\X^*)=0$ and $\blambda^*\ge0$ to the KKT conditions in the proof of Lemma \ref{lemma:range_null} do not change the proof, and so we have that Lemma \ref{lemma:range_null} holds for this setting with $\nabla_{\X}\mL(\X^*,\y^*,\blambda^*)=\nabla{}f(\X^*)-\mA^{\top}(\y^*)+\sum_{i=1}^d\blambda^*_i\nabla{}g_i(\X^*)\succeq0$. 
Therefore, the proof follows using the same arguments as the proof of Lemma \ref{lemma:radius} where we replace $\nabla_{\X}\mL(\X^*,\y^*)$  with $\nabla_{\X}\mL(\X^*,\y^*,\blambda^*)$ and we denote $\P:=\X-\eta\nabla{}\mL(\Z,\w,\bmu)$ and $\P^*:=\X^*-\eta\nabla{}\mL(\X^*,\y^*,\blambda^*)$.  
The only difference is that we replace inequality \eqref{ineq:inProofSmoothnessXBound_1} with the following bound.
\begin{align} 
& \Vert\nabla_{\X}\mL(\Z,\w,\bmu)-\nabla_{\X}\mL(\X^*,\y^*,\blambda^*)\Vert_F \nonumber
\\ & \le \Vert\nabla_{\X}\mL(\Z,\w,\bmu)-\nabla_{\X}\mL(\X^*,\w,\bmu)\Vert_F + \Vert\nabla_{\X}\mL(\X^*,\w,\bmu)-\nabla_{\X}\mL(\X^*,\y^*,\bmu)\Vert_F \nonumber
\\ & \ \ \ + \Vert\nabla_{\X}\mL(\X^*,\y^*,\bmu)-\nabla_{\X}\mL(\X^*,\y^*,\blambda^*)\Vert_F \nonumber
\\ & = \Vert\nabla{}f(\Z)-\nabla{}f(\X^*)+\sum_{i=1}^{d}\blambda_i^*(\nabla{}g_i(\Z)-\nabla{}g_i(\X^*))\Vert_F  + \Vert\mA^{\top}(\w-\y^*))\Vert_F+ \Vert\sum_{i=1}^{d}(\bmu_i-\blambda_i^*)\nabla{}g_i(\X^*)\Vert_F \nonumber
\\ & \underset{(a)}{\le} (\beta+\Vert\bmu\Vert_1\beta_g)\Vert\Z-\X^*\Vert_F+\Vert\mA\Vert\Vert\w-\y^*\Vert_2 +\max_{i\in[d]}\Vert\nabla{}g_i(\X^*)\Vert_F\Vert\bmu-\blambda^*\Vert_1 \nonumber
\\ & \le (\beta+\Vert\bmu\Vert_1\beta_g)\Vert\Z-\X^*\Vert_F+\Vert\mA\Vert\Vert\w-\y^*\Vert_2 +\sqrt{d}\max_{i\in[d]}\Vert\nabla{}g_i(\X^*)\Vert_F\Vert\bmu-\blambda^*\Vert_2 \nonumber
\\ & \le \max\lbrace\beta+\sup_{\blambda\in\mD_{\lambda}}\Vert\blambda\Vert_1\beta_g,\Vert\mA\Vert,\sqrt{d}\sup_{\X\in\mD_X}\max_{i\in[d]}\Vert\nabla{}g_i(\X)\Vert_F\rbrace\left(\Vert\Z-\X^*\Vert_F+\Vert\w-\y^*\Vert_2+\Vert\bmu-\blambda^*\Vert_2\right) \nonumber
\\ & \le \sqrt{2}\max\lbrace\beta+\sup_{\blambda\in\mD_{\lambda}}\Vert\blambda\Vert_1\beta_g,\Vert\mA\Vert,\sqrt{d}\sup_{\X\in\mD_X}\max_{i\in[d]}\Vert\nabla{}g_i(\X)\Vert_F\rbrace\left(\Vert(\Z,\w)-(\X^*,\y^*)\Vert+\Vert\bmu-\blambda^*\Vert_2\right)  \nonumber
\\ & \le 2\max\lbrace\beta+\sup_{\blambda\in\mD_{\lambda}}\Vert\blambda\Vert_1\beta_g,\Vert\mA\Vert,\sqrt{d}\sup_{\X\in\mD_X}\max_{i\in[d]}\Vert\nabla{}g_i(\X)\Vert_F\rbrace \Vert(\Z,\w,\bmu)-(\X^*,\y^*,\blambda^*)\Vert
\end{align}
where (a) follows from the smoothness of $\mL$.

\end{proof}

The following lemma shows the relationship between a convergence of the Lagrangian and a convergence of $f$ to a solution that approximately satisfies the linear constraints and nonlinear inequalities.  
\begin{lemma} \label{lemma:boundf_and_norm_ineqConstrained}
Let $\X^*\succeq0$ be an optimal solution to Problem \eqref{problem:ineqConstrained} and let $(\widehat{\X},(\widehat{\y},\widehat{\blambda}))\in\mathbb{S}^n_+\times(\reals^m\times\reals^d_+)$. Let $(\y^*,\blambda^*)\in\argmax_{\y\in\reals^m,\ \blambda\in\reals^d_+}\mL(\X^*,\y,\blambda)$.
Denote $\mC_y=\lbrace\y\in\reals^m\ \vert\ \Vert\y\Vert_2\le2\Vert\y^*\Vert_2\rbrace$ and $\mC_{\lambda}=\lbrace\blambda\in\reals^d_+\ \vert\ \Vert\blambda\Vert_2\le\Vert\blambda^*\Vert_2+\sqrt{d}\rbrace$. Assume that for all $(\y,\blambda)\in\mC_y\times\mC_{\lambda}$ 
\begin{align*}
\mL(\widehat{\X},\y,\blambda)-\mL(\X^*,\widehat{\y},\widehat{\blambda})\le\varepsilon.
\end{align*}
Then it holds that
$f(\widehat{\X})-f(\X^*)\le\varepsilon$, $\Vert\mA(\widehat{\X})-\b\Vert_2\le\varepsilon/\Vert\y^*\Vert_2$, and $\sum_{i\in\lbrace i\in[d]~\vert~g_i(\widehat{\X})>0\rbrace}g_i(\widehat{\X})\le\varepsilon$.
\end{lemma}

\begin{proof}
For all $(\y,\blambda)\in\mC_y\times\mC_{\lambda}$  it holds that
\begin{align}  \label{ineq:func&normBound_regularized_general}
& \mL(\widehat{\X},\y,\blambda)-\mL(\X^*,\widehat{\y},\widehat{\blambda})=f(\widehat{\X})+\blambda^{\top}g(\widehat{\X})+\y^{\top}(\b-\mA(\widehat{\X}))-f(\X^*)-\widehat{\blambda}^{\top}g(\X^*)\le\varepsilon.
\end{align}
In particular, denoting $\alpha:=2\Vert\y^*\Vert$ and plugging $\y=\frac{\alpha(\b-\mA(\widehat{\X}))}{\Vert\mA(\widehat{\X})-\b\Vert_2}\in\mC_y$ and $\blambda=0\in\mC_{\lambda}$ into \eqref{ineq:func&normBound_regularized_general}, we obtain
\begin{align*}
f(\widehat{\X})-f(\X^*)-\widehat{\blambda}^{\top}g(\X^*)+\alpha\Vert\mA(\widehat{\X})-\b\Vert_2^2\le\varepsilon.
\end{align*}
Since $\Vert\mA(\widehat{\X})-\b\Vert_2\ge0$ and $-\widehat{\blambda}^{\top}g(\X^*)\ge0$ it follows immediately that $f(\widehat{\X})-f(\X^*)\le\varepsilon$. 

Alternatively, plugging $\y=\frac{\alpha(\b-\mA(\widehat{\X}))}{\Vert\mA(\widehat{\X})-\b\Vert_2}\in\mC_y$ and $\blambda=\blambda^*\in\mC_{\lambda}$ into \eqref{ineq:func&normBound_regularized_general}, we obtain that
\begin{align} \label{ineq:func&normBound_regularized_1}
f(\widehat{\X})+{\blambda^*}^{\top}g(\widehat{\X})-f(\X^*)-\widehat{\blambda}^{\top}g(\X^*)+\alpha\Vert\mA(\widehat{\X})-\b\Vert_2^2\le\varepsilon.
\end{align}

In addition, since $(\X^*,\y^*,\blambda^*)$ is a saddle-point it holds that 
$\mL(\X^*,\y^*,\widehat{\blambda}) \le \mL(\X^*,\y^*,\blambda^*) \le \mL(\widehat{\X},\y^*,\blambda^*)$.
Therefore,
\begin{align} \label{ineq:func&normBound_regularized_3}
f(\X^*)+\widehat{\blambda}^{\top}g(\X^*) \le  f(\widehat{\X}) + {\y^*}^{\top}(\b-\mA(\widehat{\X}))+{\blambda^*}^{\top}g(\widehat{\X}).
\end{align}
Plugging \eqref{ineq:func&normBound_regularized_3} into \eqref{ineq:func&normBound_regularized_1}, we obtain that
\begin{align*}
\alpha\Vert\mA(\widehat{\X})-\b\Vert_2 \le \varepsilon+{\y^*}^{\top}(\b-\mA(\widehat{\X})) \le \varepsilon+\Vert\y^*\Vert_2\Vert\mA(\widehat{\X})-\b\Vert_2. 
\end{align*}
Rearranging, we obtain the desired result of $\Vert\mA(\widehat{\X})-\b\Vert_2\le\varepsilon/(\alpha-\Vert\y^*\Vert_2)=\varepsilon/\Vert\y^*\Vert_2$.
Finally, plugging $\y=\y^*\in\mC_y$ and $\blambda=\widetilde{\blambda}\in\mC_{\lambda}$ such that $$\widetilde{\blambda}_i=\bigg\lbrace\begin{array}{ll} \blambda^*_i, & \textrm{if } g_i(\widehat{\X})\le0
\\ \blambda^*_i+1,  & \textrm{if } g_i(\widehat{\X})>0\end{array}$$ into \eqref{ineq:func&normBound_regularized_general} we have that
\begin{align} \label{ineq:func&normBound_regularized_2}
f(\widehat{\X})+{\widetilde{\blambda}}^{\top}g(\widehat{\X})+{\y^*}^{\top}(\b-\mA(\widehat{\X}))-f(\X^*)-\widehat{\blambda}^{\top}g(\X^*)\le\varepsilon.
\end{align}
Plugging \eqref{ineq:func&normBound_regularized_3} into \eqref{ineq:func&normBound_regularized_2} we obtain that
$
(\widetilde{\blambda}-\blambda^*)^{\top}g(\widehat{\X})=\sum_{i\in\lbrace i\in[d]~\vert~g_i(\widehat{\X})>0\rbrace}g_i(\widehat{\X})\le\varepsilon$.
\end{proof}

\subsection{Proof of \cref{thm:inequlalityConstraints}}
\label{appx:proofThm56}

We first restate the theorem and then prove it. 
\begin{theorem}
Fix an optimal solution $\X^*\succeq0$ to Problem \eqref{problem:ineqConstrained}. Let $(\y^*,\blambda^*)\in\reals^m\times\reals^d_+$ be a corresponding dual solution and suppose $(\X^*,(\y^*,\blambda^*))$ satisfies the complementarity condition \eqref{eq:compcond:const} with rank $\tilde{r}:=n-\rank(\nabla_{\X}\mL(\X^*,\y^*,\blambda^*))$.  
Fix some $r\ge\tilde{r}$ and denote
$\delta(r):=
 \max\left\lbrace\lambda_{n-r}(\nabla_{\X}\mL(\X^*,\y^*,\blambda^*)),\sqrt{\tilde{r}}\lambda_{n-r+\tilde{r}-1}(\nabla_{\X}\mL(\X^*,\y^*,\blambda^*))\right\rbrace$. Denote $\mD_X=\lbrace\X\succeq0~\vert~\Vert\X-\X^*\Vert_F\le2\delta(r)\rbrace$ and $\mD_{\lambda}=\lbrace\blambda\in\reals^d_+~\vert~\Vert\blambda-\blambda^*\Vert_2\le2\delta(r)\rbrace$.
Let $\lbrace(\X_t,(\y_t,\blambda_t))\rbrace_{t\ge1}$ and $\lbrace(\Z_t,(\w_t,\bmu_t))\rbrace_{t\ge2}$ be the sequences of iterates generated by \cref{alg:EGregularized} (with $\mK=\reals^d_+$) with a fixed step-size
\begin{align*}
& \eta \le 
\min\left\lbrace\begin{array}{l}\frac{1}{\sqrt{2}\sqrt{(\beta+\max_{\blambda\in\mD_{\lambda}}\Vert\blambda\Vert_1\beta_g)^2+\Vert\mA\Vert^2+\beta_g^2}},\frac{1}{2\max\lbrace \Vert\mA\Vert,\sqrt{d}\max_{\X\in\mD_X}\max_{i\in[d]}\Vert\nabla{}g_i(\X)\Vert_F\rbrace},
\\ \frac{1}{\beta+\max_{\blambda\in\mD_{\lambda}}\Vert\blambda\Vert_1\beta_g+\sqrt{2}\max\lbrace \Vert\mA\Vert,\max_{\X\in\mD_X}\max_{i\in[d]}\Vert\nabla{}g_i(\X)\Vert_F\rbrace}
\end{array}\right\rbrace.
\end{align*}
Assume the initialization $(\X_1,(\y_1,\blambda_1))$ satisfies that $\Vert(\X_1,\y_1,\blambda_1)-(\X^*,\y^*,\blambda^*)\Vert_F\le R_0(r)$ where
\begin{align*}
& R_0(r):= 
\min\left\lbrace1,\frac{(1/2)\eta}{1+2\eta\max\lbrace\beta+\max\limits_{\blambda\in\mD_{\lambda}}\Vert\blambda\Vert_1\beta_g,\Vert\mA\Vert,\sqrt{d}\max\limits_{\X\in\mathcal{D}_X}\max
\limits_{i\in[d]}\Vert\nabla{}g_i(\X)\Vert_F\rbrace}\right\rbrace\delta(r).
\end{align*}
Then, for all $t\ge1$ the projections $\Pi_{\mbS^n_+}[\cdot]$ in \cref{alg:EGregularized} could be replaced with rank-r truncated projections without changing the sequences $\lbrace(\X_t,(\y_t,\blambda_t))\rbrace_{t\ge1}$ and $\lbrace(\Z_t,(\w_t,\bmu_t))\rbrace_{t\ge2}$. Moreover, for all $T\ge0$ it holds that
\begin{align*}
 f\left(\frac{1}{T}\sum_{t=1}^T\Z_{t+1}\right)-f(\X^*) & \le R^2/(2\eta T),
\\   \left\Vert\mA\left(\frac{1}{T}\sum_{t=1}^T\Z_{t+1}\right)-\b\right\Vert_2
& \le R^2/(2\Vert\y^*\Vert_2\eta T),
\\ \sum_{i\in\lbrace i\in[d]~\vert~g_i\left(\frac{1}{T}\sum_{t=1}^T\Z_{t+1}\right)>0\rbrace}g_i\left(\frac{1}{T}\sum_{t=1}^T\Z_{t+1}\right) & \le R^2/(2\eta T),
\end{align*}
where $R^2:=(\Vert\X_1-\X^*\Vert_F^2+\max_{\y\in\mC_y}\Vert\y_1-\y\Vert_2^2+\max_{\blambda\in\mC_{\lambda}}\Vert\blambda_1-\blambda\Vert_2^2)$ such that $\mC_y:=\lbrace\y\in\reals^m\ \vert\ \Vert\y\Vert_2\le2\Vert\y^*\Vert_2\rbrace$ and $\mC_{\lambda}:=\lbrace\blambda\in\reals^d_+\ \vert\ \Vert\blambda\Vert_2\le\Vert\blambda^*\Vert_2+\sqrt{d}\rbrace$.

Furthermore, if we take a step-size 
\begin{align*}
& \eta = 
\min\left\lbrace\begin{array}{l}\frac{1}{2\sqrt{(\beta+\max_{\blambda\in\mD_{\lambda}}\Vert\blambda\Vert_1\beta_g)^2+\Vert\mA\Vert^2+\beta_g^2}},\frac{1}{2\max\lbrace \Vert\mA\Vert,\sqrt{d}\max_{\X\in\mD_X}\max_{i\in[d]}\Vert\nabla{}g_i(\X)\Vert_F\rbrace},
\\ \frac{1}{\beta+\max_{\blambda\in\mD_{\lambda}}\Vert\blambda\Vert_1\beta_g+\sqrt{2}\max\lbrace \Vert\mA\Vert,\max_{\X\in\mD_X}\max_{i\in[d]}\Vert\nabla{}g_i(\X)\Vert_F\rbrace}
\end{array}\right\rbrace,
\end{align*}
then for all $T\ge0$ it holds that
\begin{align*}
 f(\Z_{T+1})-f(\X^*) & \le R_1/(\eta\sqrt{T}),
\\ \Vert\mA(\Z_{T+1})-\b\Vert_2 & \le R_1/(\Vert\y^*\Vert_2\eta\sqrt{T})
\\ \sum_{i\in\lbrace i\in[d]~\vert~g_i(\Z_{T+1})>0\rbrace}g_i(\Z_{T+1}) & \le R_1/(\eta \sqrt{T}),
\end{align*}
where
$R_1:=3\sqrt{2} \max\lbrace 2R_0(r),3\sqrt{2}\Vert\y^*\Vert_2,\sqrt{2}(2\Vert\blambda^*\Vert+\sqrt{d})\rbrace\Vert(\X_1,\y_1,\blambda_1)-(\X^*,\y^*,\blambda^*)\Vert$.


\end{theorem}

\begin{proof}
We first calculate the smoothness constants of $\mL(\X,\y,\blambda)=f(\X) + \boldsymbol{\lambda}^{\top}g(\X)+\y^{\top}(\b-\mathcal{A}(\X))$ as defined in \eqref{ineq:betas} with respect to the variables $\X$ and $(\y,\blambda)$ over $\mD_X\times\reals^m\times\mD_{\lambda}$. 
Here  $\nabla_{\X}\mL=\frac{\partial \mL}{\partial \X}$ and $\nabla_{(\y,\blambda)}\mL=(\frac{\partial \mL}{\partial \y},\frac{\partial \mL}{\partial \blambda})$. We have that
\begin{align*}
\Vert\nabla_{\X}\mL(\X,\y,\blambda)-\nabla_{\X}\mL(\widetilde{\X},\y,\blambda)\Vert_{F} 
& = \Vert\nabla f(\X)-\nabla f(\widetilde{\X})+\sum_{i=1}^d\blambda_i(\nabla{}g_i(\X)-\nabla{}g_i(\widetilde{\X}))\Vert_F
\\ & \le(\beta+\beta_g\max_{\blambda\in\mathcal{D}_{\lambda}}\Vert\blambda\Vert_1)\Vert\X-\widetilde{\X}\Vert_{F},  \\
\Vert\nabla_{(\y,\blambda)}\mL(\X,\y,\blambda)-\nabla_{(\y,\blambda)}\mL(\X,\widetilde{\y},\widetilde{\blambda})\Vert_{2} & = 0 = 0\Vert(\y,\blambda)-(\widetilde{\y},\widetilde{\blambda})\Vert_{2}, \\
\Vert\nabla_{\X}\mL(\X,\y,\blambda)-\nabla_{\X}\mL(\X,\widetilde{\y},\widetilde{\blambda})\Vert_{F}
& =\Vert\mA^{\top}(\y-\widetilde{\y})+\sum_{i=1}^d(\blambda_i-\widetilde{\blambda}_i)^{\top}\nabla{}g_i(\X)\Vert_{F}
\\ & \le \Vert\mA\Vert\Vert\y-\widetilde{\y}\Vert_2+\max_{\X\in\mathcal{D}_X}\max_{i\in[d]}\Vert\nabla{}g_i(\X)\Vert_F\Vert\blambda-\widetilde{\blambda}\Vert_{1}
\\ & \le \Vert\mA\Vert\Vert\y-\widetilde{\y}\Vert_2+\sqrt{d}\max_{\X\in\mathcal{D}_X}\max_{i\in[d]}\Vert\nabla{}g_i(\X)\Vert_F\Vert\blambda-\widetilde{\blambda}\Vert_{2}
\\ & \le \sqrt{2}\max\lbrace \Vert\mA\Vert,\sqrt{d}\max_{\X\in\mathcal{D}_X}\max_{i\in[d]}\Vert\nabla{}g_i(\X)\Vert_F\rbrace\Vert(\y,\blambda)-(\widetilde{\y},\widetilde{\blambda})\Vert_{2},\\
\Vert\nabla_{(\y,\blambda)}\mL(\X,\y,\blambda)-\nabla_{(\y,\blambda)}\mL(\widetilde{\X},\y,\blambda)\Vert_{2} & = \left\Vert\left(\begin{array}{c}\mA(\X-\widetilde{\X})
\\ g(\X)-g(\widetilde{\X})\end{array}\right)\right\Vert_{2}  
= \sqrt{\Vert\mA(\X-\widetilde{\X})\Vert_2^2+\Vert g(\X)-g(\widetilde{\X})\Vert_2^2}
\\ & \le\sqrt{\Vert\mA\Vert^2+\beta_g^2}\Vert\X-\widetilde{\X}\Vert_{F}.
\end{align*}
Thus, we obtain that $\beta_X = \beta+\beta_g\max_{\blambda\in\mathcal{D}_{\lambda}}\Vert\blambda\Vert_2$, $\beta_{(y,\lambda)} = 0$, $\beta_{(y,\lambda)X} = \sqrt{\Vert\mA\Vert^2+\beta_g^2}$, $\beta_{X(y,\lambda)} =  \sqrt{2}\max\lbrace \Vert\mA\Vert,\sqrt{d}\max_{\X\in\mathcal{D}_X}\max_{i\in[d]}\Vert\nabla{}g_i(\X)\Vert_F\rbrace$, and  
\begin{align} \label{eq:betaL_ineqConstrained}
& \beta_{\mL}= \max\left\lbrace\sqrt{2}\sqrt{(\beta+\beta_g\max_{\blambda\in\mathcal{D}_{\lambda}}\Vert\blambda\Vert_1)^2+\Vert\mA\Vert^2+\beta_g^2},2\Vert\mA\Vert,2\sqrt{d}\max_{\X\in\mathcal{D}_X}\max_{i\in[d]}\Vert\nabla{}g_i(\X)\Vert_F\right\rbrace,
\end{align}
over $\mD_X\times\reals^m\times\mD_{\lambda}$.

We will prove by induction that for all $t\ge1$ all the iterates $\lbrace(\X_t,(\y_t,\blambda_t))\rbrace_{t\ge1}$ and $\lbrace(\Z_t,(\w_t,\bmu_t))\rbrace_{t\ge2}$ satisfy that
\begin{align*}
& \Vert(\X_t,\y_t,\blambda_t)-(\X^*,\y^*,\blambda^*)\Vert \le R_0(r),
\\ & \Vert(\Z_t,\w_t,\bmu_t)-(\X^*,\y^*,\blambda^*)\Vert
\le 2R_0(r).
\end{align*}

By the choice of the initialization $(\X_1,\y_1,\blambda_1)$ it satisfies the condition trivially. We will assume that all iterates up to some $t\ge1$ satisfy that 
$\Vert(\X_t,\y_t,\blambda_t)-(\X^*,\y^*,\blambda^*)\Vert_F \le R_0(r)$ and $\Vert(\Z_{t},\w_{t},\bmu_{t})-(\X^*,\y^*,\blambda^*)\Vert
\le 2R_0(r)$.
Then, all the generated iterates $\lbrace(\X_t,(\y_t,\blambda_t))\rbrace_{k=1}^t,\lbrace(\Z_{k},(\w_{k},\bmu_{k}))\rbrace_{k=2}^t\subset\mD_X\times\reals^m\times\mD_{\lambda}$, and so we can take $\beta_{\mL}$ as defined in \eqref{eq:betaL_ineqConstrained}. Since we choose $\eta\le 1/\beta_{\mL}$, by invoking Lemma 8 in \cite{ourExtragradient}, we have that
\begin{align*}
&\Vert(\X_{t+1},\y_{t+1},\blambda_{t+1})-(\X^*,\y^*,\blambda^*)\Vert
\le \Vert(\X_{t},\y_{t},\blambda_{t})-(\X^*,\y^*,\blambda^*)\Vert\le R_0(r),
\end{align*}
and 
from the nonexpansiveness of the Euclidean projection, we have that 
\begin{align*}
&\Vert(\Z_{t+1},\w_{t+1},\bmu_{t+1})-(\X^*,\y^*,\blambda^*)\Vert
\\ & \le \left\Vert\left(\begin{array}{c}\X_{t}-\eta\nabla_{\X}\mL(\X_t,\y_t,\blambda_t)
\\ \y_{t}+\eta\nabla_{\y}\mL(\X_t,\y_t,\blambda_t)
\\ \blambda_{t}+\eta\nabla_{\blambda}\mL(\X_t,\y_t,\blambda_t)\end{array}\right)-\left(\begin{array}{c}\X^*-\eta\nabla_{\X}\mL(\X^*,\y^*,\blambda^*)
\\ \y^*+\eta\nabla_{\y}\mL(\X^*,\y^*,\blambda^*)
\\ \blambda^*+\eta\nabla_{\blambda}\mL(\X^*,\y^*,\blambda^*)\end{array}\right)\right\Vert
\\ & \le \left\Vert\left(\begin{array}{c}\X_{t} \\ \y_{t} \\ \blambda_{t}\end{array}\right)-\left(\begin{array}{c}\X^* \\ \y^* \\ \blambda^*\end{array}\right)\right\Vert 
+\eta\left\Vert\left(\begin{array}{c}\nabla_{\X}\mL(\X_t,\y_t,\blambda_t) \\ -\nabla_{\y}\mL(\X_t,\y_t,\blambda_t) \\ -\nabla_{\blambda}\mL(\X_t,\y_t,\blambda_t)\end{array}\right)-\left(\begin{array}{c}\nabla_{\X}\mL(\X^*,\y^*,\blambda^*) \\ -\nabla_{\y}\mL(\X^*,\y^*,\blambda^*) \\ -\nabla_{\blambda}\mL(\X^*,\y^*,\blambda^*)\end{array}\right)\right\Vert
\\ & \le (1+\eta\beta_{\mL})\Vert(\X_{t},\y_{t},\blambda_{t})-(\X^*,\y^*,\blambda^*)\Vert\le 2 R_0(r),
\end{align*}
as desired. The second to last inequality follows since $(\X_{t},\y_{t},\blambda_{t})\subset\mD_X\times\reals^m\times\mD_{\lambda}$ and so $\beta_{\mL}$ is indeed the correct smoothness parameter. Therefore, we obtain that $(\X_{t+1},\y_{t+1},\blambda_{t+1}),(\Z_{t+1},\w_{t+1},\bmu_{t+1})\subset\mD_X\times\reals^m\times\mD_{\lambda}$

Since for all $t\ge1$ it holds that 
\begin{align*}
\Vert\X_{t}-\X^*\Vert & \le \Vert(\X_{t},\y_{t},\blambda_t)-(\X^*,\y^*,\blambda^*)\Vert 
\\ & \le \max\left\lbrace\Vert(\X_{t},\y_{t},\blambda_t)-(\X^*,\y^*,\blambda^*)\Vert,\Vert(\Z_{t+1},\w_{t+1},\bmu_{t+1})-(\X^*,\y^*,\blambda^*)\Vert\right\rbrace,
\end{align*}
we have that for all $t\ge1$ the condition in Lemma \ref{lemma:radius} holds for $(\X_{t},\y_{t},\blambda_{t})$ and $(\Z_{t+1},\w_{t+1},\bmu_{t+1})$, and so for all $t\ge1$ it follows that 
\begin{align*}
& \rank(\Pi_{\mbS^n_+}[\X_t-\eta\nabla{}_{\X}\mL(\X_t,\y_t,\blambda_t)])\le r
\\ & \rank(\Pi_{\mbS^n_+}[\X_t-\eta\nabla{}_{\X}\mL(\Z_{t+1},\w_{t+1},\bmu_{t+1})])\le r.
\end{align*}
Hence, the iterates of projected extragradient method will remain unchanged when replacing all projections onto the PSD cone with their rank-$r$ truncated counterparts, and so, the method will also maintain its original convergence rate.

Replacing the $\y$ variable in Theorem \ref{thm:SDP} with the product $(\y,\blambda)$, replacing the function $f(\X)$ with $f(\X)+\max_{\blambda\in\reals^d_+}\blambda^{\top}g(\X)$, and adjusting the smoothness parameters as calculated above, we have from Theorem \ref{thm:SDP} that all projections can be replaced with rank-$r$ truncated  projections while maintaining the original convergence rate of the extragradient method.
Denoting $\mC_y=\lbrace\y\in\reals^m\ \vert\ \Vert\y\Vert_2\le2\Vert\y^*\Vert_2\rbrace$ and $\mC_{\lambda}=\lbrace\blambda\in\reals^d_+\ \vert\ \Vert\blambda\Vert_2\le\Vert\blambda^*\Vert_2+\sqrt{d}\rbrace$, we also have that the inequality of \eqref{ineq:inProof1111} can be replaced 
with the following inequality with for all $\y\in\mC_y$ and $\blambda\in\mC_{\lambda}$
\begin{align} \label{ineq:inProof222}
& \frac{1}{T}\sum_{t=1}^T \mL(\Z_{t+1},\y,\blambda) - \frac{1}{T}\sum_{t=1}^T \mL(\X^*,\w_{t+1},\bmu_{t+1})\le\varepsilon,
\end{align}
where $\varepsilon:=\frac{1}{2\eta T}\left(\Vert\X_1-\X^*\Vert_F^2+\max\limits_{\y\in\mC_y}\Vert\y_1-\y\Vert_2^2+\max\limits_{\blambda\in\mC_{\lambda}}\Vert\blambda_1-\blambda\Vert_2^2\right)$.
The second term in the LHS of \eqref{ineq:inProof222} is equivalent to 
\begin{align} \label{ineq:inProof221}
\frac{1}{T}\sum_{t=1}^T \mL(\X^*,\w_{t+1},\bmu_{t+1}) & = f(\X^*)+\frac{1}{T}\sum_{t=1}^T \bmu_{t+1}^{\top}g(\X^*)+\frac{1}{T}\sum_{t=1}^T\w_{t+1}^{\top}(\b-\mA(\X^*)) \nonumber 
\\ & = \mL\left(\X^*,\frac{1}{T}\sum_{t=1}^T\w_{t+1},\frac{1}{T}\sum_{t=1}^T\bmu_{t+1}\right).
\end{align}
 Plugging \eqref{ineq:inProof221} into \eqref{ineq:inProof222} 
 and using the convexity of $\mL(\cdot,\y,\blambda)$ we have in particular that for all $\y\in\mC_y$ and $\blambda\in\mC_{\lambda}$ it holds that
\begin{align*}
& \mL\left(\frac{1}{T}\sum_{t=1}^T\Z_{t+1},\y,\blambda\right) - \mL\left(\X^*,\frac{1}{T}\sum_{t=1}^T\w_{t+1},\frac{1}{T}\sum_{t=1}^T\bmu_{t+1}\right) \le \varepsilon.
\end{align*}
Invoking Lemma \ref{lemma:boundf_and_norm_ineqConstrained} we obtain that $f\left(\frac{1}{T}\sum_{t=1}^T\Z_{t+1}\right)-f(\X^*)\le\varepsilon$, $\left\Vert\mA\left(\frac{1}{T}\sum_{t=1}^T\Z_{t+1}\right)-\b\right\Vert_2 \le\varepsilon/\Vert\y^*\Vert_2$, and $\sum_{i\in\lbrace i\in[d]~\vert~g_i\left(\frac{1}{T}\sum_{t=1}^T\Z_{t+1}\right)>0\rbrace}g_i\left(\frac{1}{T}\sum_{t=1}^T\Z_{t+1}\right)
\le \varepsilon$.

For the second part of the theorem, 
denote $$\mD=\left\lbrace(\X,\y,\blambda)\in\mathbb{S}^n_+\times\reals^m\times\reals^d_+~|~\Vert(\X,\y,\blambda)-(\X^*,\y^*,\blambda^*)\Vert\le \max\lbrace 2R_0(r),3\sqrt{2}\Vert\y^*\Vert_2,\sqrt{2}(2\Vert\blambda^*\Vert+\sqrt{d}\rbrace\right\rbrace.$$ 
For all $\y\in\mC_y$ and $\blambda\in\mC_{\lambda}$ it holds that $(\X^*,\y,\blambda)\in\mD$ since 
\begin{align*}
\Vert(\X^*,\y,\blambda)-(\X^*,\y^*,\blambda^*)\Vert & = \Vert(\y,\blambda)-(\y^*,\blambda^*)\Vert_2=\sqrt{\Vert\y-\y^*\Vert_2^2+\Vert\blambda-\blambda^*\Vert_2^2}
\\ & \le\sqrt{(\Vert\y\Vert_2+\Vert\y^*\Vert_2)^2+(\Vert\blambda\Vert_2+\Vert\blambda^*\Vert_2)^2}\le\sqrt{(3\Vert\y^*\Vert_2)^2+(2\Vert\blambda^*\Vert+\sqrt{d})^2}
\\ & \le\sqrt{2}\max\lbrace 3\Vert\y^*\Vert_2,2\Vert\blambda^*\Vert+\sqrt{d}\rbrace.
\end{align*}
The inequalities in \eqref{ineq:inProof0001} still hold here for our choice of $\mD$. Therefore, summing the two inequalities in \eqref{ineq:inProof0001} and taking the maximum over all  $(\X',\y',\blambda')\in\mD$ we obtain that in particular for all $\y\in\mC_y$ and $\blambda\in\mC_{\lambda}$ it holds that
\begin{align*}
& \mL(\Z_{T+1},\y,\blambda)-\mL(\X^*,\w_{T+1},\bmu_{T+1})
\\ & \le \max_{(\X',\y',\blambda')\in\mD}\left(\begin{array}{l} \langle\Z_{T+1}-\X',\nabla_{\X}\mL(\Z_{T+1},\w_{T+1},\bmu_{T+1})\rangle + \langle\w_{T+1}-\y',-\nabla_{\y}\mL(\Z_{T+1},\w_{T+1},\bmu_{T+1})\rangle
\\ + \langle\bmu_{T+1}-\blambda',-\nabla_{\blambda}\mL(\Z_{T+1},\w_{T+1},\bmu_{T+1})\rangle\end{array}\right)
\\ & \underset{(a)}{\le} \frac{1}{\sqrt{T}}\frac{3 \max\lbrace 2R_0(r),3\sqrt{2}\Vert\y^*\Vert_2,\sqrt{2}(2\Vert\blambda^*\Vert+\sqrt{d})\rbrace\Vert(\X_1,\y_1,\blambda_1)-(\X^*,\y^*,\blambda^*)\Vert}{\eta\sqrt{1-\eta^2\beta_{\mL}^2}}
\\ & \underset{(b)}{\le} \frac{1}{\sqrt{T}}\frac{3\sqrt{2} \max\lbrace 2R_0(r),3\sqrt{2}\Vert\y^*\Vert_2,\sqrt{2}(2\Vert\blambda^*\Vert+\sqrt{d})\rbrace\Vert(\X_1,\y_1,\blambda_1)-(\X^*,\y^*,\blambda^*)\Vert}{\eta}.
\end{align*}
where (a) follows from Theorem 6 in \cite{lastIterateDecrease} and (b) follows from our choice of step-size which satisfies that $\eta\le 1/(\sqrt{2}\beta_{\mL})$. Invoking \cref{lemma:boundf_and_norm_ineqConstrained} we obtain the result.
\end{proof}

\section{Bounds on Norms of Optimal Dual Solutions}
\label{appx:boundsOptimalDualSolutions}

In the following lemma we show that under Slater's condition there exists a bounded dual optimal solution for SDPs.
Many of the arguments in the proof are presented in \cite{boundInequalityDualVariables1, boundInequalityDualVariables2}, yet we bring the full proof for completeness. 

\begin{lemma}
Let $\X^*\succeq0$ be an optimal solution to Problem \eqref{problem:ineqConstrained} and assume that Assumptions \ref{Ass:slater_inequality} and \ref{Ass:primalBounded} hold for some $\bar{\X}\in\mathbb{S}^n_{++}$. Let $\vect(\cdot)$ be an operator that transforms a matrix into a vector by vertically stacking the columns of the matrix, and denote $\M:=\mat(\mathcal{A}) = (\vect(\A_1),\ldots,\vect(\A_m))\in\reals^{n^2\times m}$. Then, for any corresponding dual solution $\boldsymbol{\lambda}^*\in\reals^d_+$ it holds that
\begin{align*}
\Vert\boldsymbol{\lambda}^*\Vert_1\le \frac{1}{\min\limits_{i\in[d]}\lbrace-g_i(\bar{\X})\rbrace}(f(\bar{\X})-f(\X^*))
\end{align*}
and there exist a corresponding dual solution $\y^*\in\reals^m$ such that
\begin{align*}
& \Vert\y^*\Vert_2\le 
\frac{1}{\sigma_{min}(\M)}\left(\Vert\nabla f(\X^*)\Vert_F + \left(\frac{\max\limits_{i\in[d]}\Vert\nabla g_i(\X^*)\Vert_F}{\min\limits_{i\in[d]}\lbrace-g_i(\bar{\X})\rbrace}+\frac{1}{\lambda_{n}(\bar{\X})}\right)(f(\bar{\X})-f(\X^*))\right),
\end{align*}
where $\sigma_{min}(\M)$ is the smallest non-zero singular value of $\M$.

If $g(\X)\equiv0$ then the bound on $\Vert\boldsymbol{\lambda}^*\Vert_1$ is unnecessary and the bound on $\Vert\y^*\Vert_2$ can be replaced with
\begin{align*}
& \Vert\y^*\Vert_2\le 
\frac{1}{\sigma_{min}(\M)}\left(\Vert\nabla f(\X^*)\Vert_F + \frac{1}{\lambda_{n}(\bar{\X})}(f(\bar{\X})-f(\X^*))\right).
\end{align*}
\end{lemma}

\begin{proof}

Under Assumptions \ref{Ass:slater_inequality} and \ref{Ass:primalBounded} strong duality holds. Therefore, $f(\X^*)$ is equal to the optimal dual solution, and so,
\begin{align} \label{ineq:strongDualitySlaterPoint}
f(\X^*) & =\min_{\X\succeq0}\lbrace f(\X) + {\boldsymbol{\lambda}^*}^{\top}g(\X)+{\y^*}^{\top}(\b-\mathcal{A}(\X))-\langle\Z^*,\X\rangle\rbrace \nonumber
\\ & \le f(\bar{\X}) + {\boldsymbol{\lambda}^*}^{\top}g(\bar{\X})+{\y^*}^{\top}(\b-\mathcal{A}(\bar{\X}))-\langle\Z^*,\bar{\X}\rangle \nonumber
\\ & = f(\bar{\X}) + {\boldsymbol{\lambda}^*}^{\top}g(\bar{\X})-\langle\Z^*,\bar{\X}\rangle.
\end{align}
$\Z^*\succeq0$ and $\bar{\X}\succ0$, and therefore, it follows that $\langle\Z^*,\bar{\X}\rangle\ge0$. Thus, rearranging and using the facts that $g(\bar{\X})\le0$ and $\boldsymbol{\lambda}^*\ge0$,
\begin{align*}
\min_{i\in[d]}\lbrace-g_i(\bar{\X})\rbrace\sum_{i=1}^d\boldsymbol{\lambda}^*_i \le -\sum_{i=1}^d\boldsymbol{\lambda}^*_i g_i(\bar{\X}) \le f(\bar{\X})-f(\X^*).
\end{align*}
If $g(\X)\not=0$ then since $\boldsymbol{\lambda}^*\ge0$ this is equivalent to 
\begin{align} \label{ineq:lambdaStarBound}
\Vert\boldsymbol{\lambda}^*\Vert_1\le \frac{1}{\min_{i\in[d]}\lbrace-g_i(\bar{\X})\rbrace}(f(\bar{\X})-f(\X^*)).
\end{align}

In addition, from \eqref{ineq:strongDualitySlaterPoint} since ${\boldsymbol{\lambda}^*}^{\top}g(\bar{\X})\le0$, $\Z^*\succeq0$, and $\bar{\X}\succ0$ and by invoking von Neumann's trace inequality,
\begin{align*}
\lambda_{n}(\bar{\X})\sum_{i=1}^n\lambda_i(\Z^*) & \le \sum_{i=1}^n\lambda_i(\Z^*)\lambda_{n-i+1}(\bar{\X}) \le \langle\Z^*,\bar{\X}\rangle \le f(\bar{\X}) - f(\X^*).
\end{align*}

Since $\Z^*\succeq0$ and $\bar{\X}\succ0$ this is equivalent to 
\begin{align} \label{ineq:ZStarBound}
\Vert\Z^*\Vert_* \le \frac{1}{\lambda_{n}(\bar{\X})}(f(\bar{\X}) - f(\X^*)).
\end{align}

Strong duality holds, and so, by the KKT conditions optimal primal-dual solutions $\X^*$ and $(\y^*,\blambda^*,\Z^*)$ satisfy that
\begin{align} \label{eq:KKTforConstrainedProblem}
\sum_{i=1}^m \y^*_i\A_i = \nabla f(\X^*) + \sum_{i=1}^d \boldsymbol{\lambda}^*_i\nabla g_i(\X^*)-\Z^*.
\end{align}

The linear system in \eqref{eq:KKTforConstrainedProblem} can be written as the following overdetermined linear system:
\begin{align*}
\mat(\mathcal{A})\y^* = \vect\left(\nabla f(\X^*) + \sum_{i=1}^d \boldsymbol{\lambda}^*_i\nabla g_i(\X^*)-\Z^*\right).
\end{align*}
Note that if $\A_1,\cdots,\A_m$ are linearly dependent then there are infinitely many solutions to this linear system.

Denote $\M:=\mat(\mathcal{A})$ and $\M^{\dagger}$ its pseudo-inverse. Thus, the solution $\widetilde{\y}^*$ to the linear system of minimum norm satisfies that
\begin{align*}
\widetilde{\y}^* = \M^{\dagger}\vect\left(\nabla f(\X^*) + \sum_{i=1}^d \boldsymbol{\lambda}^*_i\nabla g_i(\X^*)-\Z^*\right),
\end{align*}
and assuming $g(\X)\not=0$ and using \eqref{ineq:lambdaStarBound} and \eqref{ineq:ZStarBound} it can be bounded by
\begin{align}
\Vert\widetilde{\y}^*\Vert_2 & 
 \le \Vert\M^{\dagger}\Vert_2\left\Vert\vect\left(\nabla f(\X^*) + \sum_{i=1}^d \boldsymbol{\lambda}^*_i\nabla g_i(\X^*)-\Z^*\right)\right\Vert_2 \nonumber
\\ & = \sigma_{max}(\M^{\dagger})\left\Vert \nabla f(\X^*) + \sum_{i=1}^d \boldsymbol{\lambda}^*_i\nabla g_i(\X^*)-\Z^*\right\Vert_F \nonumber
\\ & \le \frac{1}{\sigma_{min}(\M)}\left(\Vert\nabla f(\X^*)\Vert_F + \sum_{i=1}^d \vert\boldsymbol{\lambda}^*_i\vert\Vert\nabla g_i(\X^*)\Vert_F+\Vert\Z^*\Vert_F\right) \nonumber
\\ & \le \frac{1}{\sigma_{min}(\M)}\left(\Vert\nabla f(\X^*)\Vert_F + \Vert\boldsymbol{\lambda}^*\Vert_1\max_{i\in[d]}\Vert\nabla g_i(\X^*)\Vert_F+\Vert\Z^*\Vert_*\right) \label{ineq:midProof}
\\ & \le \frac{1}{\sigma_{min}(\M)}\left(\Vert\nabla f(\X^*)\Vert_F + \left(\frac{\max\limits_{i\in[d]}\Vert\nabla g_i(\X^*)\Vert_F}{\min\limits_{i\in[d]}\lbrace-g_i(\bar{\X})\rbrace}+\frac{1}{\lambda_{n}(\bar{\X})}\right)(f(\bar{\X})-f(\X^*))\right), \nonumber
\end{align}
as desired. If $g(\X)\equiv0$ then \eqref{ineq:midProof} can be bounded as $\Vert\y^*\Vert_2\le 
\frac{1}{\sigma_{min}(\M)}\left(\Vert\nabla f(\X^*)\Vert_F + \frac{1}{\lambda_{n}(\bar{\X})}(f(\bar{\X})-f(\X^*))\right)$.
\end{proof}

\section{Additional Details Regarding the Empirical Evidence in \cref{sec:experiments}}
\label{appx:moreExperiments}

As stated in \cref{sec:experiments}, we initialize the $\X$ variable with the matrix $\X_1=(1/\trace(\boldsymbol{\Lambda}_{r^*}))\sign(\V_{r^*})\boldsymbol{\Lambda}_{r^*}\sign(\V_{r^*})^{\top}$, where $r^*=\rank(\X^*)$ and $\V_{r^*}(-\boldsymbol{\Lambda}_{r^*}){\V_{r^*}}^{\top}$ is the rank-$r^*$ approximation of $-\C$. Since $\C$ is not positive or negative semidefinite there is no guarantee that the values on the diagonal of $\boldsymbol{\Lambda}_{r^*}$ are all with the same sign. However, since we take only the bottom $r^*$ eigenvalues of $\C$ and $r^*$ is small, in practice in all cases we obtain that the values on the diagonal of $\boldsymbol{\Lambda}_{r^*}$ are all negative, and thus $\X_1\succeq0$ as desired. In addition, this choice of $\X_1$ is feasible as $\diag(\X_1)=1$.

We run the experiments in Matlab R2022a. For computing the rank-r SVDs necessary for the rank-r truncated projections we use the $\texttt{eigs}$ function with the default maximum size of Krylov subspace of $\max\lbrace 2r,20\rbrace$, yet we lower the convergence tolerance parameter to $10^{-5}$  to ensure all the eigenvalues converge.

In \cref{table:etaTChoices} we specify the step-size $\eta$ and number of iterations $T$ we used for the experiments.
\begin{table}[h]
\caption{Step-size and number of iterations in each experiment in \cref{sec:experiments}.}
\label{table:etaTChoices}
\vskip 0.15in
\begin{center}
\begin{small}
\begin{sc}
\begin{tabular}{lcccccccc}
\toprule
graph  & $\textbf{G1}$-$\textbf{G10}$ & $\textbf{G11}$ & $\textbf{G12}$ & $\textbf{G13}$ & $\textbf{G14}$-$\textbf{G15}$ & $\textbf{G16}$ & $\textbf{G17}$ & $\textbf{G18}-\textbf{G20}$ \\
\midrule
$\eta$ & $4$ & $2$ & $1.9$ & $2.2$ & $2.4$ & $2.2$ & $2.3$ & $2.8$ \\
$T$ & $1000$ & $20000$ & $10000$ & $10000$ & $5000$ & $5000$ & $5000$ & $5000$ \\
\bottomrule
\end{tabular}
\end{sc}
\end{small}
\end{center}
\vskip -0.1in
\end{table}

In \cref{fig:time_tableTC_full} we plot the time it takes to run \cref{alg:EG} with varying SVD ranks for the Max-Cut problem for all the graphs missing from \cref{fig:time_tableTC_some}. All plots are averaged over 5 i.i.d. runs.
\begin{figure*}[t] 
\vskip 0.2in
\begin{center}
\centering   
\subfigure[G1]{\includegraphics[width=0.24\textwidth]{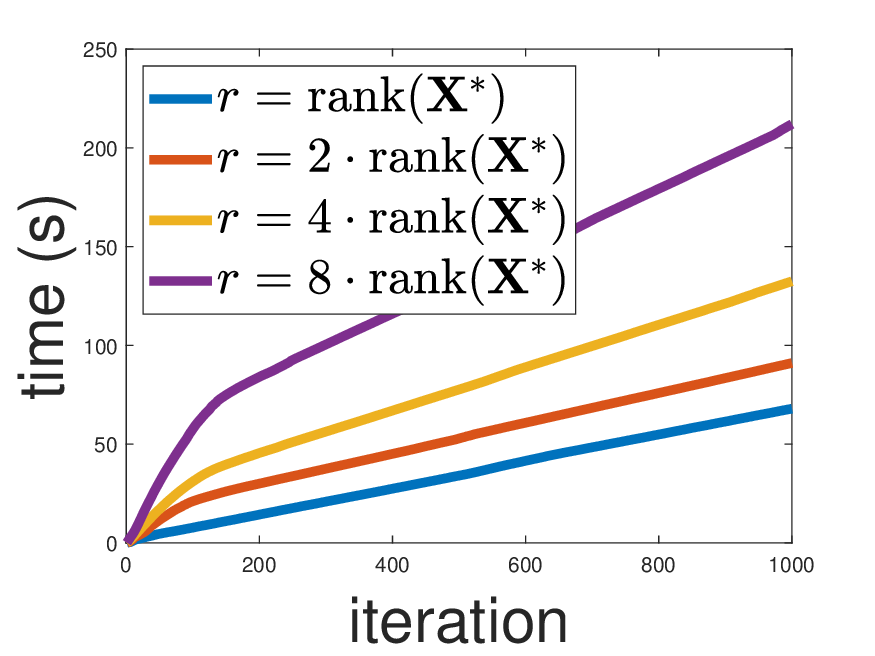}}
\subfigure[G2]{\includegraphics[width=0.24\textwidth]{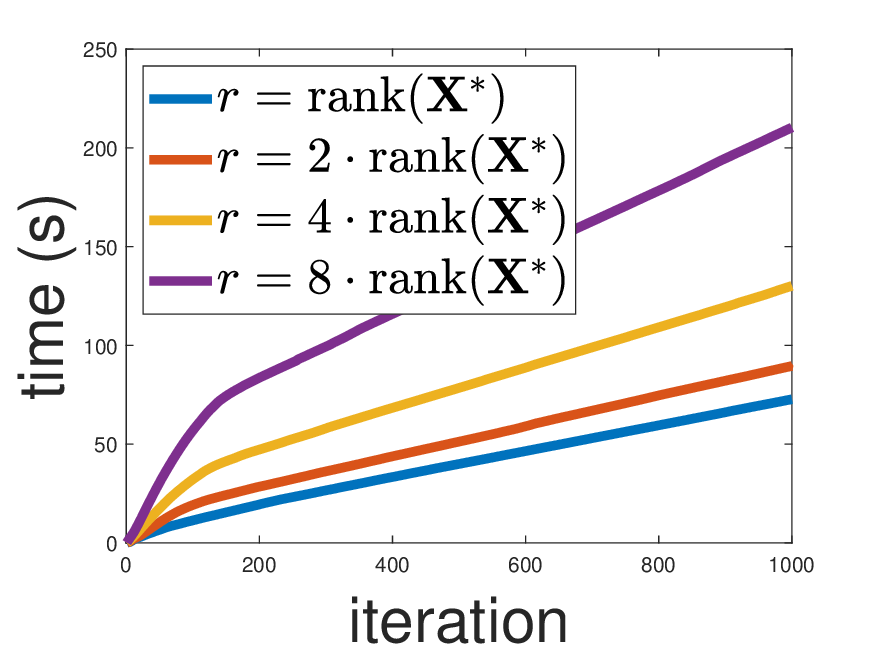}}
\subfigure[G3]{\includegraphics[width=0.24\textwidth]{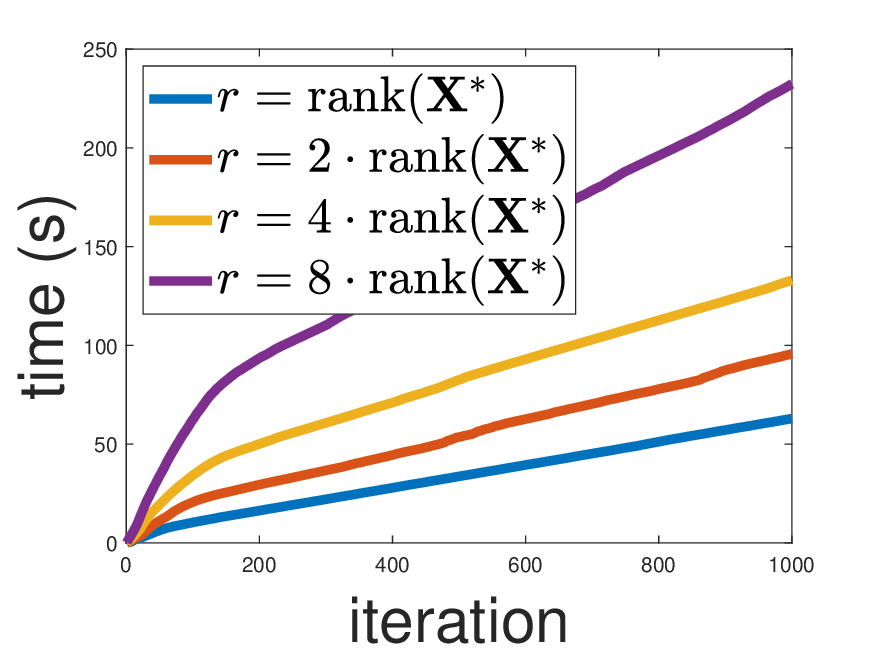}}
\subfigure[G4]{\includegraphics[width=0.24\textwidth]{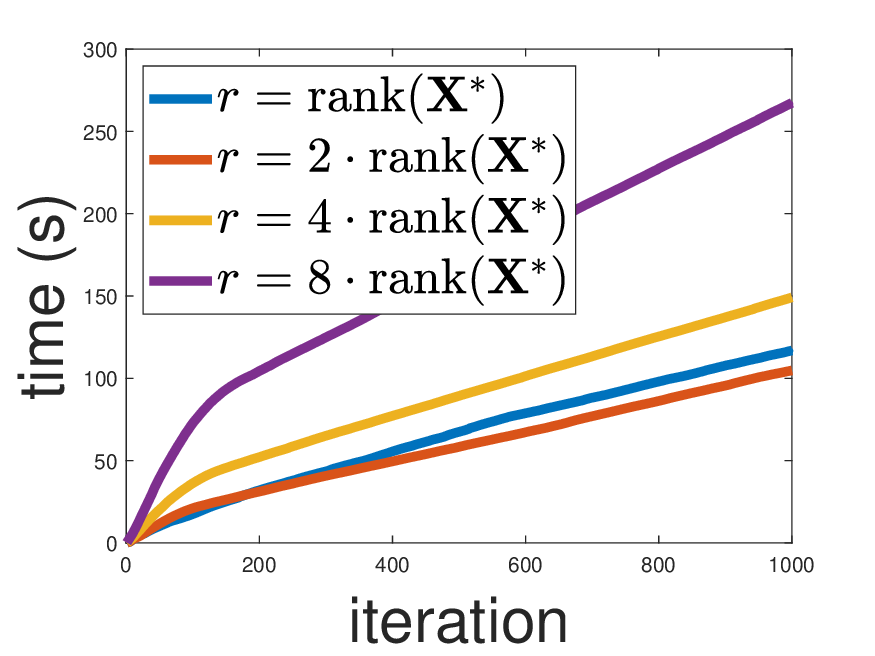}}
\medskip
\subfigure[G5]{\includegraphics[width=0.24\textwidth]{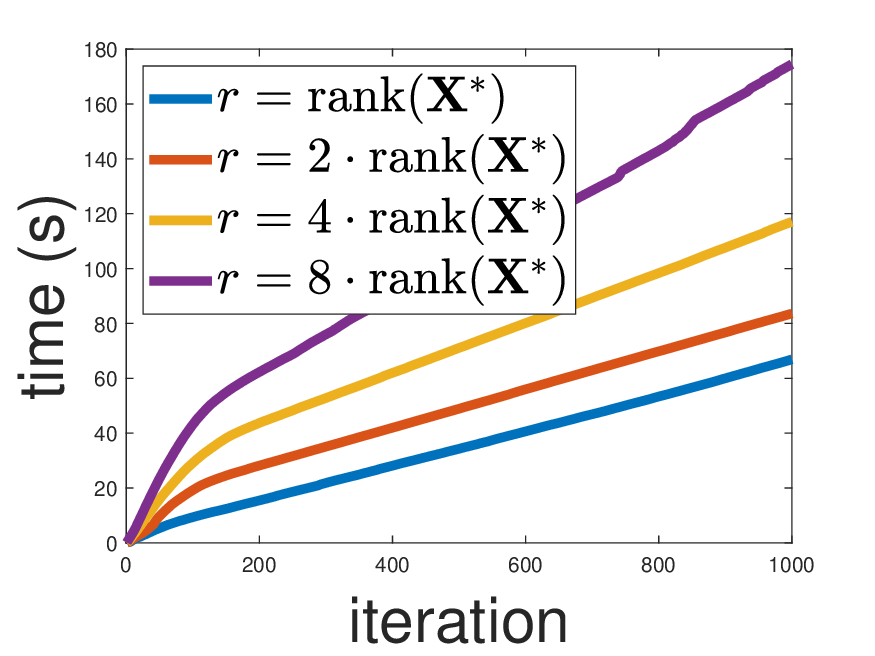}}
\subfigure[G6]{\includegraphics[width=0.24\textwidth]{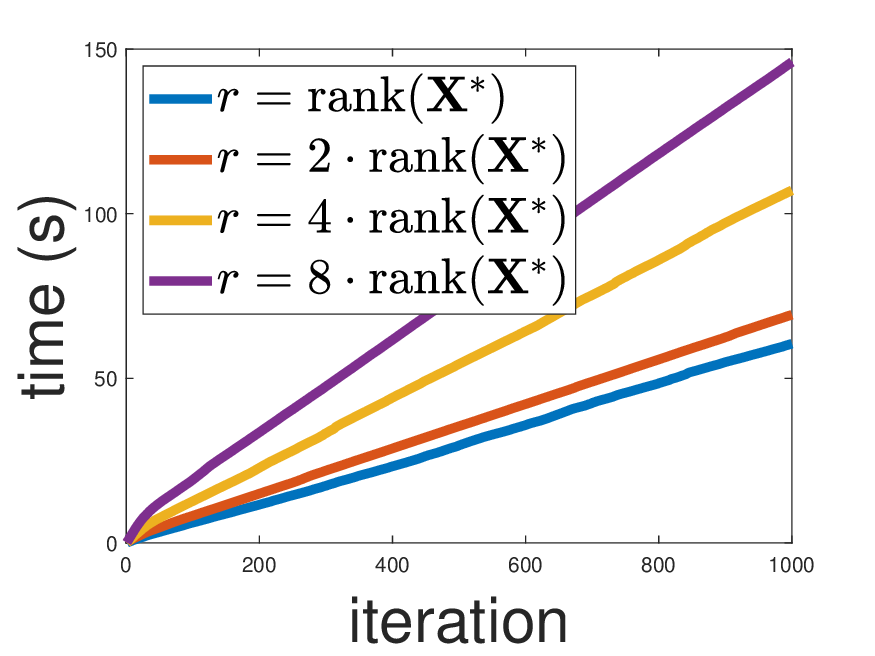}}
\subfigure[G7]{\includegraphics[width=0.24\textwidth]{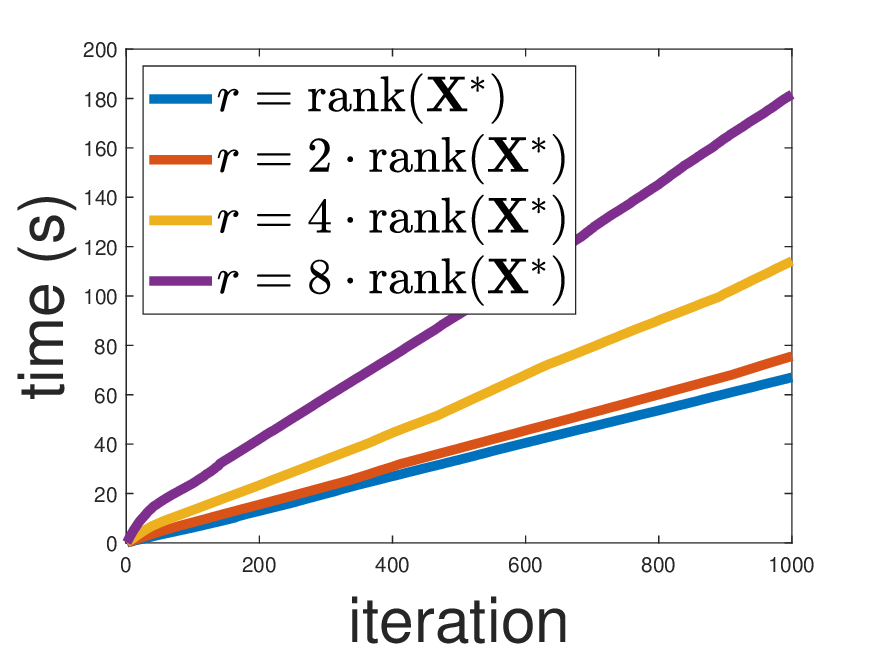}}
\subfigure[G8]{\includegraphics[width=0.24\textwidth]{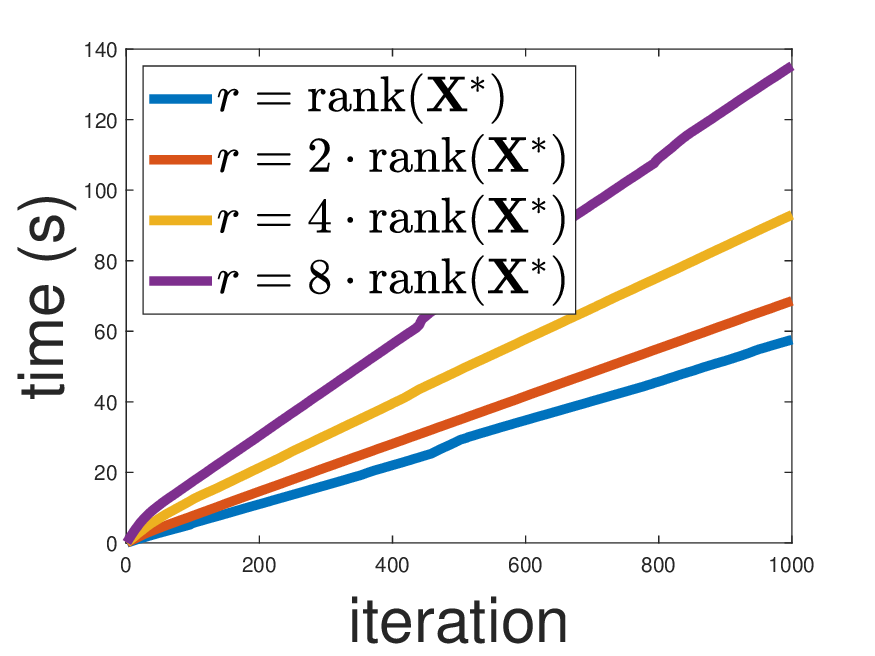}}
\medskip
\subfigure[G9]{\includegraphics[width=0.24\textwidth]{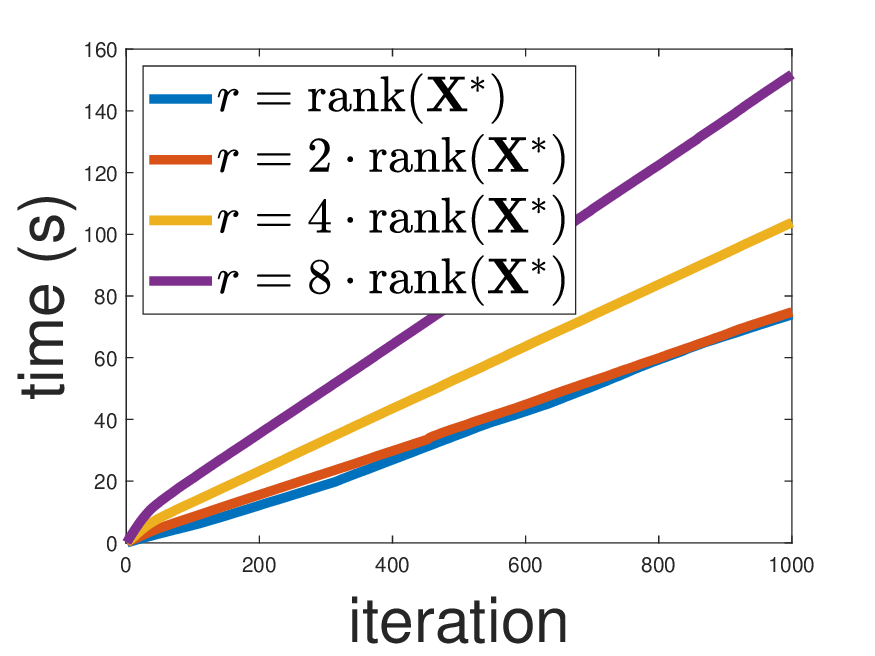}}
\subfigure[G10]{\includegraphics[width=0.24\textwidth]{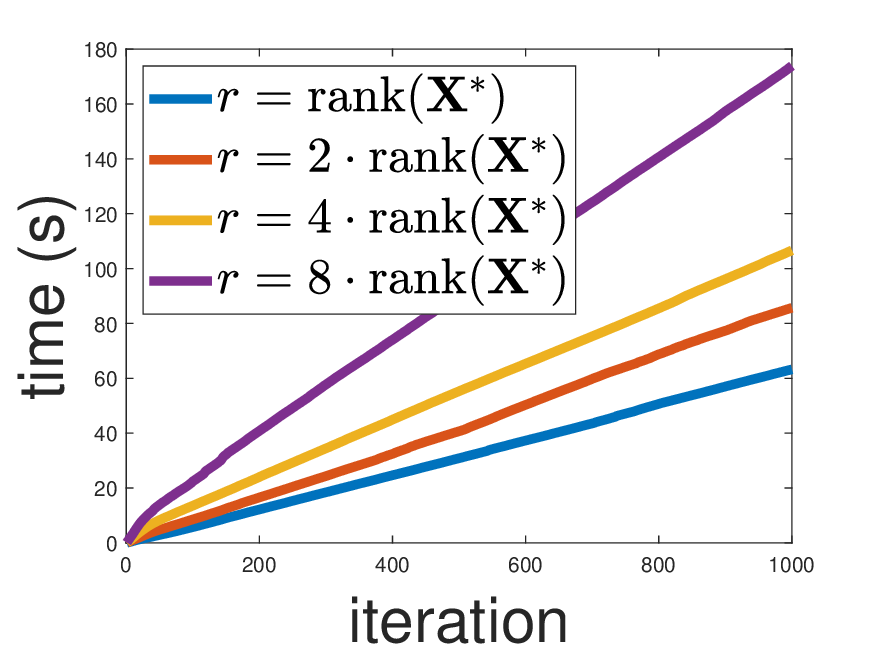}}
\subfigure[G11]{\includegraphics[width=0.24\textwidth]{G11_time}}
\subfigure[G12]{\includegraphics[width=0.24\textwidth]{G12_time}}
\medskip
\subfigure[G13]{\includegraphics[width=0.24\textwidth]{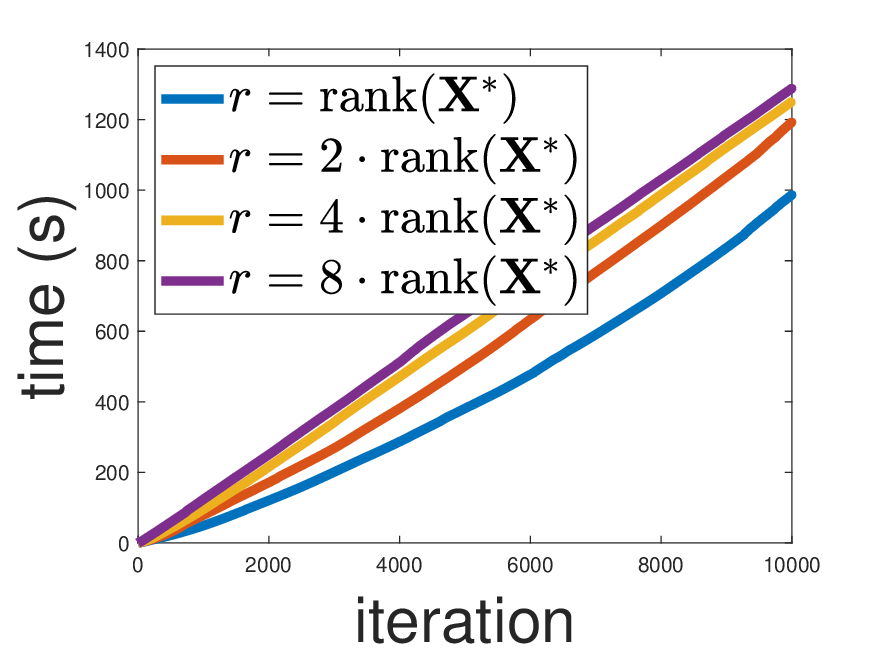}}
\subfigure[G14]{\includegraphics[width=0.24\textwidth]{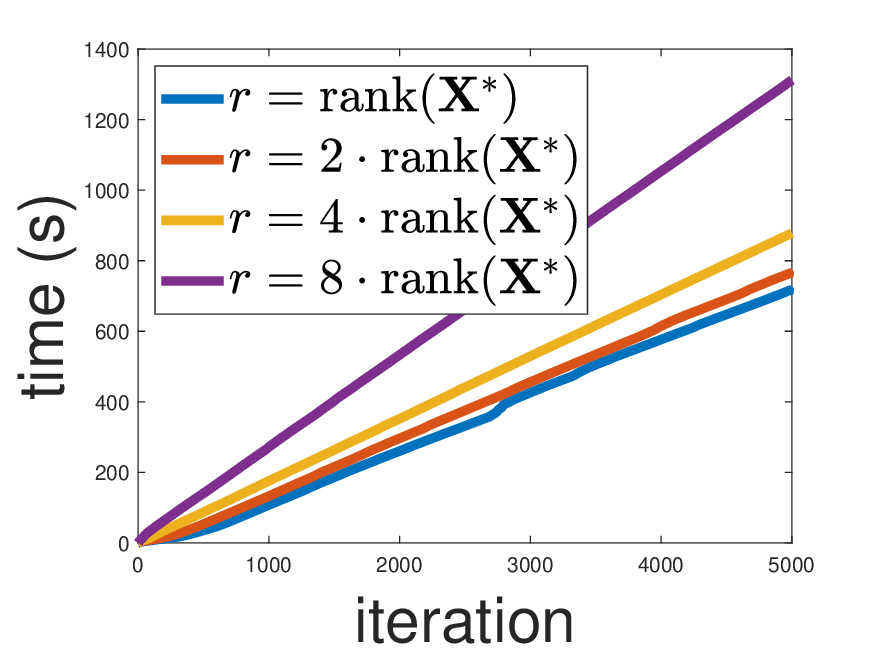}}
\subfigure[G15]{\includegraphics[width=0.24\textwidth]{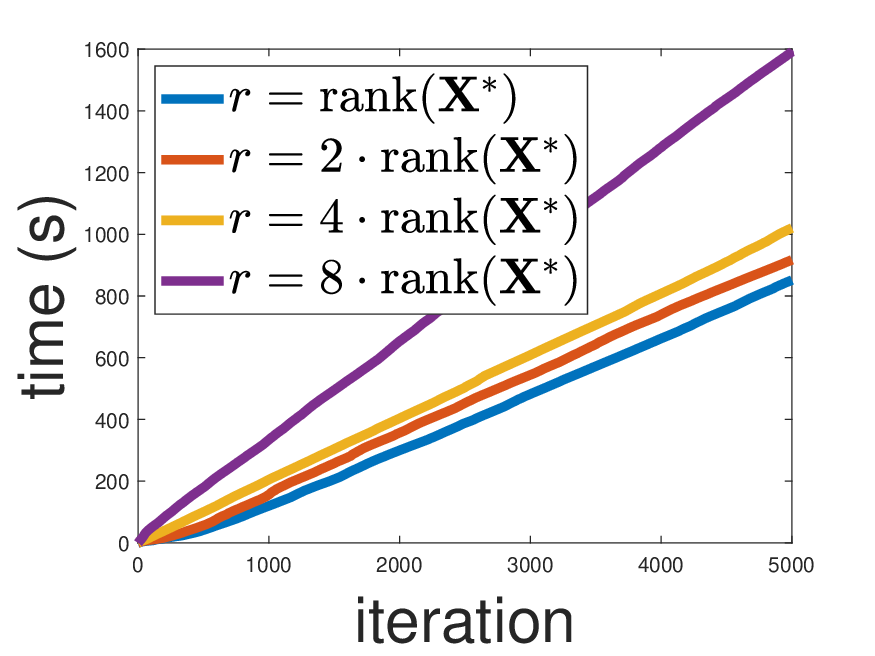}}
\subfigure[G16]{\includegraphics[width=0.24\textwidth]{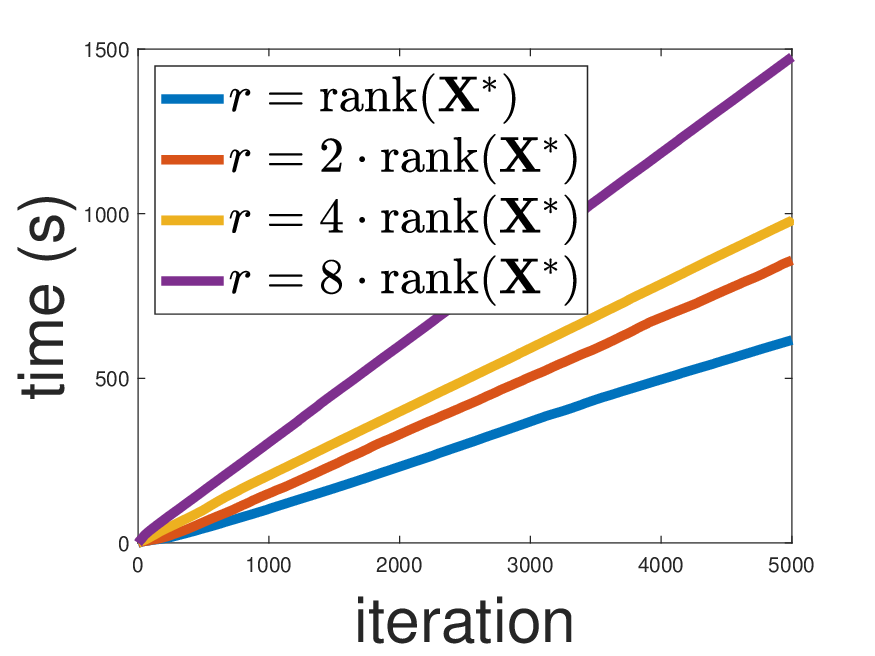}}
\medskip
\subfigure[G17]{\includegraphics[width=0.24\textwidth]{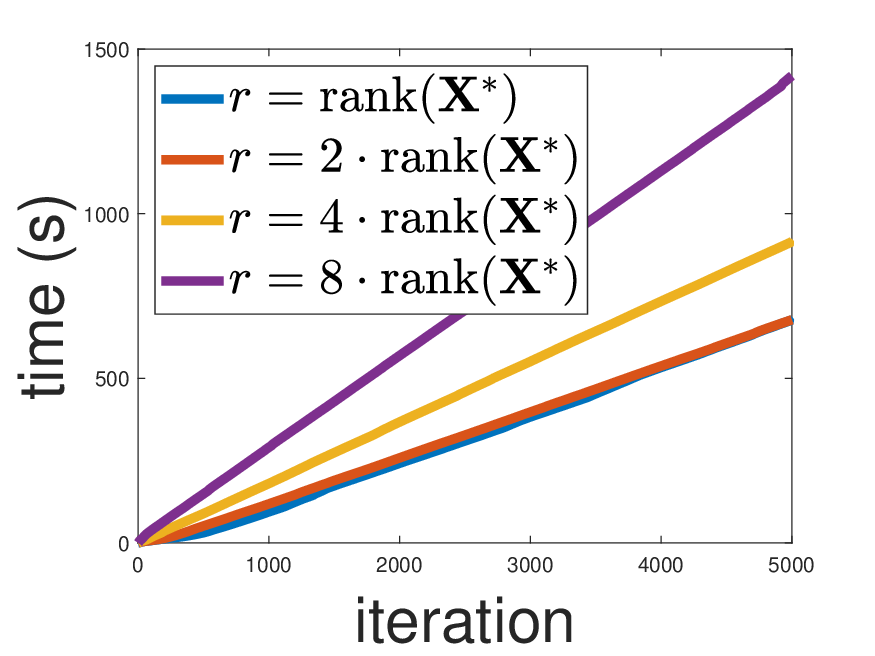}}
\subfigure[G18]{\includegraphics[width=0.24\textwidth]{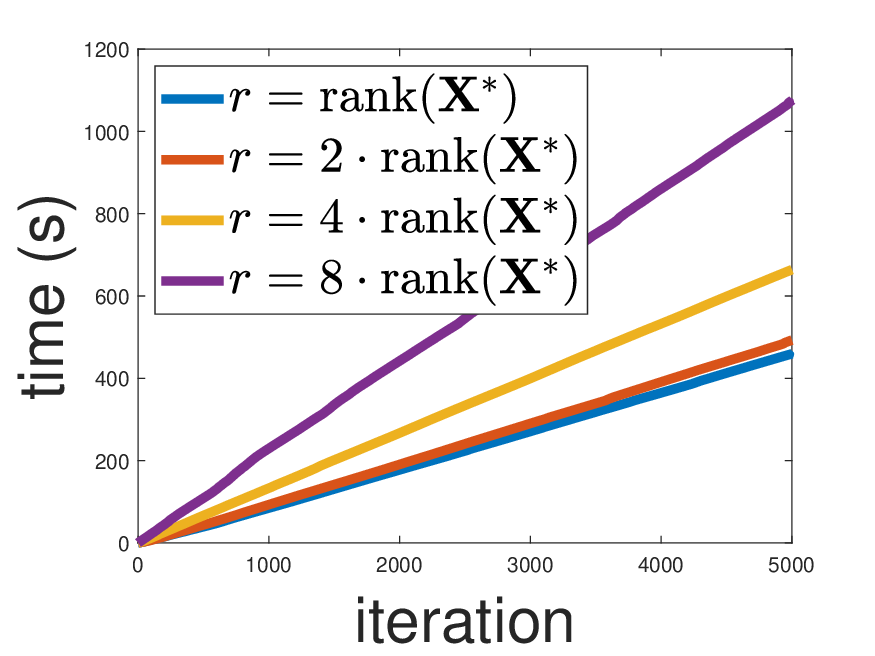}}
\subfigure[G19]{\includegraphics[width=0.24\textwidth]{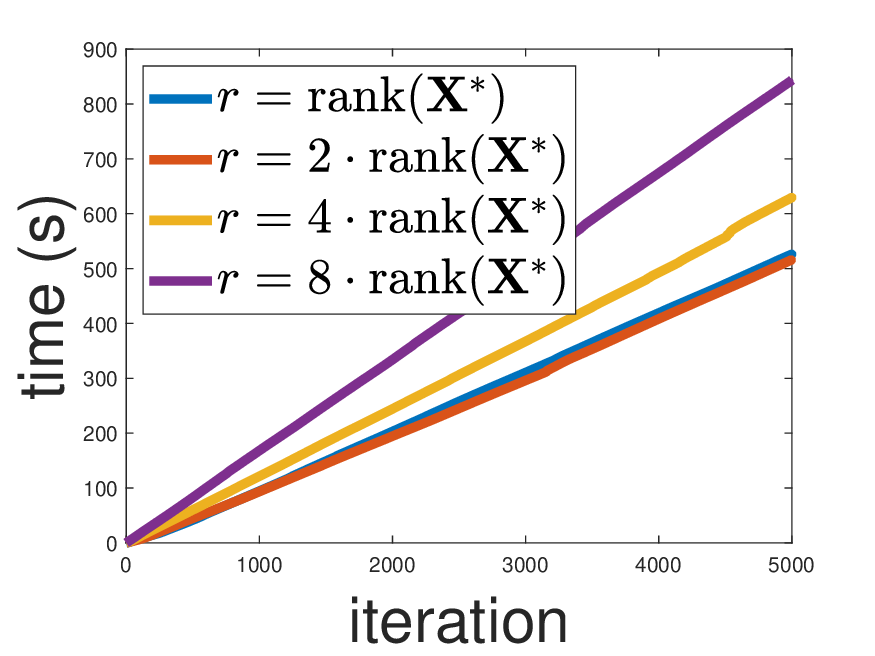}}
\subfigure[G20]{\includegraphics[width=0.24\textwidth]{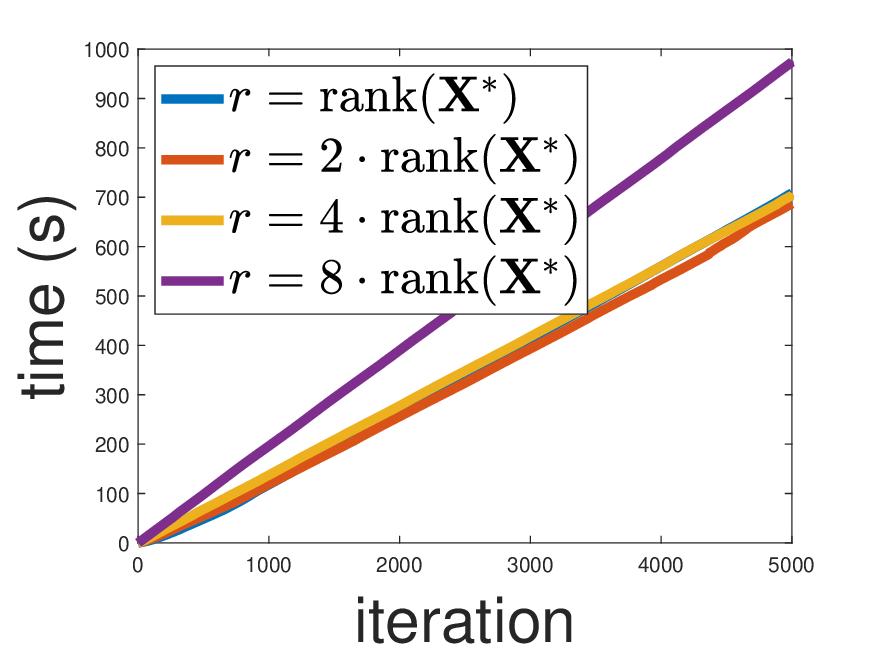}}
\caption{Time as a function of number of iterations it takes to run Algorithm \ref{alg:EG} with varying SVD ranks for the Max-Cut problem.}
\label{fig:time_tableTC_full}
\end{center}
\vskip -0.2in
\end{figure*}

\end{document}